\documentclass[onecolumn]{sn-jnl}
\usepackage[numbers]{natbib}

\theoremstyle{thmstyleone} % bold header, italic body

\ifx\theorem\undefined
  \newtheorem{theorem}{Theorem}
\fi
\newtheorem{lemma}[theorem]{Lemma}
\newtheorem{proposition}[theorem]{Proposition}

\theoremstyle{thmstyletwo} % roman header, italic body

\theoremstyle{thmstylethree} % bold header, roman body
\newtheorem{remark}{Remark}
\newtheorem{corollary}[theorem]{Corollary}

\makeatletter
\@ifundefined{MaxMatrixCols}{
  \newcounter{MaxMatrixCols}
  \setcounter{MaxMatrixCols}{20}
}{\setcounter{MaxMatrixCols}{20}}
\makeatother

\usepackage{subcaption}
\usepackage{graphicx}
\usepackage{multirow}
\usepackage{amsmath,amssymb,amsfonts}
\usepackage{amsthm}
\usepackage{mathrsfs}
\usepackage[title]{appendix}
\usepackage{xcolor}
\usepackage{textcomp}
\usepackage{manyfoot}
\usepackage{booktabs}
\usepackage{algorithm}
\usepackage{algorithmicx}
\usepackage{algpseudocode}
\usepackage{listings}
\usepackage{mathtools}
\usepackage{cleveref}

\usepackage{accents}

\newcommand{\boldheading}[1]{\par\textbf{{#1}.}}
\newcommand{\Bezier}{B\'{e}zier}
\newcommand{\mathup}[1]{\text{\textup{#1}}}

\newcommand{\mc}[1]{\mathcal{#1}}
\newcommand{\mb}[1]{\mathbb{#1}}
\newcommand{\ms}[1]{\mathscr{#1}}

\newcommand{\R}{\mathbb{R}}
\newcommand{\N}{\mathbb{N}}

\newcommand{\mE}{\mb{E}}
\newcommand{\mV}{\mb{V}}

\newcommand{\Lo}{L^{2}(\Omega)}
\newcommand{\Lot}{L^{2}(\Omega \times T)}

\newcommand{\RNM}{\R^{\nx}_{\mat M}}
\newcommand{\RNMT}{\R^{\nx \times \nt}_{\mat M}}

\newcommand{\M}{\ms{M}}
\newcommand{\U}{\ms{U}}

\newcommand{\Lsym}{\mc{L}_{\textnormal{sym}}}
\newcommand{\Ltr}{\mc{L}_{\textnormal{tr}}}

\newcommand{\Ty}{T_{y}}

\newcommand{\lvertiii}{\left\vert\mkern-1mu\left\vert\mkern-1mu\left\vert}
\newcommand{\rvertiii}{\right\vert\mkern-1mu\right\vert\mkern-1mu\right\vert}
\newcommand{\opnorm}[1]{\mathopen{}\lvertiii #1 \rvertiii\mathclose{}}

\newcommand{\bdot}[1]{\accentset{\scalebox{0.35}{\(\bullet\)}}{#1}}

\newcommand{\ny}{n_{y}}
\newcommand{\nt}{n_{t}}
\newcommand{\nx}{n_{x}}

\renewcommand{\vec}[1]{\boldsymbol{#1}}
\newcommand{\mat}[1]{\mathbf{#1}}

\newcommand{\pr}{\textnormal{pr}}
\newcommand{\po}{\textnormal{post}}
\newcommand{\mis}{\textnormal{mis}}
\newcommand{\MAP}{\textnormal{MAP}}

\newcommand{\Cpo}{\mc{C}_{\po}}
\newcommand{\Cpr}{\mc{C}_{\pr}}

\newcommand{\Gpo}{\mat{\Gamma}_{\po}}
\newcommand{\Gpr}{\mat{\Gamma}_{\pr}}
\newcommand{\Gpor}{\mat{\Gamma}_{\po, r}}

\newcommand{\tr}{\mathrm{tr}}

\newcommand{\defeq}{\vcentcolon=}

\newcommand{\iHmis}{\mc{H}_{\mis}}

\newcommand{\fHmis}{\mat H_{\mis}}

\newcommand{\mpr}{m_{\pr}}

\newcommand{\mtrue}{m_{\textnormal{true}}}

\newcommand{\mupoy}{\mu_{\po}^{\vec y, \vec \xi}}

\newcommand{\fmMAPy}{\vec m_{\textnormal{MAP}}^{\vec y}}

\newcommand{\grad}{\nabla}

\newcommand{\sig}{\sigma}

\newcommand{\alp}{\alpha}
\newcommand{\bet}{\beta}

\newcommand{\lam}{\lambda}
\newcommand{\gam}{\gamma}

\newcommand{\vep}{\varepsilon}

\newcommand{\Om}{\Omega}
\newcommand{\om}{\omega}
\newcommand{\eps}{\varepsilon}

\newcommand{\dt}{\Delta t}

\newcommand{\vxi}{\vec \xi}
\newcommand{\vr}{\vec r}
\newcommand{\vx}{\vec x}
\newcommand{\xbar}{\bar{\vx}}
\newcommand{\vy}{\vec y}
\newcommand{\dr}{\delta_{\vec r}}

\newcommand{\llangle}{\langle\!\langle}
\newcommand{\rrangle}{\rangle\!\rangle}
\newcommand{\bllangle}{\big\langle\!\big\langle}
\newcommand{\brrangle}{\big\rangle\!\big\rangle}

\newcommand{\pj}{\partial_j}

\newcommand{\rf}{\vr_{\textnormal{f}}}
\newcommand{\rb}{\vr_{\textnormal{b}}}
\newcommand{\xif}{\vxi_{\textnormal{f}}}
\newcommand{\xib}{\vxi_{\textnormal{b}}}
\newcommand{\Nf}{N_{\textnormal{f}}}
\newcommand{\Nb}{N_{\textnormal{b}}}
\newcommand{\Rd}{R_{\Omega}}
\newcommand{\Tf}{\mat T_{\textnormal{f}}}

\newcommand{\void}{p}
\newcommand{\Ovoid}{D}
\newcommand{\Rvoid}{R_{D}}
\newcommand{\xvoid}{\vec x_{D}}
\newcommand{\Void}{\mat P}
\newcommand{\Hvoid}{\mc{H}_{D}}

\newcommand{\cv}{\mathrm{CV}}

\newcommand{\PsiA}{\Psi_{\text{A}}}

\allowdisplaybreaks

\begin{document}

\title[Path-OED for linear inverse problems governed by PDEs]{Path-OED for infinite-dimensional Bayesian linear inverse problems governed by PDEs}

\author*[1]{\fnm{J.~Nicholas} \sur{Neuberger}}\email{jnneuber@ncsu.edu}
\author[1]{\fnm{Alen} \sur{Alexanderian}}
\author[2]{\fnm{Bart} \spfx{van} \sur{Bloemen Waanders}}
\author[3]{\fnm{Ahmed} \sur{Attia}}

\affil*[1]{\orgdiv{Department of Mathematics}, \orgname{North Carolina State University}, \city{Raleigh}, \state{NC}, \country{USA}}

\affil[2]{\orgdiv{Center for Computing Research}, \orgname{Sandia National Laboratories}, \city{Albuquerque}, \state{NM}, \country{USA}}

\affil[3]{\orgdiv{Mathematics and Computer Science}, \orgname{Argonne National Laboratory}, \city{Lemont}, \state{IL}, \country{USA}}

\abstract{
We consider infinite-dimensional Bayesian linear inverse problems governed by 
time-dependent partial differential equations (PDEs) and develop a mathematical and computational framework for
optimal design of mobile sensor paths in this setting.  The proposed path
optimal experimental design (path-OED) framework is established rigorously in a 
function space setting and elaborated for the case of  Bayesian c-optimality, 
which quantifies the posterior variance in a linear functional of the inverse
parameter. The latter is motivated by goal-oriented formulations, where we seek to minimize the uncertainty 
in a scalar prediction of interest. To facilitate computations, we complement the proposed
infinite-dimensional framework with discretized formulations, in suitably
weighted finite-dimensional inner product spaces, and derive efficient methods
for finding optimal sensor paths.  The resulting computational framework is
flexible, scalable, and can be adapted to a broad range of linear inverse
problems and design criteria.  We also present extensive computational
experiments, for a model inverse problem constrained by an advection-diffusion
equation, to demonstrate the effectiveness of the proposed approach.}

\maketitle

\section{Introduction}
\label{sec:intro}

Physical models governed by partial differential equations (PDEs) often contain
unknown parameters that must be inferred. This requires solving an inverse problem.  We
consider linear Bayesian inverse problems governed by PDEs with
infinite-dimensional parameters.  By utilizing measurement data and a prior
distribution, we apply Bayes' theorem to obtain a posterior distribution for the
unknown parameter.  The informativeness of the data greatly influences the
quality of the solution to the inverse problem. Consequently, optimizing the
data acquisition process is essential and can be framed as an optimal
experimental design (OED) problem.  We consider inverse problems
in which data are collected using mobile sensors such as drones, human-operated
devices, or robots with programmable motion. 

In the present work, solving an OED problem is equivalent to minimizing an optimality criterion $\Psi$---a functional that
quantifies uncertainty in the unknown parameter or a function of the
parameter. The criterion is a function of an experimental \emph{design vector} $\vxi$. 
This vector parameterizes how data are acquired and, in our case, 
determines the path of a sensor.  Given an admissible design set $\Xi$, the
OED problem is
\begin{equation}
\label{eq:OED_intro}
\underset{\vxi \in \Xi}{\text{argmin}} \ \Psi(\vxi).
\end{equation} 
We consider OED problems where data acquisition is constrained to a set of
permissible sensor paths. 
Henceforth, we use the term \emph{path-OED} to refer to such problems.

\boldheading{Related work}
There is a rich body of literature devoted to various aspects of OED.  Textbook
references on this topic include~\cite{AtkinsonDonev92,Pukelsheim93,Ucinski05}.
These books consider OED within a frequentist setting. For an overview of
Bayesian OED, we refer to~\cite{Chaloner95,HuanJagalurMarzouk24}. For background and references on 
optimal design of infinite-dimensional Bayesian inverse problems governed by 
PDEs, see, e.g.,~\cite{Alexanderian21}. The majority of recent efforts on OED
for PDE-based Bayesian inverse problems concern optimal placement of
static sensors.

The field of path-planning for mobile sensing platforms has become quite
extensive in light of modern computational and robotic advancements.  The
books~\cite{choset2005principles,lavalle2006planning,latombe2012robot,roka2018advanced} provide a
comprehensive coverage of path-planning developments. See
also~\cite{gasparetto2015path,karthik2021}.  Methods for planning paths range
from graph-based algorithms~\cite{hart1968formal}, trajectory parameterization~\cite{klanvcar2019drivable, su2018time},
machine and reinforcement learning~\cite{singh2023review, tullu2021machine}, and rapidly-exploring random
trees~\cite{lavalle2001rapidly}.  Often, these approaches are applied to
manufacturing and autonomous vehicles. There have also been works on path
planning for environmental monitoring and control; see, e.g.,~\cite{dunbabin2012robots}.

Considerably less work has been developed in the area of path-planning of mobile
sensors for the purpose of parameter identification and prediction in physical
systems governed by PDEs. An early attempt at this problem is provided
in~\cite{raf1986}, where optimal sensor trajectories are obtained by
minimizing the log determinant of the Fischer information matrix. Subsequent
works built on this idea. In particular, the books~\cite{tricaud2011optimal,
Ucinski05} develop frameworks for obtaining optimal sensor trajectories in
parameter estimation problems. In both works, sensor trajectories are
characterized by ordinary differential equations and the optimal paths reduce a
frequentist description of uncertainty in a small number of scalar-valued parameters.
The paper~\cite{Ucinski052008} provides a concise
overview of this approach. Other methods use ensembles of machine learning
models or Kalman filters. 
In particular, optimal sensor trajectories are
obtained in the work~\cite{Mei24} using a greedy algorithm to enhance Kalman
filter performance.

\boldheading{Our approach and contributions}
We model sensor trajectories with parametric equations of the form $\vr(t;
\vxi)$. Here, $t \in \R$ is the time variable and the parameter $\vxi \in
\R^{N}$ is the design parameter.  The proposed path-OED framework is agnostic to
the specific parameterization of $\vr$ and enables us to encode various constraints related to the sensor apparatus.
We first formulate the path-OED problem in function space and specialize it using the c-optimality criterion, chosen for its prevalence and flexibility. While we present results for c-optimality, the framework extends to other infinite-dimensional criteria, including A- and D-optimality. The function-space setting provides a rigorous probabilistic foundation.  This also guides proper choices of the prior measure and a
consistent discretization strategy. In particular, upon successive mesh
refinements, the discretized design criteria approximate well-defined design
criteria in the function space setting.  A key difference between 
the present path-OED problem formulation with existing works on optimal 
design of PDE-based inverse problems is the complex nature 
of the observation operator. Namely, in the present setting, the observation 
operator is time-dependent and is a nonlinear function of the design vector $\vxi$.
We translate the function space
formulation to a discrete setting by considering suitable finite-dimensional
Hilbert spaces.  Subsequently, we derive novel and efficient path-tailored
techniques that utilize low-rank approximations, fast solvers, and adjoint-based
methods for derivative computations.
The resulting computational framework is flexible, scalable,
and can be adapted to a broad range of linear inverse problems.   

We elaborate the proposed approach in the case of the Bayesian c-optimality
criterion, which quantifies the posterior variance in a linear functional of the
inversion parameter. Moreover, we investigate path constraints for two specific
path parameterizations; one based on \Bezier{} curves and the other motivated by
Fourier expansions. Lastly, in our numerical results, we
compute optimal sensor paths for a model inverse problem governed by a
time-dependent advection diffusion equation. 

The key contributions of this work are 
\begin{itemize}
\item a rigorous and extensible mathematical framework for path-OED in
infinite-dimensional linear Bayesian inverse problems (see
Section~\ref{sec:OED});

\item specialized techniques 
for two classes of parameterized paths---\Bezier{} and Fourier---and 
for incorporating path constraints (see Section~\ref{sec:parameterized_paths});
\item 
a scalable computational framework for path-OED (see Section~\ref{sec:methods}); and 
\item 
comprehensive computational results
demonstrating the effectiveness of our approach (see Section~\ref{sec:numerics}). 
\end{itemize}

Before detailing our proposed approach, we briefly discuss some preliminaries 
regarding the class of infinite-dimensional Bayesian inverse problems under study in
Section~\ref{sec:inverse_problem}.

\section{The inverse problem}
\label{sec:inverse_problem}
In this section we consider the problem of estimating an unknown parameter in a
model governed by a linear time-dependent PDE.  Herein, the governing dynamics
are referred to as the \emph{inversion model} and the unknown parameter is
called the \emph{inversion parameter}. We seek to infer the inversion parameter
from discrete measurements of the inversion model state variable 

We begin by introducing some notation involving linear operators on Hilbert
spaces. Subsequently, we formulate the \emph{data model} that simulates sensor
measurements. This is followed by a brief description of the Bayesian approach to
solving infinite-dimensional linear Bayesian inverse problems.

\textbf{Notation.}\quad
Let $\ms{X}$ and $\ms{Y}$ be infinite-dimensional real separable
Hilbert spaces.  We use $\mc{L}(\ms{X}, \ms{Y})$ to denote the Banach space of
bounded linear operators from $\ms{X}$ to $\ms{Y}$ and $\mc{L}(\ms{X})$ as the
space of bounded linear operators from $\ms{X}$ to $\ms{X}$.  We let
$\Lsym(\ms{X})$ denote the space of self-adjoint bounded linear operators on
$\ms{X}$, $\Lsym^{+}(\ms{X})$ the set of positive self-adjoint bounded linear
operators on $\ms{X}$, $\Lsym^{++}(\ms{X})$ the set of \emph{strictly} positive self-adjoint bounded linear
operators on $\ms{X}$, and $\Ltr(\ms{X})$ as the set of trace-class operators on
$\ms{X}$.  Lastly, we denote the operator norm of $\mc{A} \in \mc{L}(\ms{X},
\ms{Y})$ by $\opnorm{\mc{A}}$.  For details on concepts from operator theory, 
see~\cite[Chapters VI--VIII]{reed1980methods}.

\textbf{The data model.}
The inversion model has a state variable $u \in \U$ and an inversion parameter
$m \in \M$, where $\M$ and $\U$ are infinite-dimensional real separable Hilbert
spaces.  The \emph{solution operator} is defined to be an element of $\mc{L}(\M,
\U)$ that maps the inversion parameter to the state that the sensor observes.

We assume that the \emph{experimental design} (the measurement
configuration) is parameterized by a design vector $ \vxi \in \R^{N}$. In traditional sensor 
placement problems,
this specifies the locations of a static sensor array; in our
setting, $\vxi$ instead determines the path followed by a mobile sensor (see
Subsection~\ref{sec:parameter_to_obs}). For $\vxi \in \R^{N}$, the design
induces an \emph{observation operator}
\begin{equation}
\label{eq:observation}
\mc{B}(\vxi) \in \mc{L}(\U, \R^{\ny}).
\end{equation}
Applying $\mc{B}(\vxi)$ to $u$ extracts $\ny$ discrete measurements of $u$ at locations specified by the design.
We compose the solution and observation operators to obtain the \emph{forward operator}
\begin{equation}
\mc{F}(\vxi) \defeq \mc{B}(\vxi) \circ \mc{S}, \quad \vxi \in \R^{N}.
\end{equation}
Note that $\mc{F} \in \mc{L}(\M, \R^{\ny})$ is of finite rank. 
Assuming an additive Gaussian error model and uncorrelated measurements, the
data model is
\begin{equation}
\label{eq:data_model}
\vec y = \mc{F}(\vxi) m + \vec\eta, \quad \vec\eta \sim \mc{N}(\vec 0, \sigma^{2}\mat I).
\end{equation}
 
We next specify the spaces $\M$ and $\U$ used in our formulations.
We denote the spatial domain of 
the inversion model by $\Om$, which we assume to be an open, bounded, and convex 
subset of $\R^2$. Furthermore, the 
temporal domain is a bounded time interval denoted by $T$. Subsequently, we assume that $m: \Om \to \R$ and $u: \Om \times T \to \R$. However, our formulations can be adapted to different setups: for example, we may consider a three-dimensional spatial domain or a time-dependent inversion parameter. 

In the present work, we let
$\M = \Lo$ and $\U = \Lot$. Generally, the solutions to the inversion model lie in a 
Sobelev space that is a subspace of $\U$.
Here, $\M$ is equipped with the inner product
\begin{equation}
\label{eq:Lo_inner_prod}
\langle f, g \rangle \defeq \int_\Om f(\vx) g(\vx) \, d\vx, \quad f,g \in \M,
\end{equation}
and induced norm $\|\cdot\|_{\M}\defeq\langle \cdot, \cdot \rangle^{1/2}$.
Similarly, the inner product on $\U$ is
\begin{equation}
\label{eq:Lot_inner_prod}
\llangle u, v \rrangle \defeq \int_{T}\int_{\Om} u(\vx, t) v(\vx, t) \, d\vx \, dt, \quad u,v \in \U,
\end{equation}
with induced norm $\|\cdot\|_{\U}\defeq\llangle \cdot, \cdot \rrangle^{1/2}$. 
In what follows, we will also require the Euclidean inner product of vectors $\vx$ and $\vy$. We denote this 
by $\langle \vx, \vy\rangle_{2} \defeq \vy^{\top} \vx$.

\textbf{Bayesian inversion.}\quad
In the Bayesian paradigm we treat the inversion parameter as an $\ms{M}$-valued
random variable and seek a posterior measure $\mu_{\po}^{\vec y}$ describing the
conditional distribution law of $m$ given data $\vec y$. We assume a Gaussian prior
measure $\mu_{\pr} \defeq \mc{N}(\mpr, \Cpr)$ for $m$.  Following the approach
in~\cite{Bui-ThanhGhattasMartinEtAl13}, we define the prior covariance operator
$\Cpr$ using an elliptic differential operator of the form 
\begin{equation}\label{eq:prior_operator}
\mc{E}
\defeq - a_{1} \Delta + a_{2} \mc{I},
\end{equation}
where $a_{1},a_{2} > 0$. Namely, we choose $\Cpr \defeq \mc{E}^{-2}$ ensuring
that $\Cpr \in \Lsym^{++}(\M) \cap \Ltr(\M)$ in two- and three-dimensional
computational domains.

By Bayes' theorem, the posterior measure $\mupoy$ is absolutely continuous with
respect to $\mu_{\pr}$ and its Radon--Nikodym derivative with respect 
to $\mu_{\pr}$ satisfies
\begin{equation}
\label{eq:Bayes_thm}
\frac{d\mupoy}{d\mu_{\pr}} \propto \exp \left\{ - \frac{1}{2\sig^{2}} \| \mc{F}(\vxi)m - \vec y \|_{2}^{2} \right\}.
\end{equation}
See~\cite{Stuart10} for further details.
Here, we have made the dependence of the posterior measure on 
$\vxi$ explicit.

Under the assumption of a linear model, Gaussian noise, and Gaussian prior, the posterior measure is Gaussian, with
$\mupoy \defeq \mc{N}\big( m_\MAP^{\vec y}(\vxi), \Cpo(\vxi) \big)$,
where both the mean and covariance admit analytic expressions~\cite{Stuart10}. Specifically,
\begin{equation}
\label{eq:post}
\Cpo(\vxi) \defeq \big(\sig^{-2}(\mc{F}^*\mc{F})(\vxi) + \Cpr^{-1}\big)^{-1} \quad \text{and} \quad m_\MAP^{\vec y}(\vxi) \defeq \Cpo \big(\sig^{-2}\mc{F}^*(\vxi)\vec y + \Cpr^{-1}\mpr\big).
\end{equation}
Here, we have also used the shorthand $(\mc{F}^*\mc{F})(\vxi)$ for $[\mc{F}(\vxi)]^*\mc{F}(\vxi)$. Note also that $\Cpo \in \Lsym^{++}(\ms{M}) \cap \Ltr(\ms{M})$. 

In the Gaussian linear setting, the posterior mean coincides with the maximum a
posteriori probability (MAP) point~\cite{DashtiStuart17}. This is the motivation
behind the notation for the posterior mean. The MAP point admits a
variational characterization. For a $\vxi \in \R^{N}$, it can be shown that
$m_\MAP^{\vec y}(\vxi)$ is the unique 
global minimizer of  
\begin{equation}
\label{eq:map_var}
J(m; \vxi) \defeq \frac{1}{2\sig^2} \|\mc{F}(\vxi) m - \vec y \|_2^2 + \frac{1}{2} \big\langle \Cpr^{-1/2}(m - \mpr), \Cpr^{-1/2}(m - \mpr)\big\rangle.
\end{equation}
Note that the domain of $J$ is the Cameron--Martin space $\M_\text{pr} \defeq \mathrm{Range}(\Cpr^{1/2})$.
The Hessian of $J$, denoted by $\mc{H}$, is an unbounded, self-adjoint, 
strictly positive, and densely defined operator. 
Two identities involving $\mc{H}$ that we require are
\begin{equation}
\label{eq:hessian}
\Cpo (\vxi) = \mc{H}^{-1} (\vxi) \quad \text{and} \quad \mc{H} (\vxi) = \iHmis (\vxi) + \Cpr^{-1},
\end{equation}
where $\iHmis (\vxi) = \sig^{-2}(\mc{F}^{*}\mc{F})(\vxi)$ is the Hessian of
the data misfit-term in~\eqref{eq:map_var}. 
We remark that, for each $\vxi$, $\iHmis (\vxi) \in \Lsym^{+}(\M)$ and is of finite
rank. Additionally, $\Cpr^{-1}$ is a densely defined, unbounded, self-adjoint,
and strictly positive operator.

The posterior covariance operator $\Cpo$ and its inverse $\mc H$ are central to
the OED formulation that follows. Specifically, $\Cpo$ encodes posterior uncertainty in the inversion parameter and depends on the
design vector $\vxi$ through the data-misfit Hessian. Optimality criteria are expressed in terms of $\Cpo$, while gradient derivations and computational methods require manipulation of $\mc H$.

\section{The path-OED problem}
\label{sec:OED}

As discussed in Section~\ref{sec:intro}, the OED problem is given by~\eqref{eq:OED_intro}. Herein, we specialize this abstract formulation by taking the design to be the trajectory of a single mobile sensor, parameterized by the \emph{design vector} $\vxi \in \R^{N}$. The set $\Xi$ characterizes the admissible sensor paths and solving~\eqref{eq:OED_intro} yields an optimal path.

In Subsection~\ref{sec:parameter_to_obs} we derive the observation operator
in~\eqref{eq:observation} that collects data along a prescribed path. To
elaborate our framework, we take $\Psi$ to be a c-optimality
criterion~\cite{Chaloner95}, expressed in terms of the posterior variance of a
linear functional of the inversion parameter; see Subsection~\ref{sec:copt}. We
also derive the gradient of $\Psi$ with respect to the design vector $\vxi$ in
Subsection~\ref{sec:copt_deriv}, enabling gradient-based optimization. 
To facilitate this, there, we prove a more general result regarding  the derivative of
a parameterized covariance-like operator.
Lastly,
Section~\ref{sec:copt_specific} specifies the particular c-optimality criterion
used in our numerical demonstrations.

\subsection{Data acquisition with mobile sensors}
\label{sec:parameter_to_obs}
Here, we derive an observation operator $\mc{B}$ that maps $u \in \U$ to
measurements along a specified path. We construct $\mc{B}$ in  two stages. First,
we impose paths with a structure via parameterization.  Second, we obtain
measurements by locally averaging $u$ in neighborhoods about points on the path.  
This is necessary because the solution to the inversion model lies in a Sobolev
space and is not necessarily defined pointwise. Having discussed this
formulation, we present the properties of $\mc{B}$ in~Theorem~\ref{thm:obs}.

We consider parameterized paths of the form $\vr = \vr(t; \vxi)$, 
\begin{equation}
\label{eq:r}
\vr : \Ty \times \R^{N} \to \Om,
\end{equation}
with $\Ty \subset T$.
We call  $\Ty$ the \emph{inversion interval}. Given a \emph{path parameter}
$\vxi \in \R^{N}$, the parametrization $\vr(t; \vxi)$ describes the position of
the sensor over the inversion interval. We assume that the partial derivatives
of $\vr$ with respect to the path parameter exist and are continuous. Additionally, we assume
the sensor performs $\ny$ measurements of $u = \mc{S}m$ over 
$\Ty$. Note that the choices of $\Ty$ and $\ny$ are problem dependent.

Let $\{t_{k}\}_{k=1}^{\ny} \subset \Ty$ be a set of measurement times and $\{\vr_k\}_{k=1}^{\ny} \defeq
\{\vr(t_{k}; \vxi)\}_{k=1}^{\ny}$ be corresponding measurement locations on the path $\vr$. For $k \in \{1, 2, \dots, \ny\}$, we consider the mollifier functions
\begin{equation}
\label{eq:mollifiers}
\begin{alignedat}{2}
\dr(\vec x, t)(\vxi) &\defeq (2\pi\eps_{\vx})^{-1}\exp\left\{-\frac{\|\vec x - \vr(t; \vxi)\|^{2}_{2}}{2\eps_{\vx}}\right\}; \quad\text{and}\\
\tau_{k}(t) &\defeq \big(\sqrt{2\pi\eps_{t}}\big)^{-1}\exp\left\{-\frac{(t-t_{k})^{2}}{2\eps_{t}}\right\}.
\end{alignedat}
\end{equation}
The constants $\vep_{\vx},\vep_{t} > 0$ control the scaling of the mollifiers. As these numbers tend towards zero, the mollifiers weakly converge to the delta Dirac distribution.
For $\vxi \in \R^{N}$, the action of the observation operator $\mc{B}(\vxi)$ on $u \in \U$ produces a vector in $\R^{\ny}$; we define its $k$th component by
\begin{equation}
\label{equ:observation_operator_def}
\left(\mc{B}(\vxi)u\right)_{k} \defeq \llangle \tau_{k} \dr(\vxi), u \rrangle, \quad k \in \{1, \ldots, \ny\}.
\end{equation}
The following result establishes the key properties of this observation 
operator. 
\begin{theorem}
\label{thm:obs}
The observation operator $\mc{B}$ defined in~\eqref{equ:observation_operator_def} 
satisfies the following properties:
\begin{enumerate}
\item if $u \in C(\Om \times T)$, then
\begin{equation}
\underset{(\eps_{\vx}, \eps_t) \to (0, 0)}{\text{lim}} \left(\mc{B}(\vxi)u\right)_k = u(\vr_k(\vxi), t_k);
\end{equation}
\item $\mc{B}(\vxi) \in \mc{L}(\U; \R^{\ny})$, for all $\vxi \in \R^{N}$;
\item $\mc{B}(\cdot): \R^{N} \to \mc{L}(\U; \R^{\ny})$ is continuous in operator norm; and
\item the adjoint of $\mc{B}$, denoted by $\mc{B}^*$, satisfies 
\begin{equation}
\mc{B}^{*}(\vxi)\vec y\defeq \dr(\vxi) \sum_{k=1}^{\ny}\tau_{k}y_{k}, \quad \vec y \in \R^{\ny}.
\end{equation}
\end{enumerate}
\end{theorem}

\begin{proof}
See Appendix~\ref{appdx:path}.
\end{proof}

\begin{remark}
\label{rmk:moll}
Henceforth, the terms \emph{design vector} and \emph{path parameter} are
interchangeable. Also, in the present setting, 
the observation operator depends nonlinearly on the design vector $\vxi$. 
Subsequently, the resulting design criteria are nonlinear and nonconvex 
functions of $\vxi$.
\end{remark}

\subsection{The design criterion}
\label{sec:copt}

We define the design criterion as the posterior variance of a 
bounded linear functional $f:\ms{M} \to \R$. 
Namely,
\begin{equation}
\label{eq:c_opt_def}
\Psi \defeq \mV_{\mupoy}\{f\}.
\end{equation}
This is motivated by a \emph{goal-oriented}
formulation~\cite{attia2018goal,ButlerJakemanWildey20,neuberger2025goal},
where we seek designs that minimize the uncertainty in a prediction quantity
that is defined as a functional of the inversion parameter.  This is in contrast
to \emph{classical} design criteria, which quantify uncertainty in the inversion
parameter itself. 

By the Riesz representation theorem, there exists a unique $c \in \M$ such that
$f(m) = \langle c, m \rangle$, for all $m \in \ms{M}$.
Recall also that $\mupoy$ is
a Gaussian measure with mean and covariance specified in~\eqref{eq:post}. 
Thus,
$\mV_{\mupoy}\{f\} =\langle \Cpo c, c\rangle$,
which follows directly from the definition of the posterior covariance operator; 
see, e.g.,~\cite{Alexanderian21}.
Since $\Cpo$ is a function of 
the design vector, we have 
\begin{equation}
\label{eq:c_optimal}
\Psi(\vxi) = \langle \Cpo(\vxi) c, c \rangle.
\end{equation}
This is a Bayesian \emph{c-optimality criterion}---a terminology rooted in traditional
statistics literature.  Specifically, in that context, one considers the
posterior variance of a linear combination of finitely many scalar-valued
parameters. Hence,~\eqref{eq:c_opt_def} defines an infinite-dimensional
c-optimality criterion.  In Subsection~\ref{sec:copt_specific} we provide a
specific example of a c-optimality criterion used in our numerical demonstrations.

\subsection{The derivatives of $\Psi$}\label{sec:copt_deriv}
The main results of this subsection are Theorem~\ref{thm:Cpo_deriv} and 
Theorem~\ref{thm:Cpo_partial}. These results provide formulas for the derivatives of
a parameterized covariance-like operator and the path-dependent posterior
covariance operator.  In particular, Theorem~\ref{thm:Cpo_partial} is essential
for computing the gradient of $\Psi$ (see Corollary~\ref{lem:c_opt_derivs}) as
well as other choices of path-dependent design criteria.

For $\xi \in \R$, let $\mc{H}_{0}(\xi) \in \Lsym^{+}(\M)$ and $\mc{C}_{0}(\xi)
\in \Lsym^{++}(\M)$, where $\mc{C}_0$ is not assumed to be surjective.
Hence, $\mc{C}_0^{-1}$ an operator from the range of
$\mc{C}_0$ to $\M$ and possibly unbounded. 
In the case where $\mc{C}_0$ is a prior covariance operator, 
which is trace-class, 
the range of $\mc{C}_0$ is a dense subspace of $\M$ and $\mc{C}_0^{-1}$ is unbounded.
We then consider the operator
\begin{equation}\label{eq:generic_H}
\mc{H}(\xi) \defeq \mc{H}_0(\xi) + \mc{C}_0^{-1}.
\end{equation}
Note that $\mc{H}$ is a self-adjoint, strictly positive, and a possibly unbounded
operator defined on the range of $\mc{C}_{0}$. It can be shown that the inverse
$\mc{C} \defeq \mc{H}^{-1}$ exists and for every $\xi \in \R$, belongs to
$\Lsym^{++}(\M)$. Here, $\mc{C}$, is an analogue of
the posterior covariance operator in~\eqref{eq:post} with some assumptions relaxed.

We say that $\mc{C}: \R \to \Lsym^{++}(\M)$ is Frech\'{e}t differentiable at $\xi$ if there
exists $\mc{A}: \R \to \mc{L}(\M)$ such that
\begin{equation}
\label{eq:frechet}
\underset{h \to 0}{\text{lim}} \ \opnorm{\frac{\mc{C}(\xi + h) - \mc{C}(\xi)}{h} - \mc{A}(\xi)} = 0.
\end{equation}
In this case, we write 
\begin{equation}
\label{eq:C_dot}
\bdot{\mc{C}}(\xi) \defeq \mc{A} = \underset{h \to 0}{\text{lim}} \frac{1}{h}\left(\mc{C}(\xi + h) - \mc{C}(\xi)\right),
\end{equation}
where the dot above $\mc{C}$ denotes derivative with respect to $\xi$. We use this notion of differentiability in Theorem~\ref{thm:Cpo_deriv}---providing us with a formula for $\bdot{\mc{C}}$.

\begin{theorem}
\label{thm:Cpo_deriv}
Consider the operator $\mc{C} = \mc{H}^{-1}$ with $\mc{H}$ as defined in~\eqref{eq:generic_H}. If $\mc{H}_{0}$ is Frech\'{e}t differentiable in $\R$, then
\begin{equation}
\label{eq:Cpo_deriv}
\bdot{\mc{C}} \equiv -\mc{C} \bdot{\mc{H}}_0 \mc{C}.
\end{equation}
\end{theorem}
\begin{proof}
See Appendix~\ref{appdx:Cpo_deriv}.
\end{proof}
\noindent
The proof of Theorem~\ref{thm:Cpo_deriv} is insightful and requires careful
treatment of linear operators in Hilbert space. 

We also need a notion of differentiability for parameterized elements of the solution space $\U$. Suppose that $u:\R \to \U$. Similar to the operator setting, we say that $u(\xi) \in \U$ is Frech\'{e}t differentiable at $\xi$ if there exits $a \in \U$ such that
\begin{equation}
\label{eq:frechet_hilbert}
\underset{h \to 0}{\text{lim}} \Big\| \frac{u(\xi + h) - u(\xi)}{h} - a \Big\|_{\ms{U}} = 0.
\end{equation}
In this case, we write
\begin{equation}
\label{eq:u_dot}
\bdot{u} = a = \underset{h \to 0}{\text{lim}} \frac{1}{h}\left( u(\xi + h) - u(\xi) \right).
\end{equation}

Recall that the tensor product of $u,v \in \U$, 
denoted by $u \otimes v$, is an element of $\mc{L}(\U)$, defined by 
$(u \otimes v)w = u \llangle v, w \rrangle$, for all $w \in \U$.
To find the derivatives of $\Cpo$, we need to differentiate through such tensor product
operators. This is facilitated by the following lemma:
\begin{lemma}
\label{lem:tens_prod_deriv}
Suppose that $u,v:\R \to \U$ are Frech\'{e}t differentiable in $\R$ and define the operator $\mc{K}:\R \to \mc{L}(\U)$ by $\mc{K}(\xi) = (u \otimes v)(\xi).$ Then, $\bdot{\mc{K}} \equiv u \otimes \bdot{v} + \bdot{u} \otimes v.$

\end{lemma}

\begin{proof}
We note that
\begin{align*}
\underset{h \to 0}{\text{lim}} \frac{1}{h} \left( \mc{K}(\xi + h)  - \mc{K}(\xi)\right) &= \underset{h \to 0}{\text{lim}} \frac{1}{h} \left( u(\xi + h) \otimes v(\xi + h) - u(\xi) \otimes v(\xi) \right)\\
&= \underset{h \to 0}{\text{lim}} \frac{1}{h} u(\xi + h) \otimes (v(\xi + h) - v(\xi)) + \underset{h \to 0}{\text{lim}} \frac{1}{h} (u(\xi + h) - u(\xi)) \otimes v(\xi)\\
&= u \otimes \bdot{v} + \bdot{u} \otimes v.
\end{align*}
Here, we have also used the fact that
$u$ and $v$ are continuous in $\xi$. 
\end{proof}
Now we have the tools needed to derive the derivatives of $\Cpo(\vxi)$ as stated in Theorem~\ref{thm:Cpo_partial}.

\begin{theorem}
\label{thm:Cpo_partial}
Consider the posterior covariance operator $\Cpo$ defined in~\eqref{eq:post}. Then, for $\vxi \in \R^{N}$ and 
$j \in \{1, \ldots, N\}$,

\begin{subequations}
\begin{align}
\pj \Cpo(\vxi) &= - \Cpo(\vxi) \big(\pj\mc{H}(\vxi)\big)\Cpo(\vxi),\label{eq:Cpo_j}\\
\pj\mc{H}(\vxi) &= -\sig^{-2}\mc{S}^{*}\Big(\sum_{k=1}^{\ny} \tau_{k}\dr(\vxi) \otimes \tau_{k}\pj\dr(\vxi) + \tau_{k}\pj\dr(\vxi) \otimes \tau_{k}\dr(\vxi)\Big)\mc{S},\label{eq:H_partial_big}
\end{align}
\end{subequations}
where we define $\pj \defeq \frac{\partial}{\partial \xi_j}$.

\end{theorem}
\begin{proof}
Viewing $\Cpo$ as a function of the $j$th component of $\vxi$, we
apply Theorem~\ref{thm:Cpo_deriv} to obtain~\eqref{eq:Cpo_j}. We are then
left to derive an explicit representation for $\pj\mc{H}$.
The design vector appears in the operator $(\mc{B}^{*}\mc{B})(\vxi) \defeq [\mc{B}(\vxi)]^{*}\mc{B}(\vxi)$ and the solution operator $\mc{S}$ is independent of $\vxi$. Hence, the derivative of the Hessian is
\begin{equation}
\label{eq:Hj_BB}
\pj\mc{H}(\vxi) = \sig^{-2} \mc{S}^* \pj (\mc{B}^*\mc{B})(\vxi) \mc{S}.
\end{equation}
Note that for $u \in \U$, the definition of $\mc{B}^{*}$ in Theorem~\ref{thm:obs} provides that
$$
(\mc{B}^{*}\mc{B})(\vxi)u = \dr (\vxi) \sum_{k=1}^{\ny} \tau_k \llangle \tau_k \dr(\vxi), u \rrangle.
$$
We then obtain an explicit representation of this operator with the tensor product $\otimes$. Specifically,
$$
(\mc{B}^{*}\mc{B})(\vxi) = \sum_{k=1}^{\ny} \tau_{k}\dr(\vxi) \otimes \tau_{k}\dr(\vxi)
$$
Then, by invoking Lemma~\ref{lem:tens_prod_deriv}, we are left with
\begin{equation}
\label{eq:BB_partial}
\pj (\mc{B}^{*}\mc{B})(\vxi) = \sum_{k=1}^{\ny} \tau_{k}\dr(\vxi) \otimes \tau_{k}\pj\dr(\vxi) +  \tau_{k}\pj\dr(\vxi) \otimes \tau_{k}\dr(\vxi).
\end{equation}
Subsequently,~\eqref{eq:Cpo_j} follows from combining~\eqref{eq:BB_partial} with~\eqref{eq:Hj_BB}.
\end{proof}

Using Theorem~\ref{thm:Cpo_partial}, we are now ready 
to derive the derivatives of the general c-optimality criterion:
\begin{corollary}
\label{lem:c_opt_derivs}
For $j \in \{1, \dots, N\}$, the $j$th partial derivative of the c-optimality criterion $\Psi$ in~\eqref{eq:c_optimal} is given by
\begin{equation}
\label{eq:psi_c_j}
\pj\Psi (\vxi) = - 2 \sig^{-2} \langle \mc{B} (\vxi) \mc{S} \Cpo (\vxi) c, \big( \pj \mc{B} (\vxi) \big) \mc{S} \Cpo (\vxi) c \rangle_{2}.
\end{equation}
\end{corollary}

\begin{proof}

For notational convenience, we suppress dependencies on the design vector $\vxi$ in what follows. Using that $\Cpo$ is self-adjoint and the identity~\eqref{eq:Cpo_j} from Theorem~\ref{thm:Cpo_partial}, we obtain
\begin{equation}
\label{eq:Psi_c_j_intermediate}
\pj\Psi = \big\langle \pj \Cpo c, c \big\rangle = - \big\langle \big(\pj\mc{H}\big) \Cpo c , \Cpo c  \big\rangle.
\end{equation}
Then,~\eqref{eq:H_partial_big} and manipulating the tensor products reveals
$$
\big(\pj\mc{H}\big)\Cpo c = -\sig^{-2}\sum_{k=1}^{\ny} \llangle \tau_{k} \pj\dr, \mc{S}\Cpo c\rrangle \mc{S}^{*}\tau_{k}\dr + \llangle \tau_{k}\dr, \mc{S}\Cpo c\rrangle \mc{S}^{*}\tau_{k} \pj\dr.
$$
We then substitute this expression back into~\eqref{eq:Psi_c_j_intermediate} and use properties of the inner product 
to obtain
$$
\pj\Psi = - 2 \sig^{-2} \sum_{k=1}^{\ny} \llangle \tau_k \dr, \mc{S} \Cpo c \rrangle \llangle \tau_k \pj\dr, \mc{S}\Cpo c\rrangle.
$$
Denoting the $k$th component of $\mc{B} \mc{S} \Cpo c$ and $(\pj \mc{B}) \mc{S} \Cpo c$ by $\left[ \mc{B} \mc{S} \Cpo c\right]_{k}$ and $\left[ (\pj \mc{B}) \mc{S} \Cpo c\right]_{k}$, respectively,
\begin{equation}
\label{eq:B_components}
\begin{alignedat}{1}
\llangle \tau_k \dr, \mc{S} \Cpo c \rrangle &= \left[ \mc{B} \mc{S} \Cpo c\right]_{k},\\
\llangle \tau_k \pj\dr, \mc{S}\Cpo c\rrangle &= \left[ (\pj \mc{B}) \mc{S} \Cpo c\right]_{k}.
\end{alignedat}
\end{equation}

Recalling~\eqref{eq:psi_c_j}, the proof is concluded since
$$
\pj\Psi = - 2 \sig^{-2} \sum_{k=1}^{\ny} \left[ \mc{B} \mc{S} \Cpo c\right]_{k} \left[ (\pj \mc{B}) \mc{S} \Cpo c\right]_{k}.
$$

\end{proof}

We remark that the formula for $\pj \Psi$ presented in
Corollary~\ref{lem:c_opt_derivs} holds for every $c \in \ms{M}$. In the
following subsection, we specify the c-optimality criterion and its gradient for
a certain $c$.

\subsection{A specific criterion}\label{sec:copt_specific}

Here we present the c-optimality criterion used in our numerical demonstrations; see Section~\ref{sec:numerics}. In particular, we consider a certain bounded linear functional of $m$, referred to as a \emph{goal-functional}, which represents a prediction or quantity of interest. Having derived the vector $c \in \M$ which characterizes this goal-functional, we specialize the formulas for the criterion and its derivatives.

Let $v \in \U$ and consider the
goal-functional $Z:\M \to \R$ defined by
\begin{equation}
\label{eq:goal}
Z(m) \defeq \llangle v, \mc{S}m \rrangle, \quad m \in \ms{M}.
\end{equation}
If $v$ is an indicator function of a spatio-temporal subdomain
of $\Om \times T$, then $Z$ is a random variable describing
the integration of the solution $u = \mc{S}m$ over this subdomain. 
The resulting optimality criterion is then
\begin{equation}
\label{eq:goal_var}
\Psi(\vxi) \defeq \mV_{\mupoy}\{Z\}.
\end{equation} 
Using the adjoint of $\mc{S}$, an equivalent characterization of the goal-functional is
$Z(m) = \langle \mc{S}^* v, m \rangle$.
Hence, $\Psi$ is the c-optimality criterion in~\eqref{eq:c_optimal} with 
$c = \mc{S}^* v$. Equivalently,
\begin{equation}
\label{eq:psi}
\Psi(\vxi) = \langle \Cpo (\vxi) \mc{S}^* v, \mc{S}^* v\rangle.
\end{equation}
We can also substitute $c = \mc{S}^* v$ in the derivative
formula~\eqref{eq:psi_c_j}, providing
\begin{equation}
\label{eq:psi__j}
\pj\Psi = - 2 \sig^{-2} \langle \mc{B} (\vxi) \mc{S} \Cpo (\vxi) \mc{S}^* v, \big( \pj \mc{B} (\vxi) \big) \mc{S} \Cpo (\vxi) \mc{S}^* v \rangle_{2}.
\end{equation}

\begin{remark}

The above developments extend naturally to other optimality criteria in the infinite-dimensional setting. In particular, the A- and D-optimality criteria emerge in a straightforward manner. We recall that the A-optimality criterion \cite{Alexanderian21} is defined by $\PsiA \defeq \tr(\Cpo)$, which can be approximated using a Monte Carlo trace estimator. The formula for the partial derivative of $\Psi$ in Corollary~\ref{lem:c_opt_derivs} can then be used to compute the gradient of trace-estimated A-optimality criterion. The D-optimality criterion \cite{AlexanderianGhattasEtAl16} is also within reach. Evaluating this criterion requires computing the spectrum of the \emph{prior-preconditioned data misfit Hessian} $\Cpr^{1/2}\iHmis\Cpr^{1/2}$, which we already do in our numerical implementation; see Subsection~\ref{sec:considerations}.

\end{remark}

\section{Parameterized sensor paths}
\label{sec:parameterized_paths}

Our framework is compatible with any parameterized curve $\vr(\cdot \, ; \vxi)
\in \R^{2}$ that is continuously differentiable with respect to the path
parameter $\vxi$, and it readily extends to three-dimensional domains. The
specific form of parameterization depends on problem-specific factors related to
the inversion model and sensor apparatus. For example, an aerial drone has a
limited battery life and capability to accelerate---both influential on the
choice of parameterization or constraints on $\vxi$. We present two flexible
parameterizations: 
\emph{\Bezier} and \emph{Fourier} paths. After defining these
path types in Subsection~\ref{sec:path_types}, we show how constraints and other
practical considerations can be incorporated in the path-OED problem in 
Subsection~\ref{sec:path_constraints}.  

\subsection{Path types}
\label{sec:path_types}

In this subsection, we discuss \Bezier{} and Fourier paths for path-OED.

\boldheading{\Bezier{} curves}
\Bezier{} curves~\cite{farouki12,forrest90} are commonly used in
computer graphics, approximation theory, and optimization. 
Such curves have been utilized in optimal path-planning as well.
In particular, 
\Bezier{} curves were used to solve an autonomous vehicle path-planning
problem in~\cite{raju22}, where the path is constrained to a corridor.
\Bezier{} curves are also utilized in optimal control literature; 
see e.g.,~\cite{kmet19,LeeKim21}. The existing theory and software for \Bezier{} curves make them a natural choice
for parameterizing sensor trajectories.

A degree $\Nb$ \Bezier{} curve is a polynomial of degree $\Nb$ that is uniquely characterized by a 
set of \emph{control points}
$\{\vec p_j\}_{j=0}^{N_b} \subset \R^2$. 
The \Bezier{} curve corresponding to this set of control points is
\begin{equation}
\label{eq:bezier_path}
\rb(t; \xib) \defeq \sum_{j=0}^{\Nb} {\Nb\choose j} t^{j}(1-t)^{\Nb-j}\vec p_{j},
\quad t \in [0, 1],
\end{equation}
where $\xib \defeq [\vec p_0^\top \; \vec p_1^\top \; \cdots \; \vec
p_{\Nb}^\top]^\top \in \R^{2(\Nb + 1)}$.  We use~\eqref{eq:bezier_path} to describe the path of a mobile sensor. While the formula for $\rb$ is defined 
over $t \in [0, 1]$, we can use a linear change of variable to define
$\rb$ over the inversion interval $\Ty$.
The first and last control points specify the endpoints of the \Bezier{} curve. If
these points coincide, then the curve is closed. Intermediate control points do
not necessarily lie on the curve, but have the behavior of attracting the curve
in the direction of the control points. We
depict two \Bezier{} curves in Figure~\ref{fig:bezier_examples}.

\begin{figure}[ht]
\centering
\includegraphics[height=0.3\textwidth]{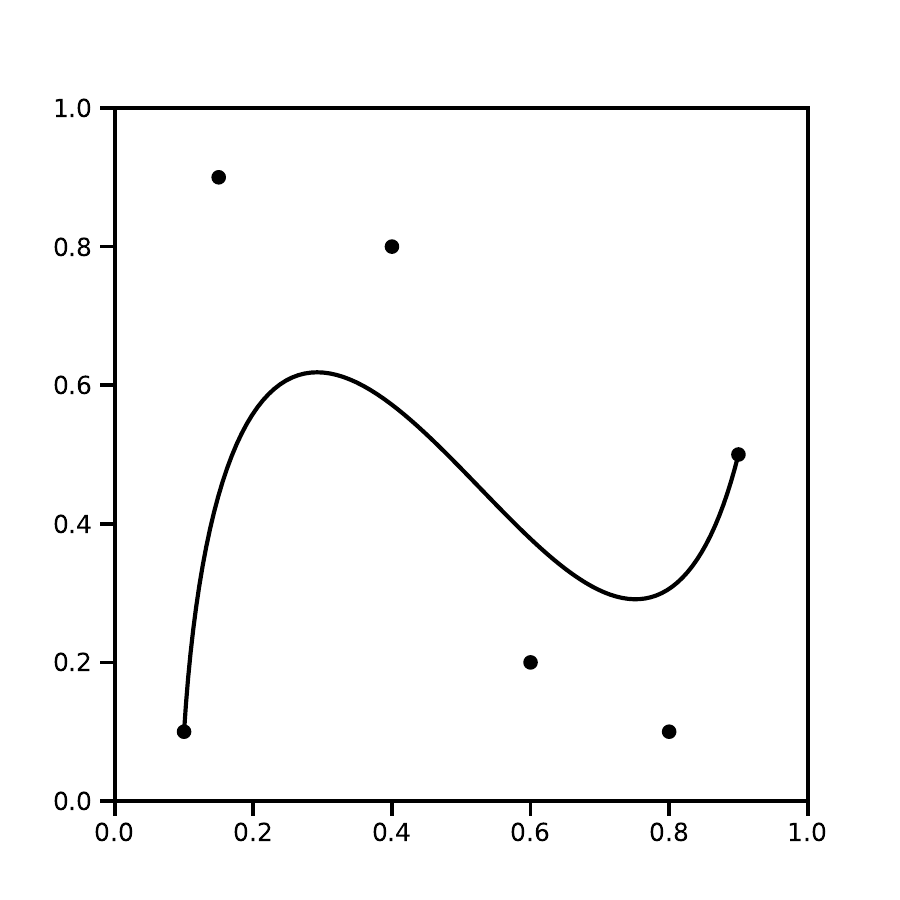}
\includegraphics[height=0.3\textwidth]{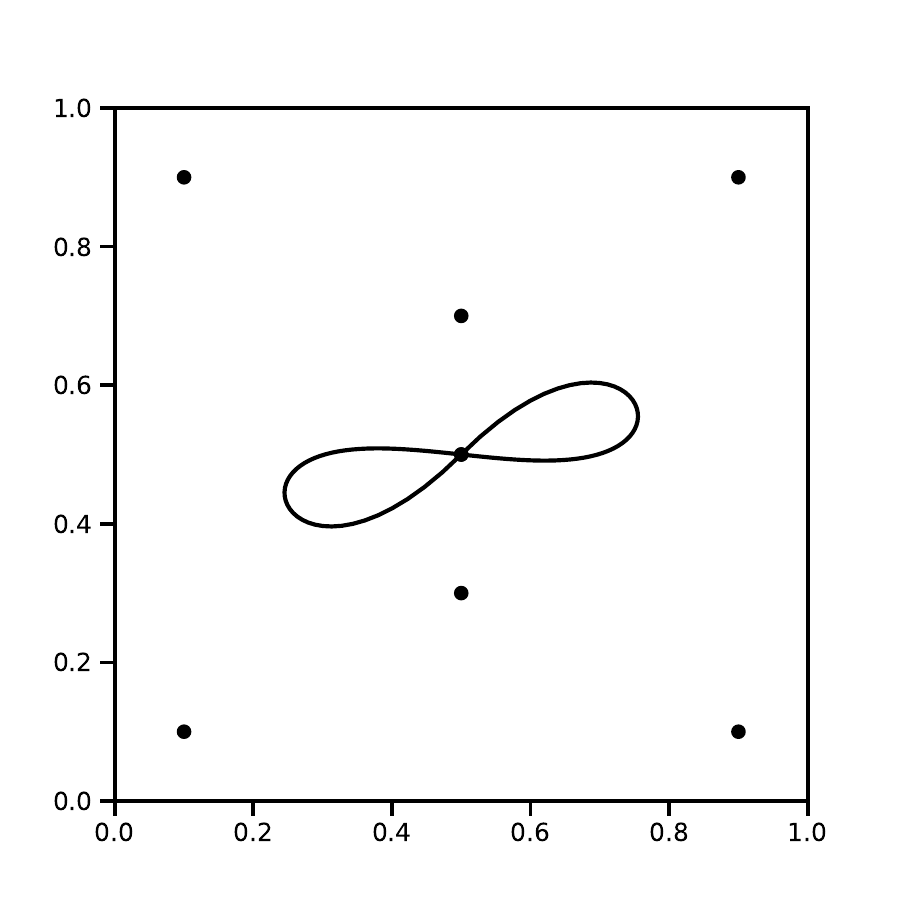}
\includegraphics[height=0.3\textwidth]{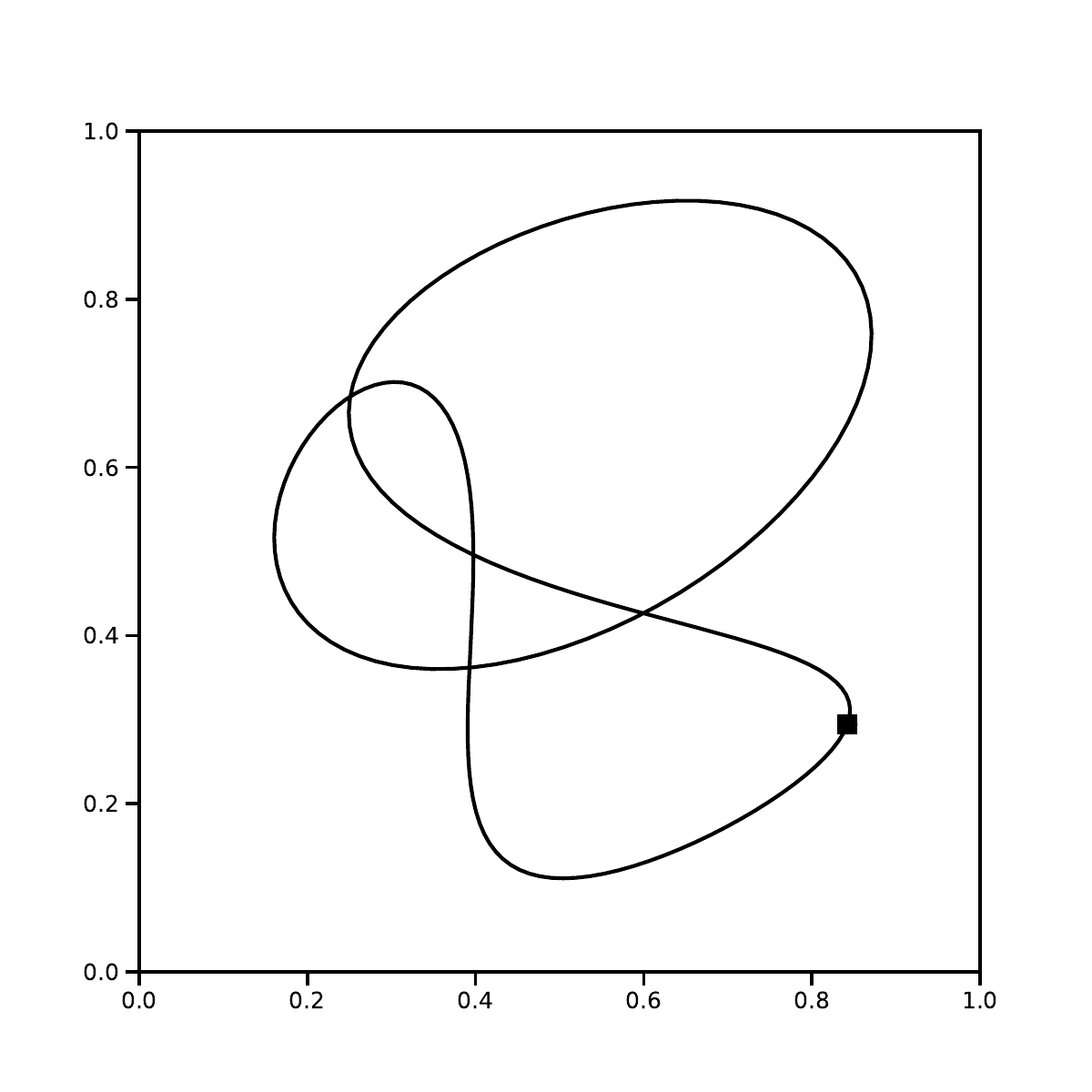}
\caption{Left and Middle: Bezier{} curves of degree $\Nb = 5$ with control points plotted as dots. Right: A Fourier path with $\Nf = 3$ modes.}
\label{fig:bezier_examples}
\end{figure}

\boldheading{Fourier paths}
Let $\bar{\vec x} \in \Om$ and $\Nf \in \N$. Given $t \in \Ty$ and a set of $\Nf$ frequencies $\{\om_{j}\}_{j=1}^{\Nf}$, the Fourier path $\rf$ with $\Nf$ modes is defined by
\begin{equation}
\label{eq:fourier_path}
\vec \rf(t; \xif) \defeq \xbar + \sum_{j=1}^{\Nf} 
\begin{bmatrix}
a_{j} \cos(\om_{j}t) + b_{j} \sin(\om_{j} t)\\
c_{j} \cos(\om_{j} t) + d_{j} \sin(\om_{j} t)
\end{bmatrix}
,
\end{equation}
where $\xif \defeq [a_1 \, b_1 \, c_1 \, d_1 \, \dots \, a_{\Nf} \, b_{\Nf} \, c_{\Nf} \, d_{\Nf}]^\top \in \R^{4\Nf}$.
For a path $\vr(t): \R \to \R^{2}$ whose components lie in the space $L^2(\Ty)$, the Fourier basis
\begin{equation}
\label{eq:fourier_basis}
\{\cos(\om_j t), \sin(\om_j t)\}_{j=0}^\infty \quad \text{with} \quad \om_j \defeq \frac{2\pi j}{|\Ty|},
\end{equation}
can be used to approximate each component of $\vr$. This motivates the
representation in~\eqref{eq:fourier_path}. We illustrate an example of a Fourier
path in Figure~\ref{fig:bezier_examples}.

Note that the summation term
in~\eqref{eq:fourier_path} is a linear transformation of the path
parameter (design vector). Thus, $\rf$ can be viewed as an affine transformation on $\xif$,
parameterized over $\Ty$. The shift term $\bar{\vx}$ accounts for the constant
basis function in~\eqref{eq:fourier_basis}. We use this to translate the path
into $\Om$. Fourier paths are periodic and
can generate a rich family of curves, even with a low number of modes.

\subsection{Imposing physical constraints}
\label{sec:path_constraints}
Here, we discuss the following practical considerations (i) confining
sensor paths to $\Omega$; (ii) regularizing penalty terms; and
(iii) incorporating regions where data are known to be obscured. 

\boldheading{Confining paths}
In many applications we want to keep the sensor in the domain $\Om$. One way of
enforcing this is to place a constraint directly on the path parameter $\vxi$.
In the case of \Bezier{} paths, it is known that the curve $\rb$ lies in the
convex hull of its control points~\cite{raju22}. Thus, in the case of a convex domain 
$\Om$, we can keep the \Bezier{}
path in $\Om$ by restricting the control
points to $\Om$. In the case that $\Om$ is rectangular, enforcing this
constraint results in an optimization problem with a linear inequality
constraint on $\xib$.

Confining $\rf$ to $\Om$ is not as simple. One idea is to constrain $\rf$ to a box about $\bar{\vx}$. Choose $\tilde{\vec x} \in \R^2_+$ such that 
\begin{equation}
\label{eq:fourier_box}
[\bar{x}_1 - \tilde{x}_1, \bar{x}_1 + \tilde{x}_1] \times [\bar{x}_2 - \tilde{x}_2, \bar{x}_2 + \tilde{x}_2] \subseteq \Om.
\end{equation}
Without loss of generality, focusing on the first component of $\rf$ in~\eqref{eq:fourier_path}, the triangle inequality provides
$$
\Big|\sum_{j=1}^{\Nf}a_j\cos(\om_j t) + b_j\sin(\om_jt)\Big| \leq \sum_{j=1}^{\Nf} |a_j| + |b_j|.
$$
Thus, we can confine $\rf$ to the box~\eqref{eq:fourier_box} by enforcing
\begin{equation}
\label{eq:fourier_box_constraint_1}
\sum_{j=1}^{\Nf}
\begin{bmatrix}
|a_j| + |b_j|\\
|c_j| + |d_j|
\end{bmatrix}
\leq 
\begin{bmatrix}
\tilde{x}_1\\
\tilde{x}_2
\end{bmatrix}.
\end{equation}
We note that this is a nonlinear constraint on $\xif$ and that~\eqref{eq:fourier_box_constraint_1} is satisfied whenever
\begin{equation}
\label{eq:fourier_box_constraint_2}
|a_j|,|b_j| \leq \frac{\tilde{x}_1}{2\Nf} \quad \text{and} \quad |c_j|,|d_j| \leq \frac{\tilde{x}_2}{2\Nf}, \quad j \in \{1, 2, \dots, \Nf\}.
\end{equation}
Hence,~\eqref{eq:fourier_box_constraint_2} is a linear constraint that confines the Fourier paths to a box about $\bar{\vx}$. Unfortunately, this is conservative. 
A constraint that allows the paths to explore the domain more freely is
desirable.  The following result provides a simple nonlinear constraint that
confines the sensor path to a circular subdomain of $\Omega$. 
\begin{theorem}
\label{thm:disk}
Let $\Rd > 0$ and consider 
the Fourier path $\rf$ in~\eqref{eq:fourier_path}. We have, 
\[ 
  \|\vec \rf - \xbar \|_{2} \leq \Rd, 
  \quad \text{whenever}\quad \|\vxi_{f}\|_{2} \leq \frac{\Rd}{\sqrt{\Nf}}.
\]
\end{theorem} 
\begin{proof}
See Appendix~\ref{appdx:disk}.
\end{proof}

\begin{remark}
\label{rmk:general_disk}

In the proof of Theorem~\ref{thm:disk} we use the affine structure of the
Fourier path type. In fact, we can derive a constraint on the design vector that allows us to restrict any affine path to a disk of radius $\Rd$.
Specifically, let $\mat T: \Ty \to \R^{2 \times N}$ and consider a path of the form $\vr(t;\vxi) = \bar{\vec x} + \mat T(t) \vxi$. Then, for any $\vxi \in \R^{N}$ and $\bar{\vec x} \in \R^{2}$,
$$
\|\vec r(t; \vxi) - \bar{\vx}\|_2 = \|\mat T(t) \vxi \| \leq \|\mat T(t)\|_2 \|\vxi\|_2 \leq \|\vxi\|_2 \ \underset{t \in T_{y}}{\sup} \ \|\mat T(t)\|_{2}.
$$
Hence,
\begin{equation}
\label{eq:general_disk}
\|\vr - \bar{\vx}\|_2 \leq \Rd, 
\quad \text{whenever} \quad 
\|\vxi\|_{2} \leq \frac{\Rd}{\underset{t \in T_{y}}{\sup} \ \|\mat T(t)\|_{2}}.
\end{equation}
For certain parameterizations, numerically computing the supremum of $\|\mat
T(t)\|_{2}$ over the inversion interval can be expensive. An advantage of the
Fourier path type is that we can compute $\| \mat{T}(t) \|_2$ analytically.  
Specifically, we can show that $\| \mat{T}(t)
\|_2 = \sqrt\Nf$ for all $t \in \Ty$; see Theorem~\ref{thm:disk}.
\end{remark}

\boldheading{Penalizing acceleration}
A common issue for solutions to the optimal path problem are curves with 
sharp turns or kinks. One method
for preventing this is to include a penalty term in the optimization problem. We
propose regularizing the optimality criterion $\Psi$ by adding the following term:
\begin{equation}
\label{eq:reg_acc}
R(\vxi; \gam) \defeq \frac{\gam}{2}\int_{\Ty} \|\ddot{\vr}(t;\vxi)\|_2^2 \, dt.
\end{equation}
This characterizes the time-integrated magnitude of the path's acceleration. Hence,
regularizing with $R$ will result in smoother optimal
paths. 
In practice, we find it necessary to penalize the Fourier paths in
this manner. The penalty $R$ for
the Fourier path is computed to be
\begin{equation}
R(\xif; \gam) = \frac{\gam|\Ty|}{4}\sum_{j=1}^{\Nf} \om_{j}^{4} (a_{j}^{2}+b_{j}^{2}+c_{j}^{2}+d_{j}^{2}).
\end{equation}

On the other hand, the regularity of the \Bezier{} paths can be easily controlled
by modulating the number of control points. In this case, adding such regularizing penalty terms were not needed.

\boldheading{Obscured regions}
In some applications, there may be a subdomain of $\Om$ where the state $u$ is
concealed or data is otherwise challenging to collect. We refer to such ares as
\emph{obscured regions}. Our motivation of considering such regions stems from 
wildfire applications, where smoke or other obstacles might degrade the ability
of a sensor to collect data. However, we assume that the sensor is not
necessarily prohibited, but rather discouraged, from entering such obscured
regions. 

We explain how to incorporate a disk-shaped obscured region into the path-OED
problem.  Namely, let $\xvoid \in \Om$ and $\Rvoid > 0$. Then, we define the
obscured region $D \subset \Om$ by
\begin{equation}
\Ovoid \defeq \{\vx \in \Om: \ \|\vx - \xvoid\|_2 \leq \Rvoid\}.
\end{equation} 
It is possible to force the sensor path to avoid $\Ovoid$ completely 
by incorporating a hard constraint in the path-OED formulation.
However, this adds complexity to the optimization 
problem. Here, we outline a tractable 
approach for discouraging data-acquisition within $\Ovoid$ that avoids 
such hard constraints. 
The key idea is to artificially modulate the noise variance of the data model,
filtering-out observations that occur within $\Ovoid$.  A challenge with this
approach is preserving the differentiability of $\Cpo(\vxi)$.

We consider a smoothed indicator function for $D$, denoted by $\void$. For
$\bet >0$, we define this function by
\begin{equation}
\label{eq:void}
\void(\vx) \defeq \left(1 + \exp\Big\{\frac{\|\vx - \xvoid\|_2 - \Rvoid}{\bet}\Big\}\right)^{-1}, 
\quad \vec x \in \Om. 
\end{equation}
We note that~\eqref{eq:void} is type of radial basis function
(RBF)~\cite{Buhmann2000} and remark on a few of its properties. For a
sufficiently small $\bet$ and $\vx \in \Ovoid$, we have that $\void(\vx) \approx
1$. Similarly, $\void(\vx) \approx 0$ for $\vx \not\in\Ovoid$. The parameter
$\bet$ controls the steepness of the transition region---a small $\beta$ resulting
in a closer approximation to an indicator function. 
The RBF $\void$ is depicted for various scaling parameters in
Figure~\ref{fig:rbf}.  We can extend the present framework for
obscured regions to more complex geometries by considering linear combinations
of RBFs. Namely, we can approximate obscured regions that are not disks by
summing RBFs.

\begin{figure}[ht]
\centering
\includegraphics[width=\textwidth]{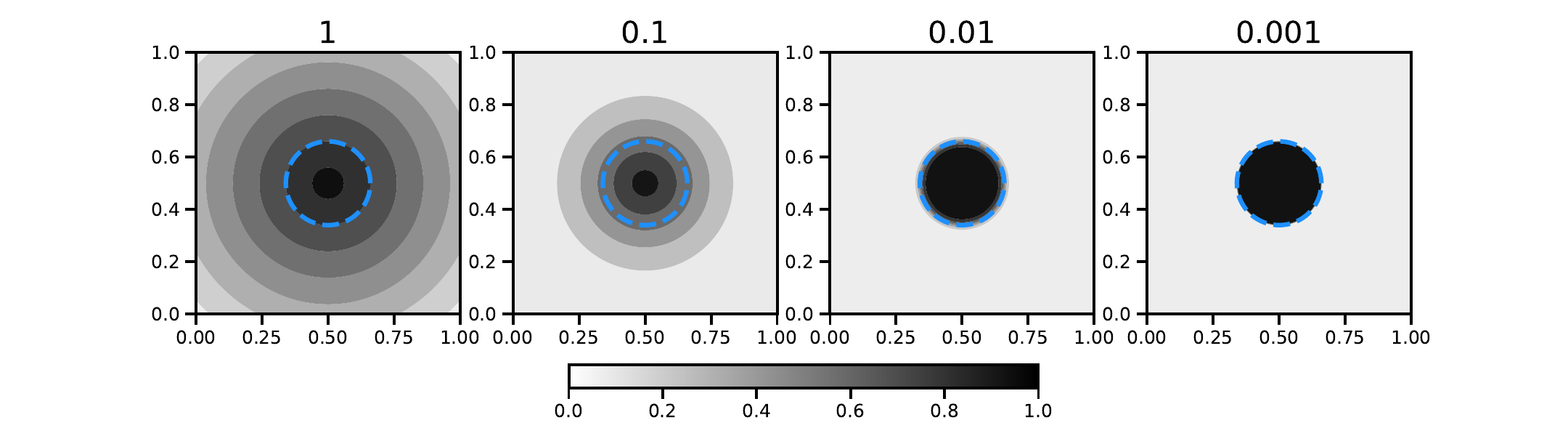}
\caption{Contours of the RBF $\void$ with $\xvoid = (0.5, 0.5)$, $\Rvoid = 0.16$, and 
$\beta \in \{1, 0.1, 0.01, 0.001\}$. The dotted blue circle represents the boundary of the obscured 
region.}
\label{fig:rbf}
\end{figure}

Consider the data-misfit Hessian $\mc{H}_{\text{mis}} (\vxi)$ in
\eqref{eq:hessian} for a given $\vxi$.  This operator can be expressed
as $\iHmis (\vxi) = \mc{F}^* (\vxi)(\sig^{-2} \mat I) \mc{F} (\vxi)$.  
This indicates that the noise covariance matrix $\sig^2\mat{I}$ (and its inverse) has the role of \emph{filtering} or weighting observations. We
extend this idea by smoothly modulating contributions of measurements locations
based on their position relative to $\Ovoid$.

Consider
$\ny$ measurement times $\{t_k\}_{k=1}^{\ny}$, and for a path $\vr$, define
$\vec\void: \R^N \to (0,1)^{\ny}$ by
\begin{equation}
\label{eq:void_vector}
\void_k(\vxi) \defeq \void(\vr(t_k;\vxi)), \quad k \in \{1, 2, \dots, \ny\}.
\end{equation}
Subsequently, we form the diagonal weighting matrix $\Void : \R^N \to \R^{\ny \times
\ny}$ defined by
\begin{equation}
\label{eq:void_matrix}
\Void (\vxi) \defeq \text{diag}(\void_1(\vxi), \ldots, \void_{\ny}(\vxi)).
\end{equation}
This is inserted into $\iHmis (\vxi)$, resulting in the ``filtered Hessian'':
\begin{equation}
\label{eq:hess_void}
\mc{\Hvoid} (\vxi) \defeq \sig^{-2}\mc{F}^*(\vxi) \big( \mat I - \Void (\vxi) \big) \mc{F}(\vxi) + \Cpr^{-1}.
\end{equation}
Using the matrix $\mat I - \Void (\vxi)$ in~\eqref{eq:hess_void} filters out observations that occur in the obscured region. 
In practice, we solve the path-OED problem using the filtered Hessian (and posterior
covariance), then discard the measurements in $\Ovoid$ when solving the
inverse problem. We derive the derivative of the filtered criterion $\Psi$ as in~\eqref{eq:psi} in Appendix~\ref{appdx:c_optimal_void}. The resulting
formula for $\pj\Psi$ contains the gradient of $\void$. This is given by
\begin{equation}
\grad \void = - \frac{\vx\void^2}{\bet \|\vx - \xvoid\|_2}\exp\Big\{\frac{\|\vx - \xvoid\|_2 - \Rvoid}{\bet}\Big\}.
\end{equation} 

The effectiveness of the present filtered approach for incorporating obscured
regions is illustrated in our numerical results in Section~\ref{sec:numerics}.

\section{Solving the discretized path-OED problem}
\label{sec:methods}

In this section, we describe how to discretize and solve the path-OED problem.
The developments here lay the foundations for an efficient computational
framework.  First, we construct finite-dimensional analogues for the parameter
space $\M$ and solution space $\U$; see Subsection~\ref{sec:hilbert}. Then, we
discretize the inverse problem in Subsection~\ref{sec:bayes_disc}.
Subsequently, we discuss the discretized version of the path-OED problem in
Subsection~\ref{sec:disc_path_OED}. In that subsection, we derive the
discretized c-optimality criterion and its gradient, and specialize the results
to the specific criterion considered in Subsection~\ref{sec:copt_specific}.
Lastly, we discuss computational considerations in
Subsection~\ref{sec:considerations}, where we measure computational cost in terms of
applications of the solution operator.

\subsection{Structure of the discretized Hilbert spaces}
\label{sec:hilbert}

We let $\{t_\ell\}_{\ell=0}^{\nt} \subset T$ be the \emph{global time discretization} and consider $\ny$ measurement times occurring during the inversion interval $\Ty$. For convenience, we assume equidistant time steps. That is, $t_{\ell+1} - t_\ell
= \Delta t$, for $\ell \in \{0, 1, \dots, n_t - 1\}$, where $\Delta t$ is a
fixed step size.  
In principle, the measurement times are independent of this discretization.
However, for simplicity, we assume they occur at some subset of the global temporal discretization. Specifically, we denote the measurement times by $\{t_{\ell_{k}}\}_{k=1}^{\ny} \subset \Ty$. This setup can be interpreted as having a sampling frequency that is a multiple
of a fixed global time step. 

For spatial discretization, we use the continuous Galerkin finite element method.
Let $\{\phi_i\}_{i=1}^{\nx}$ be a nodal finite element basis consisting of compactly
supported functions. Then, $m \in \M$ admits the approximation
\begin{equation}
m_h(\vx) \defeq \sum_{i=1}^{\nx} m_i \phi_i(\vx)
\end{equation}
and we identify $m_h$ with the
coefficient vector $\vec m \defeq [m_1 \, m_2 \, \cdots \, m_{\nx}]^\top$. Subsequently, for $c \in \M$ with finite element approximation $c_{h}$, we estimate the inner product $\langle m, c \rangle$ with the calculation
$$
\langle m_{h}, c_{h} \rangle = \int_{\Om} m_{h} c_{h} \, d\vx = \sum_{i,j=1}^{\nx} m_{i} c_{j} \langle \phi_{i}, \phi_{j} \rangle .
$$
This motivates the \emph{mass-weighted} inner product space $\RNM$. Specifically, we define the \emph{mass matrix} $\mat M$ by $M_{ij} \defeq \langle \phi_i, \phi_j \rangle$, for $i,j \in \{1, \dots, \nx\}$. Then, for $\vec m, \vec c \in \R^{\nx}$, the inner product on $\RNM$ is defined by
\begin{equation}
\label{eq:RNM}
\langle \vec m, \vec c \rangle_{\mat M} \defeq \vec m^{\top} \mat M \vec c.
\end{equation}

We use a similar strategy to discretize $\U$. Let $u\in \U$ and define the
spatial approximation $u_{h}$ of $u$ by 
\begin{equation}
u_h(\vx, t) \defeq \sum_{i=1}^{\nx} u_i(t) \phi_i(\vx).
\end{equation}
We identify $u_{h}$ with the time-dependent coefficient vector $\vec u(t)$, where $\vec u(t) \in \RNM$, for any $t \in T$. 
Next, we define $\vec u_{\ell} \defeq \vec u(t_{\ell})$, for $\ell \in \{1, \dots, \nt\}$, and form a \emph{snapshot matrix} $\mat U \defeq [\vec u_1 \, \cdots \, \vec u_{\nt}]$ of dimension $\nx \times \nt$. Thus, for $v \in \U$ with approximation $v_{h}$, we estimate the inner product $\llangle u, v \rrangle$ with
$$
\llangle u_{h}, v_{h} \rrangle = \int_{T} \left( \int_{\Om} u_{h} v_{h} \, d\vx \right) dt = \int_{T} \big \langle \vec u(t), \vec v(t) \big \rangle_{\mat M} \, dt.
$$
We obtain an inner product on the space of snapshot matrices by approximating the resulting integral via numerical quadrature. Specifically, let $\vec w \in \R^{\nt}$ denote a vector of positive quadrature weights and let $\mat V \defeq [\vec v_{1} \, \dots \, \vec v_{\nt}]$ be the snapshot matrix corresponding to $v_{h}$. Then, the inner product is
\begin{equation}
\label{eq:MT_prod}
\llangle \mat U, \mat V\rrangle_{\mat M} \defeq \sum_{\ell=1}^{\nt} w_{\ell}\langle \vec u_{\ell}, \vec v_{\ell}\rangle_{\mat M}.
\end{equation}
It is straightforward to show that~\eqref{eq:MT_prod} defines an inner product on $\R^{\nx \times \nt}$; we denote this inner product space by $\RNMT$. In Lemma~\ref{lem:time_space}, we present two useful identities for~\eqref{eq:MT_prod}. Specifically, we
require~\eqref{eq:MT_identities_one} to derive the adjoint of the discretized
observation operator.
In what follows, $\mat U \odot \mat V$ denotes the Hadamard product of 
matrices $\mat U$  and $\mat V$. Specifically, for matrices of the same dimension,
$(\mat U \odot \mat V)_{i\ell} \defeq U_{i\ell}V_{i\ell}$.

\begin{lemma}
\label{lem:time_space}
Consider~\eqref{eq:MT_prod} and let $\mat U, \mat V \in \RNMT$. Then,
\begin{subequations}
\begin{align}
\llangle \mat U, \mat V \rrangle_{\mat M} 
&= \vec 1^\top (\mat U \odot \mat M \mat V) \vec w
\label{eq:MT_identities_one}\\
&= \tr(\mat U^{\top} \mat M \mat V \mat W),
\label{eq:MT_identities_two}
\end{align}
\end{subequations}
where $\mat W \defeq \text{diag}(\vec w)$ and $\vec 1 \in \R^{\nx}$ is a constant ones vector.
\end{lemma}
\begin{proof}
We first note that $\langle \vec u_{\ell}, \vec v_{\ell} \rangle_{\mat M} = \vec 1^{\top} \left(\vec u_{\ell} \odot (\mat M \vec v_{\ell})\right)$. Subsequently, the first identity follows from 
\begin{align*}
\sum_{\ell=1}^{\nt} w_{\ell} \langle \vec u_{\ell}, \vec v_{\ell} \rangle_{\mat M} = \sum_{\ell=1}^{\nt} w_{\ell} \vec 1^{\top} \left(\vec u_{\ell} \odot (\mat M \vec v_{\ell})\right) = \vec 1^{\top} \sum_{\ell=1}^{\nt} w_{\ell} \left( \vec u_{\ell} \odot (\mat M \mat V)_{\ell}\right) = \vec 1^{\top} (\mat U \odot \mat M \mat V)\vec w.
\end{align*}
As for the second identity, we note 
$$
\tr(\mat U^\top \mat M \mat V \mat W) = \sum_{i=1}^{\nx}\sum_{\ell=1}^{\nt} U_{i\ell}(\mat M \mat V \mat W)_{i\ell} = \sum_{\ell=1}^{\nt} w_\ell \sum_{i=1}^{\nx} U_{i\ell}(\mat M \mat V)_{i\ell} = \sum_{\ell=1}^{\nt} w_{\ell} \langle \vec u_{\ell}, \vec v_{\ell} \rangle_{\mat M}.
$$
\end{proof}

\begin{remark}
Lemma~\ref{lem:time_space} shows that the $\RNMT$ inner product can be viewed
as a weighted Frobenius inner product. 
Recall that the Frobenius inner product~\cite[Chapter 5.2]{horn2012matrix} of matrices $\mat A_1$ and 
$\mat A_2$ 
(for which $\mat A_{1}^{\top} \mat A_{2}$ is well-defined) is  
\begin{equation}
\langle \mat A_{1}, \mat A_{2} \rangle_{\text{F}} \defeq \tr(\mat A_{1}^{\top} \mat A_{2}) = \sum_{i,j} A_{ij}B_{ij}.
\end{equation}
Now, consider $\mat U, \mat V \in \RNMT$. By Lemma~\ref{lem:time_space}, 
$\llangle \mat U, \mat V \rrangle_{\mat M} = \tr(\mat U^\top \mat M \mat V \mat W) = \langle \mat U, \mat M \mat V \mat W \rangle_{\text{F}}$.
Hence, $\llangle \cdot, \cdot \rrangle_{\mat M}$ is a Frobenius inner product that is weighted according  
to the spatiotemporal discretization.
\end{remark}

\subsection{The discretized inverse problem}
\label{sec:bayes_disc}
In what follows we derive the discretized version of the inverse problem
presented in Section~\ref{sec:inverse_problem}. We first discuss the discrete
solution and observation operators.  A key result is
Proposition~\ref{prop:disc_obs}, which provides the adjoint of the discretized
observation operator.  Subsequently, we provide formulas for the posterior mean and
covariance operator associated with the finite element coefficients of the
discretized inversion parameter.

\boldheading{The solution and observation operators} 
The discretized solution operator is denoted by
\begin{equation}
\label{eq:disc_solution}
\mat S: \RNM \to \RNMT.
\end{equation}
This linear operator maps coefficient vectors in $\RNM$ to snapshot
matrices in $\RNMT$. 
The adjoint of the solution operator is denoted by $\mat S^*:\RNMT \to \RNM$.
In Appendix~\ref{appdx:solution}, we derive the discretized solution operator
and its adjoint for the model used in our numerical experiments.

Next, we construct the 
discretized observation operator
\begin{equation}
\mat B: \RNMT \to \R^{\ny}.
\end{equation}
We proceed in two steps. First, we describe how to obtain pointwise measurements of the finite element approximation $u_h$ at all times in the discretization $\{t_{\ell}\}_{\ell=1}^{\nt} \subset T$. Then, we define an operator that restricts these evaluations to the measurement times $\{t_{\ell_{k}}\}_{k=1}^{\ny} \subset \Ty$. Combining these operations together produces the action of the observation operator.

For a time $t_\ell$ in the global discretization of $T$ and  a path $\vr$, we define $\vec \phi_\ell: \R^{N} \to \R^{\nx}$ by
\begin{equation}
\label{eq:pointsource_vector}
\vec \phi_\ell(\vxi) \defeq [\phi_1(\vr_\ell) \, \phi_2(\vr_\ell) \, \cdots \, \phi_{n_{x}}(\vr_\ell)]^\top, \quad \text{where} \quad \vr_\ell \defeq \vr(t_\ell; \vxi).
\end{equation}
Note that $\vec \phi_\ell$ consists of basis functions evaluated at $\vr_\ell
\in \Om$. For an approximation $m_h$ of $m \in \M$ with coefficients $\vec m \in
\RNM$, a measurement of $m$ at $\vr_\ell$ can be estimated by $\langle \vec m,
\vec \phi_\ell \rangle_2$. To see this, note that $\langle \dr(t_\ell; \vxi),
m\rangle \approx m_{h}(\vr_\ell) = \sum_{i=1}^{\nx} m_i \phi_i(\vr_\ell).$
Similarly, if $u_h$ is the finite element approximation of $u \in \ms{U}$ with associated snapshot matrix $\mat U \in \RNMT$, then a pointwise measurement of $u$ at $(\vr_{\ell},
t_{\ell})$ can be estimated by $\langle \vec u_{\ell},
\vec\phi_{\ell}\rangle_{2}$. Subsequently, we form the matrix $\mat \Phi: \R^N \to \in \R^{\nx \times \nt}$, defined by
\begin{equation}
\label{eq:pointsource_matrix}
\mat \Phi(\vxi) \defeq [\vec \phi_1(\vxi) \, \vec \phi_2(\vxi) \, \cdots \, \vec \phi_{\nt}(\vxi)].
\end{equation}
Using $\mat \Phi$, we observe the columns of $\mat U$ at the points $\{ (\vr_{\ell}, t_{\ell}) \}_{\ell=1}^{\nt}$ via the calculation
\begin{equation}
\label{eq:disc_pointwise_obs}
\big( \mat U \odot \mat \Phi(\vxi) \big)^{\top}\vec 1 = \big[ u_{h}(\vr_{1}, t_{1}) \,  u_{h}(\vr_{2}, t_{2}) \, \cdots \,  u_{h}(\vr_{\nt}, t_{\nt}) \big]^{\top}.
\end{equation}

The operation in~\eqref{eq:disc_pointwise_obs} produces a measurement vector of length $\nt$.
Hence, the final step in constructing $\mat B$ is to discard measurements that do not occur at the times $\{t_{\ell_{k}}\}_{k=1}^{\ny}$. For this, we require an \emph{extraction operator}. Let $\vec e_\ell \in \R^{\nt}$ be the
$\ell${th} standard Euclidean basis vector and define $\mat E \in
\R^{\nt \times \ny}$ by $\mat E \defeq [\vec e_{\ell_1} \, \vec e_{\ell_2} \, \cdots \,
\vec e_{\ell_{\ny}}].$ Then, the extraction operator $\mc{E}: \R^{\nx \times
\nt} \to \R^{\nx \times \ny}$ is
\begin{equation}
\label{eq:extraction_operator}
\mc{E}\mat U \defeq \mat U \mat E, \quad \mat U \in \RNMT.
\end{equation}
It is straightforward to show that right-multiplying $\mat U$ by $\mat E$ yields a matrix in $\R^{\nx \times \ny}$ whose columns are those of $\mat U$ indexed by the set $\{\ell_k\}_{k=1}^{\ny}$. By combining~\eqref{eq:disc_pointwise_obs} and~\eqref{eq:extraction_operator}, the action of $\mat B$ on $\mat U$ is defined by
\begin{equation}
\label{eq:disc_obs}
\mat B(\vxi) \mat U \defeq \big( \mc{E}\mat U \odot \mc{E}\mat \Phi(\vxi) \big)^\top\vec 1.
\end{equation}

To formulate the discretized inverse problem we need 
the adjoint operator $\mat B^{*}:\R^{\ny} \to \RNMT$. 
The following proposition provides a formula for $\mat B^{*}$ and enables 
matrix-free implementations. 
\begin{proposition}
\label{prop:disc_obs}
Let $\mat U \in \RNMT$, $\vxi \in \R^{N}$, and define $\vec w^{-1} \in \R^{\nt}$ by $(\vec w^{-1})_\ell \defeq 1 / w_\ell$, for $\ell \in \{1, 2, \dots, \nt\}$. 
Then, the adjoint of $\mat{B}(\vxi)$ is given by 
\begin{equation}
\label{eq:disc_obs_adj}
\mat B(\vxi)^* \vec y = \mat M^{-1}\mat \Phi(\vxi) \mathup{diag}\{\mat E \vec y \odot \vec w^{-1}\}, \quad \text{for} \quad \vec y \in \R^{\ny}.
\end{equation}
\end{proposition}

\begin{proof}

In what follows, we denote the
$\ell$th column of a matrix $\mat A$ by $\mat A_{*, \ell}$. Moreover, 
since the dependence of $\mat{B}$ on $\vxi$ does not 
play a role here, for simplicity, we suppress $\vxi$.

Note that by the definition of the extraction operator, 
$$
\mc{E} \mat U \odot \mc{E} \mat\Phi = \mc{E}(\mat \Phi \odot \mat U).
$$
Thus, by~\eqref{eq:disc_obs}, we have the identity
\begin{equation}
\label{eq:extract_identity}
\mat B \mat U = \mat E^\top(\mat U \odot \mat \Phi)^\top \vec 1.
\end{equation}
Using~\eqref{eq:extract_identity} and writing the matrix-vector product 
$(\mat U \odot \mat \Phi) \mat E \vec y$ as a sum, we have 
$$
\langle \mat B \mat U, \vec y \rangle_2 = \vec 1^\top (\mat U \odot \mat \Phi) \mat E \vec y = \vec 1^{\top} \sum_{\ell=1}^{\nt} (\mat E \vec y)_{\ell} (\mat U \odot \mat \Phi)_{*,\ell}.
$$
It follows directly from the definition of the Hadamard product that $(\mat U \odot \mat \Phi)_{*, \ell} = \vec u_{\ell} \odot \mat \Phi_{*,\ell}$. 
Thus, 
$$
\vec 1^{\top} \sum_{\ell=1}^{\nt} (\mat E \vec y)_{\ell} (\mat U \odot \mat \Phi)_{*,\ell} 
= \vec 1^{\top} \sum_{\ell=1}^{\nt}  (\mat E\vec y)_{\ell}(\vec u_{\ell} \odot \mat \Phi_{*,\ell})
= \vec 1^{\top} \sum_{\ell=1}^{\nt} w_{\ell} \, \vec u_{\ell} \odot \left( \frac{(\mat E\vec y)_{\ell}}{w_{\ell}}\mat \Phi_{*,\ell} \right).
$$
We recognize that the vector $\frac{(\mat E\vec y)_{\ell}}{w_{\ell}}\mat \Phi_{*,\ell}$ is the $\ell \text{th}$ column 
of the matrix $\mat{K} \defeq \mat \Phi \text{diag}\{\mat E \vec y \odot \vec w^{-1}\}$. 
So, we have
\begin{align*}
\vec 1^{\top} \sum_{\ell=1}^{\nt} w_{\ell} \, \vec u_{\ell} \odot \left( \frac{(\mat E\vec y)_{\ell}}{w_{\ell}}\mat \Phi_{*,\ell} \right) 
&= \vec 1^{\top} \sum_{\ell=1}^{\nt} w_{\ell} \left(\mat U \odot \mat K \right)_{\ell}\\
&= \vec 1^{\top} \left(\mat U \odot \mat M \mat M^{-1} \mat K\right) \vec w\\
&= \llangle \mat U, \mat M^{-1} \mat K\rrangle_{\mat M},
\end{align*}
where in the last step we have used~\eqref{eq:MT_identities_one}.
Hence, we yield the desired result, since for every $\vec{y} \in \R^{\ny}$, 
$$\langle \mat B \mat U, \vec y \rangle_2 = 
\vec 1^\top (\mat U \odot \mat \Phi) \mat E \vec y 
= \llangle \mat U, \mat M^{-1} \mat K\rrangle_{\mat M}.$$
\end{proof}

We comment on the efficiency of applying $\mat B$ and its adjoint to elements in
their respective spaces. Given a path parameter $\vxi$, we must form the matrix
$\mat \Phi$. Since $\phi_i$ are compactly supported, $\mat \Phi$ is sparse. We
can build $\mat \Phi$ with a geometric approach---only evaluating basis functions whose
support contains the point $\vr_{\ell}$. Explicit construction and manipulation
of the extraction matrix $\mat E$ can be avoided with indexing.

\boldheading{The inverse problem} We are now ready to formulate the discretized inverse problem. 
The inversion parameter is given by the finite element coefficients $\vec m \in \RNM$ corresponding to the approximation $m_h$ of $m \in \M$. We treat $\vec m$ as a random variable and seek a posterior distribution through Bayesian inversion. 
A key component in this construction is the discretized forward operator
$\mat F: \RNM \to \R^{\ny}$, defined by
\begin{equation}
\mat F(\vxi) \defeq \mat B(\vxi) \mat S.
\end{equation}
Analogous to the infinite-dimensional setting, data are acquired according 
to $\vec y = \mat F(\vxi) \mat m + \vec \eta$, where $\vec \eta$ is the same
Gaussian noise as in~\eqref{eq:data_model}. Under the assumption that $\vec m$
and $\vec \eta$ are uncorrelated, $\vec y \vert \vec m \sim \mc{N}(\mat F \vec
m, \sig^2 \mat I)$. 

The prior law for $\vec m$ is $\mu_{\pr} = \mc{N}(\vec m_{\pr}, \Gpr)$, where $\Gpr$ is the finite-element discretization of the elliptic differential operator in~\eqref{eq:prior_operator}.  Combining the data
likelihood with the prior, we apply Bayes
theorem to obtain the posterior law $\mupoy = \mc{N}(\fmMAPy, \Gpo)$, where 
\begin{equation}
\label{eq:discrete_post}
\Gpo(\vxi) = (\sig^{-2}(\mat F^*\mat F)(\vxi) + \Gpr^{-1})^{-1} \quad \text{and} \quad \fmMAPy(\vxi) = \Gpo(\vxi)\big(\sig^{-2}\mat F^*(\vxi) \vec y + \Gpr^{-1}\vec m_\pr\big).
\end{equation}
The discretized Hessian, which is derived from the variational characterization
of the MAP point, is given by
\begin{equation}
\mat H(\vxi) = \fHmis(\vxi) + \Gpr^{-1}, \quad \text{where} \quad \fHmis(\vxi) = \sig^{-2}(\mat F^{*} \mat F)(\vxi)
\end{equation}
is the data-misfit Hessian. Again, we have the important identity $\mat H(\vxi) = \Gpo^{-1}(\vxi)$.

For a scalable numerical framework we must be able to apply $\Gpo$ to vectors in
$\RNM$ efficiently. We follow the approach outlined
in~\cite{Bui-ThanhGhattasMartinEtAl13}, where a low-rank spectral decomposition
and the Sherman--Morrison--Woodbury formula are used to obtain an approximation
$\Gpor$ to $\Gpo$.  For $\vxi \in \R^{N}$, 
let $\big\{ (\lam_{i} (\vxi), \vec v_{i} (\vxi) ) \big\}_{i=1}^{r}$ be the $r$ leading
eigenpairs of the \emph{prior-preconditioned the data misfit
Hessian}
$\Gpr^{1/2}\fHmis(\vxi) \Gpr^{1/2}$. Then, the action of $\Gpor (\vxi)$ on $\vec m \in \RNM$ is defined by
\begin{equation}
\label{eq:Gpor_apply}
\Gpor (\vxi) \vec m \defeq \Gpr \vec m - \sum_{i=1}^{r} \frac{\lam_{i} (\vxi)}{\lam_{i} (\vxi) + 1} \big \langle \vec m, \Gpr^{1/2} \vec v_{i} (\vxi) \big\rangle_{\mat M} \Gpr^{1/2} \vec v_{i} (\vxi).
\end{equation}

\subsection{The discretized path-OED problem}
\label{sec:disc_path_OED}
Here we discretize the c-optimality criterion and its gradient introduced in Section~\ref{sec:OED}. We then specialize these results to the c-optimality criterion used in our numerical experiments, discussed in Subsection~\ref{sec:copt_specific}.

Recall the infinite-dimensional representation of the c-optimality
criterion in~\eqref{eq:c_optimal}. The corresponding discretized 
criterion is given by 
\begin{equation}
\label{eq:disc_c_opt}
\mat \Psi (\vxi) \defeq \langle \Gpo (\vxi) \vec c, \vec c\rangle_{\mat M}, 
\end{equation}
where $\vec c$ is a fixed vector in $\RNM$.
Also, for $j \in \{1, \dots, N\}$, we recall the formula for $\pj \Psi$ in~\eqref{eq:psi_c_j}. Replacing the present operators with their discrete analogues, we obtain the following derivative formula:
\begin{equation}
\label{eq:disc_c_opt_grad}
\pj \mat \Psi =  -2\sig^{-2} \langle \mat B  \mat S \Gpo \vec c, (\pj\mat B) \mat S \Gpo \vec c\rangle_2.
\end{equation}
It is straightforward to show that for any $\mat U \in \RNMT$, the action of $\pj \mc{B}$ on $\mat U$ is such that
\begin{equation}
\label{eq:discrete_obs_deriv}
(\pj \mat B(\vxi)) \mat U = \pj \big(\mc{E} \mat U \odot \mc{E} \mat \Phi (\vxi) \big)^\top\vec 1 = \big(\mc{E} \mat U \odot \mc{E} \pj \mat \Phi (\vxi) \big)^\top\vec 1.
\end{equation}
Furthermore, the $j\text{th}$ partial of $\mat \Phi: \R^{N} \to \R^{\nx \times \nt}$, defined in~\eqref{eq:pointsource_matrix}, is given by
\begin{equation}
\label{eq:Phi_j}
(\pj \mat \Phi (\vxi))_{i\ell} = \pj \phi_{i}(\vr_{\ell}(\vxi)) = \grad \phi_{i}(\vr_{\ell}(\vxi)) \cdot \pj \vr_{\ell}(\vxi),
\end{equation}
where $i \in \{1, 2, \dots, \nx\}$ and $\ell \in \{1, 2, \dots, \nt\}$. 

\boldheading{The specific criterion} We next consider the goal-functional $Z$ in~\eqref{eq:goal} which 
we use to form the specific c-optimality criterion in Subsection~\ref{sec:copt_specific}. For a fixed $\mat V \in \RNMT$, the discretized analogue of $Z$ is
\begin{equation}
\label{eq:disc_goal}
\mat Z(\vec m) \defeq \llangle \mat V, \mat S \vec m\rrangle_{\mat M} = 
\langle \mat S^{*}
\mat V, \vec m \rangle_{\mat M}, \quad \vec m \in \RNM.
\end{equation}
Hence, taking $\vec c = \mat S^{*} \mat V$ in~\eqref{eq:disc_c_opt}, the specific discretized criterion and its derivative are
\begin{subequations}
\label{eq:disc_psi}
\begin{align}
\mat \Psi (\vxi) 
&= \langle \Gpo (\vxi) \mat S^* \mat V, \mat S^* \mat V \rangle_{\mat M},
\label{eq:disc_psi_def}\\
\pj \mat \Psi 
&= -2\sig^{-2}
\big \langle \mat B (\vxi) \mat S \Gpo (\vxi) \mat S^* \mat V,
\big( \pj \mat B (\vxi) \big) \mat S \Gpo(\vxi) \mat S^* \mat V \big\rangle_2.
\label{eq:disc_psi_grad}
\end{align}
\end{subequations}
The procedure for computing $\mat \Psi(\vxi)$ and $\grad \mat \Psi(\vxi)$ are outlined in Algorithm~\ref{alg:grad_psi}.

\begin{algorithm}[H]
  \caption{Algorithm for computing $\mat \Psi$ in~\eqref{eq:disc_psi_def} and $\grad \mat \Psi$ defined by~\eqref{eq:disc_psi_grad}}
  \begin{algorithmic}[1]
    \State \textbf{Input:} design $\vxi \in \R^{N}$ 
    \State \textbf{Output:} $\mat \Psi(\vxi)$ and $\grad \mat \Psi(\vxi)$
    \State Compute $\vec{c} \defeq \mat S^{*} \mat V$
    \State Compute $\tilde{\vec c} \defeq \Gpor \vec c$
    \hfill\Comment{with $\Gpor$ as in \eqref{eq:Gpor_apply}}
    \State Compute $\mat \Psi(\vxi) = \langle \tilde{\vec c}, \vec c \rangle_{\mat M}$
    \State Compute $\tilde{\mat V} \defeq \mat S \tilde{\vec c}$
    \State Compute ${\vec{d}} \defeq \mat B(\vxi) \tilde{\mat V}$
    \For{$j = 1$ to $N$}
      \State Compute ${\vec{v}}_j \defeq \big( \pj \mat B (\vxi) \big) \tilde{\mat V}$
      \hfill\Comment{with $\pj \mat B (\vxi)$ as in \eqref{eq:discrete_obs_deriv}}
      \State Set $\pj \mat \Psi (\vxi) = -2\sig^{-2}\langle {\vec d}, {\vec{v}}_j \rangle_{2}$
    \EndFor
    \State \Return $\mat \Psi (\vxi)$ and $\grad \mat \Psi (\vxi)$
  \end{algorithmic}
  \label{alg:grad_psi}
\end{algorithm}

\subsection{Computational considerations}
\label{sec:considerations}

We design the discretized path-OED framework so that standard gradient-based optimization algorithms can be used to minimize $\mat \Psi$. The computational cost of this process depends primarily on two components: evaluation of the sensor parameterization $\vr(t;\vxi)$ and repeated application of operators arising from the inverse problem. In practice, the dominant cost is associated with applying the discretized solution operator and its adjoint, which appear multiple times in Algorithm~\ref{alg:grad_psi}, most notably through the low-rank approximation to the posterior covariance operator. While applying the discretized prior covariance operator $\Gpr$ also requires PDE solves, this cost is typically negligible when compared to applications of $\mat S$ and $\mat S^{*}$. Hence, in what follows, we quantify cost in terms of \emph{PDE solves}, meaning applications of $\mat S$ or $\mat S^{*}$.

\boldheading{Forming the observation operator} 
As outlined in Subsection~\ref{sec:bayes_disc}, for a given $\vxi \in \R^{N}$,
applying $\mat B(\vxi)$ and its adjoint $\mat B^{*}(\vxi)$ requires evaluating
the path $\vr(t;\vxi)$ over the temporal discretization
$\{t_{\ell}\}_{\ell=1}^{\nt}$. The expensiveness of this computation depends
entirely upon the selected parameterization. For the path types presented in
this work---Fourier and \Bezier{}---this is very cheap. Specifically, computing
the Fourier path in~\eqref{eq:fourier_path} requires $\Nf$
evaluations of cosine and sine over the set $\{t_{\ell}\}_{\ell=1}^{\nt}$. For the
\Bezier{} path~\eqref{eq:bezier_path}, we can
utilize de Casteljau's algorithm~\cite{boehm1999casteljau}. This can be
efficiently implemented with software like the \texttt{Python} package
\texttt{bezier}~\cite{Hermes2017}, used in our implementation.

\boldheading{Applying the posterior covariance operator} 
In Section~\ref{sec:numerics} we use the low-rank approximation $\Gpor$
in~\eqref{eq:Gpor_apply} to the discretized posterior covariance operator
$\Gpo$. As discussed in Subsection~\ref{sec:bayes_disc}, applying $\Gpor$ to
vectors in $\RNM$ involves solving an eigenvalue problem: we require
the leading eigenvalues and eigenvectors of the \emph{prior-preconditioned data
misfit Hessian} $\sig^{-2}\Gpr^{1/2}(\mat F^{*}\mat F)(\vxi) \Gpr^{1/2}$.  Since
$\mat B(\vxi)$ is a mapping from $\RNMT$ to $\R^{\ny}$ and $\mat F(\vxi) = \mat
B(\vxi) \mat S$, the rank of $(\mat F^{*}\mat F)(\vxi)$ is at most $\ny$. Thus, we choose the target rank $r$ such that $r \le \ny$, and computing the $r$ leading eigenpairs requires $k \geq r$ iterations of the Lanczos method. Since each iteration involves applying both $\mat S$ and $\mat S^{*}$, the resulting computation requires $2k$ PDE solves.

\boldheading{Computing the criterion and gradient} 
Referring to Algorithm~\ref{alg:grad_psi}, we observe that the computations required to evaluate $\mat\Psi$ are reused when forming $\grad \mat \Psi$. Hence, we can accumulate the total number of PDE solves required for both computations. Applying $\mat S^{*}$ to $\mat V$ in Line~3 of the Algorithm~\ref{alg:grad_psi} requires a single PDE solve and, as noted above, computing $\Gpor(\vxi)(\mat S^{*}\mat V)$ in Line~4 results in $2k$ additional solves. Thus, evaluating $\mat \Psi$ requires $2k + 1$ PDE solves. To compute the gradient, we apply $\mat S$ once more in Line~6, bringing the total number of PDE solves in Algorithm~\ref{alg:grad_psi} to $2k + 2$, where $k \ge r$ denotes the number of Lanczos iterations required to compute the rank-$r$ approximation $\Gpor$. 

\boldheading{A PDE-free approach}
We remark that PDE solves can be entirely avoided in Algorithm~\ref{alg:grad_psi} by replacing the discretized solution operator with a surrogate. However, this is not achieved by simply substituting $\mat S$ with a precomputed matrix, since its codomain $\RNMT$ is a space of matrices equipped with the inner product shown in~\eqref{eq:MT_prod}. Consequently, surrogates must respect the underlying Hilbert space structure; for example, neural operators or tensor decompositions constructed to preserve these inner products can be considered.

\section{Computational results}
\label{sec:numerics}

We now demonstrate our framework on an example problem. Two variants are considered: one with Bézier paths and the other Fourier paths. The aim is not to compare their performance, which depends on the specific problem, but to illustrate that our path-OED formulation is stable under a change in path type.

\subsection{A motivating problem}

Consider the convection-diffusion equation 
\begin{equation}
\label{eq:conv_diff}
\begin{aligned}
u_t - \alpha \Delta u + \vec F \cdot \grad u &= m(\vx)a(t) && \quad \text{in } \Om \times T, \\
u &= 0, && \quad \text{on } E_0 \times T,\\
\grad u \cdot \vec n &= 0, && \quad \text{on } E_{n} \times T,\\
u(\vx, 0) &= u_0(\vx, t), && \quad \text{in } \Om.
\end{aligned}
\end{equation}
This equation models the diffusive transport of a contaminant concentration over
time.  We let the domain $\Om$ be $(0,1)^2$ with boundary $\partial \Omega = E_0
\cup E_n$, where $E_0 \cap E_n = \emptyset$.  The diffusion constant $\alpha$ is
set to $\alpha = 0.15$.  We specify the velocity field $\vec F$ in the numerical
experiments in Sections \ref{sec:bezier_results} and \ref{sec:fourier_results}.
For all numerical experiments we assume $u_{0} \equiv 0$. 

The inverse problem seeks to estimate $m$ in~\eqref{eq:conv_diff} using measurements 
of $u$ collected by a moving sensor. Here,  the solution operator $\mc{S}$ maps
the inversion parameter $m$---the spatial component of the source term---to the
solution $u$ of~\eqref{eq:conv_diff}.  Applying the adjoint of $\mc{S}$ to a
vector requires an adjoint PDE solve.  This is described in
Appendix~\ref{appdx:solution}, in the discretized setting.
We use a Gaussian prior as discussed in Section~\ref{sec:inverse_problem}.
Specifically, we use $(a_1, a_2) = (5.5\times10^{-1}, 6 \times 10^{-3})$
in~\eqref{eq:prior_operator}.  As mentioned in Section~\ref{sec:inverse_problem}, we assumed that observation errors are constant and uncorrelated. Subsequently, we take the noise variance in the following experiments to be
$\sig^{2} = 10^{-3}$, resulting in approximately $1\%$ noise.

The path-OED problem consists of minimizing the c-optimality criterion $\Psi$
discussed in Subsection~\ref{sec:copt_specific}. Recall that $\Psi$ characterizes
the variance of a goal-functional $Z = \llangle v, u(m) \rrangle$ where $u(m) =
\mc{S} m$ is the solution to~\eqref{eq:conv_diff} and $v \in \U$. We choose $v$
to be an indicator function for a spatiotemporal region $\Om_{v} \times T_{v}$.
Here, $\Om_{v} \subset \Om$ is open and $T_{v} \subset T$ is an interval such
that $t_{v} > t$ for all $(t_{v}, t) \in T_{v} \times \Ty$. Subsequently, $v$ is
defined by
\begin{equation}
\label{eq:space_time_ind}
v(\vx, t) \defeq
\begin{cases}
1, & (\vx, t) \in \Om_{v} \times T_{v},\\
0, & (\vx, t) \notin \Om_{v} \times T_{v},
\end{cases}
\end{equation}
so our example goal-functional is then
\begin{equation}
\label{eq:numerics_goal}
Z(m) = \llangle v, u(m) \rrangle = \int_{T_{v}} \int_{\Om_{v}} v(\vx, t) u(\vx, t) \, d\vx \, dt.
\end{equation}
Hence, $Z$ quantifies the concentration of the solution $u$ over some
spatial region $\Om_{v}$ and time interval $T_{v}$. 

For the following analysis, we note that $Z(m)$ is a univariate random variable. Specifically, in the case that $\mu = \mc{N}(\bar{m}, \mc{C})$ is a Gaussian probability law, $Z$ is Gaussian with expectation and variance given by
\begin{equation}
\label{eq:goal_mean_and_var}
\mE_{\mu} \{ Z(m) \} = \langle \mc{S}^{*} v, \bar{m} \rangle \quad  \text{and} \quad \mV_{\mu} \{ Z(m) \} = \langle \mc{C} \mc{S}^{*}v, \mc{S}^{*}v \rangle.
\end{equation}
Hence, we can obtain an analytic formula for the probability density function (PDF) associated with $Z(m)$. When $\mu = \mupoy$, this PDF is termed the \emph{posterior goal-density} and is used to assess how effectively an optimal design infers the synthetic true goal-value.

We must discretize $v$ to form the discrete optimality criterion $\mat \Psi$ in~\eqref{eq:disc_psi}. For the times $\{t_{\ell}\}_{\ell=1}^{nt}$ in the discretization of $T$, define the indicator vector $\vec v_{t} \in \R^{\nt}$ by
\begin{equation}
\label{eq:time_mask}
(\vec v_{t})_{\ell} \defeq
\begin{cases}
1, & \quad t_{\ell} \in T_{v},\\
0, & \quad t_{\ell} \not\in T_{v}.
\end{cases}
\end{equation}
Similarly, let $\vec v_{\vx} \in \RNM$ be the finite element coefficients for an indicator function on $\Om_{v}$. Then, the snapshot matrix corresponding to the discretization of $v$ is $\vec v_{\vx} \vec v_{t}^{\top} \in \RNMT$ and the discretized criterion is
\begin{equation}
\label{eq:psi_indicator}
\mat \Psi(\vxi) = \langle \Gpo(\vxi) \mat S^* (\vec v_{\vx} \vec v_{t}^{\top}), \mat S^* (\vec v_{\vx} \vec v_{t}^{\top}) \rangle_{\mat M}.
\end{equation}

Lastly, we discuss the discretization level and techniques used to solve the path-OED problem. For all experiments, we use $\nx = 35^2$ uniformly spaced spatial grid points and $\nt = 400$ time steps. Finite element discretization is performed using the \texttt{FEniCS} software package~\cite{fenics2015}.
As discussed in Remark~\ref{rmk:moll}, the optimality criterion is nonconvex with potentially multiple local minima. Thus, we minimize $\Psi$ with the limited-memory Broyden–Fletcher–Goldfarb–Shanno (L-BFGS) algorithm~\cite{byrd1995limited}, as implemented in \texttt{SciPy}~\cite{2020SciPy-NMeth}. To address sensitivity of the results with respect to the choice of initial iterate, we perform multiple optimization runs with $10$ initial iterates drawn via Latin hypercube sampling~\cite{mckay2000comparison}, using relaxed convergence tolerances. The best candidate is then used to initialize a final run with stricter tolerances.

\subsection{An experiment with \Bezier{} paths}
\label{sec:bezier_results}
Here we demonstrate our framework with \Bezier{} paths. After discussing the
experimental configuration, we solve the inverse problem and visualize some posterior samples. Next, we solve two variants of the path-OED problem: one
where the initial and terminal positions of the paths do not coincide, and one
where they do coincide (resulting in closed \Bezier{} paths).  Lastly, we
perform a study with an obscured region, representing an area where
data are not useful.

\boldheading{Problem setup}
For boundary conditions in~\eqref{eq:conv_diff}, we take $E_0$ to be the union of the \emph{left} and
\emph{top} of $\Om$, and $E_n$ as the union of the \emph{right} and
\emph{bottom} sides. The global time interval is set to $T = [0, 1]$ and the
inversion interval is $\Ty = [0.2, 0.4]$. The sensor makes measurements at every
other time step in the discretization of $\Ty$, resulting in $40$
measurement times. The velocity field and amplitude function
in~\eqref{eq:conv_diff} are set to 
$$
\vec F(\vec x) =
\begin{bmatrix}
1/\sqrt{2}\\
-1/\sqrt{2}
\end{bmatrix}
\quad
\text{and}
\quad
a(t) = - \frac{1}{4} \cos(4 \pi t) + \frac{3}{4}.
$$
We depict the velocity and amplitude in Figure~\ref{fig:bezier_vel_amp}. With this model setup, the concentration alternates between amplification and damping, while being transported in the direction of outflow.
\begin{figure}[ht]
\centering
\includegraphics[height=0.25\textwidth]{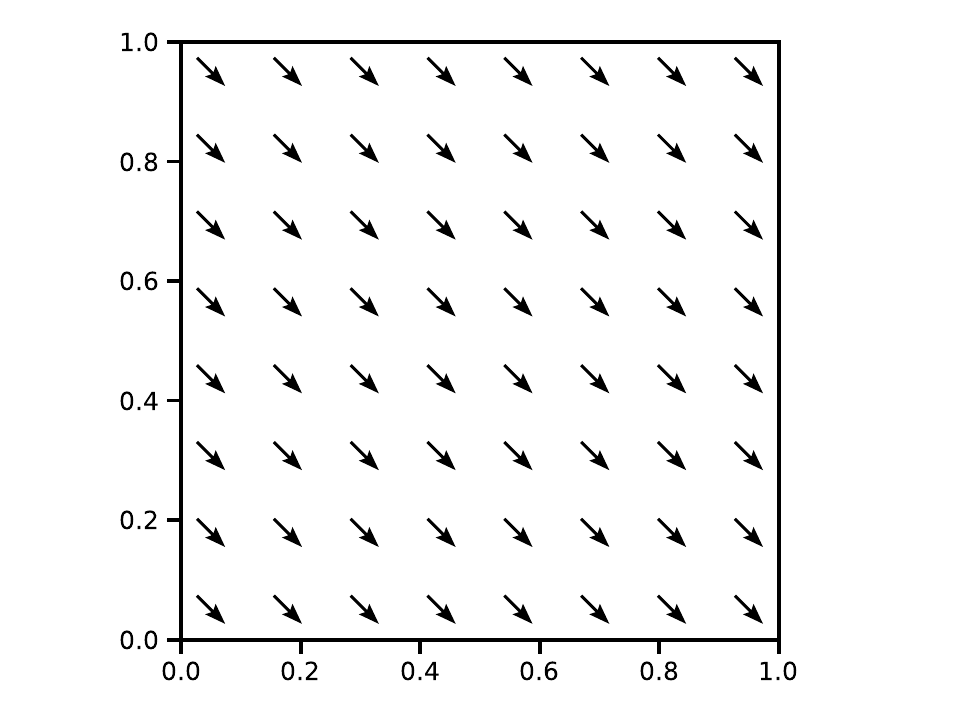}
\includegraphics[height=0.25\textwidth]{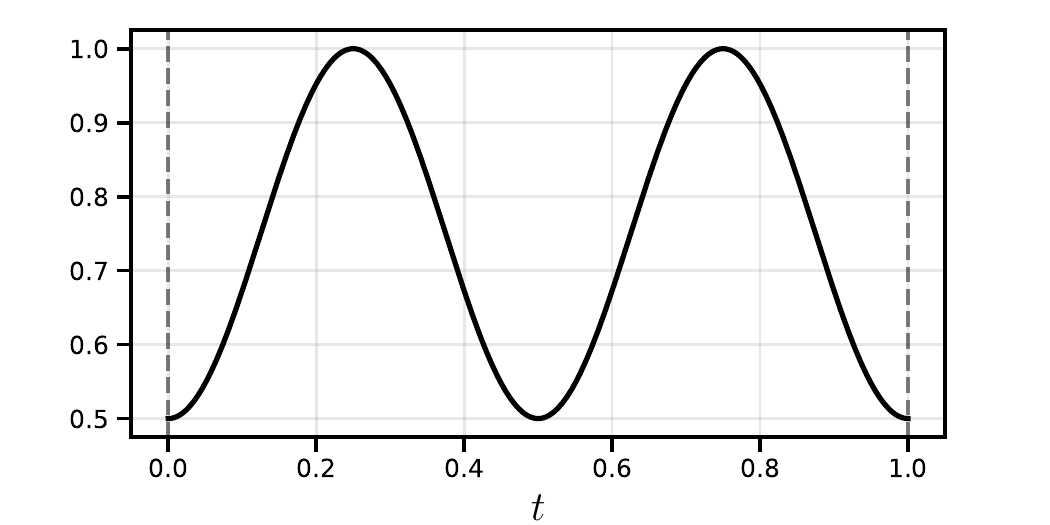}
\caption{The velocity field $\vec F(\vec x)$ (left) and amplitude $a(t)$ (right) used for the Bezier{} experiment.}
\label{fig:bezier_vel_amp}
\end{figure}

Lastly, we specify the subregions used to define the spatiotemporal indicator function $v$ in~\eqref{eq:space_time_ind}. We select
$$
\Om_{v} = (0.5, 0.9) \times (0.1, 0.5) \quad \text{and} \quad T_{v} = [0.8, 1].
$$

\boldheading{Solving the inverse problem} 
To demonstrate our framework we specify a ground truth inversion parameter
$m_{\text{true}}$. This is decided to be a Gaussian-like function oriented in the
upper-left quadrant of $\Om$. Applying the solution operator $\mc{S}$ to
$m_{\text{true}}$, we obtain the true solution $u_{\text{true}} \defeq
\mc{S}m_{\text{true}}$ to~\eqref{eq:conv_diff}. 
We then collect data along a nominal \Bezier{} path $\rb$ of degree 3 
with control points $\{\vec p_{0}, \ \vec p_{1}, \ \vec p_{2}, \ \vec p_{3}\}$ specified by
$$
\vec p_{0} =
\begin{bmatrix}
0.2\\
0.1
\end{bmatrix}, \
\vec p_{1} =
\begin{bmatrix}
0.2\\
2.5
\end{bmatrix}, \
\vec p_{2} =
\begin{bmatrix}
0.8\\
-1.5
\end{bmatrix}, \ 
\vec p_{3} =
\begin{bmatrix}
0.8\\
0.9
\end{bmatrix}.
$$
Figure~\ref{fig:bezier_time_series} depicts $u_{\text{true}}$ at the times $t \in \{0.1,,0.2,,0.3,,0.4\}$ together with the nominal \Bezier{} path $\rb$. The sensor begins its trajectory at $t=0.2$---its initial position indicated by the red dot in the second-from-left panel of Figure~\ref{fig:bezier_time_series}. The path is completed by $t=0.4$, after which the simulation continues until $t=1$.
\begin{figure}[ht]
\centering
\includegraphics[width=\textwidth]{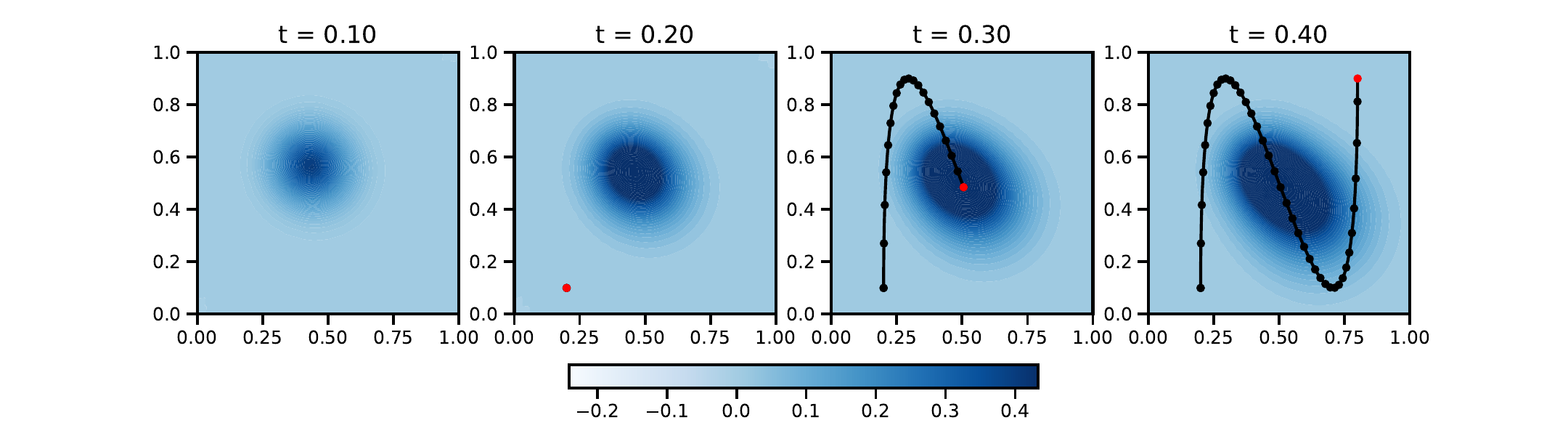}
\caption{Time-series of the solution $u_{\text{true}} = \mc{S}m_{\text{true}}$ and nominal Bézier path $\rb$ at four times. The red dot is the current position of the sensor and the black dots represent measurement locations.}
\label{fig:bezier_time_series}
\end{figure}

We solve the inverse problem with $40$ noisy measurements collected along the
nominal path $\rb$.  Figure~\ref{fig:bezier_samples} depicts $\mtrue$ (left) and
three posterior samples.  In the present experiment, the posterior samples capture
the dominant features of the ground-truth parameter, even with limited data
collected along a nominal path.

\begin{figure}[ht]
\centering
\includegraphics[width=\textwidth]{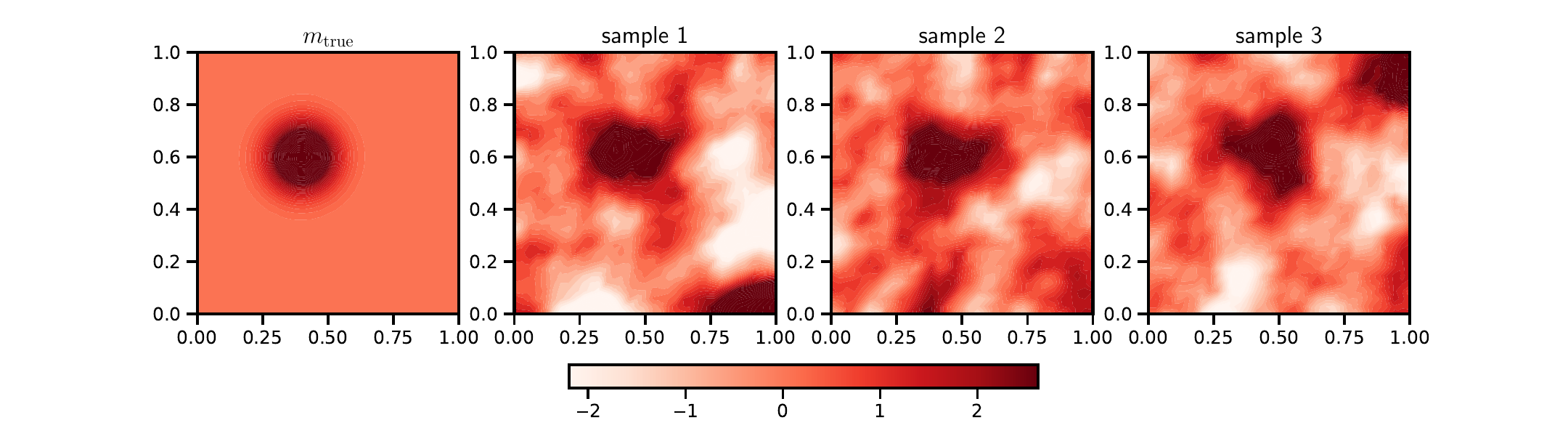}
\caption{The true parameter (left) and three posterior samples corresponding to the Bézier inverse problem.}
\label{fig:bezier_samples}
\end{figure}

\boldheading{Solving the path-OED problem} 
Next we solve the path-OED problem with \Bezier{} curves of degree $\Nb = 5$. Two
cases are considered: (i) the path starts at $(0.8, 0.2)$ and ends at $(0.2,
0.8)$; (ii) both the start and end positions are $(0.8, 0.2)$. Note that when
the endpoints of a $\Nb$-degree \Bezier{} curve are fixed, the design
vector $\xib$ is composed by the interior control points. So in our case, $\xib \in \R^{8}$. To contain the curves in $\Om$,
we require the control points to lie in $\Om$. As discussed in Subsection~\ref{sec:path_constraints}, this translates to
a linear constraint on the design vector.

After solving the path-OED problem, we perform several experiments to assess the performance of the resulting design. First, we overlay the optimal paths on the corresponding \emph{pointwise variance fields}; see~\cite{Bui-ThanhGhattasMartinEtAl13}. This field quantifies posterior variance in the inversion parameter over the spatial domain. Next, we analyze the posterior goal-densities, computed using~\eqref{eq:goal_mean_and_var}, and compare them to the true goal value $Z(\mtrue)$. Finally, to further assess optimality, we generate $1000$ uniformly random designs $\xib \in \R^{8}$, evaluate the associated criterion values $\mat\Psi(\xib)$, and compare their distribution to $Z(\mtrue)$. The results are shown in Figure~\ref{fig:loop_and_path}.

\begin{figure}[htbp]
\centering
\setlength{\tabcolsep}{6pt} 
\begin{tabular}{ccc}
\includegraphics[width=0.34\textwidth]{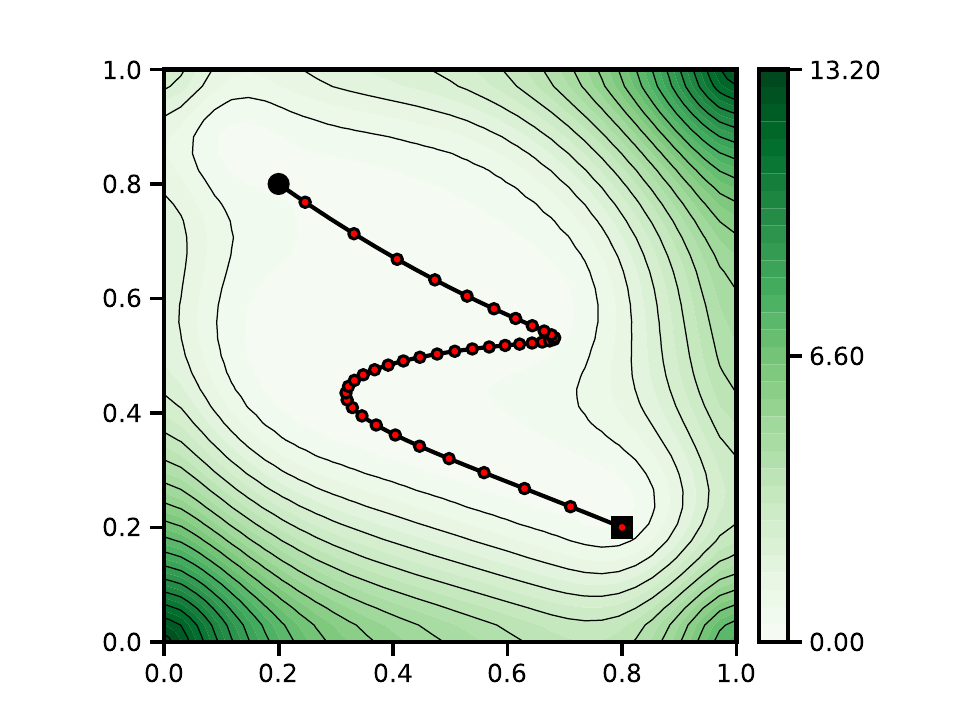} &
\includegraphics[width=0.26\textwidth]{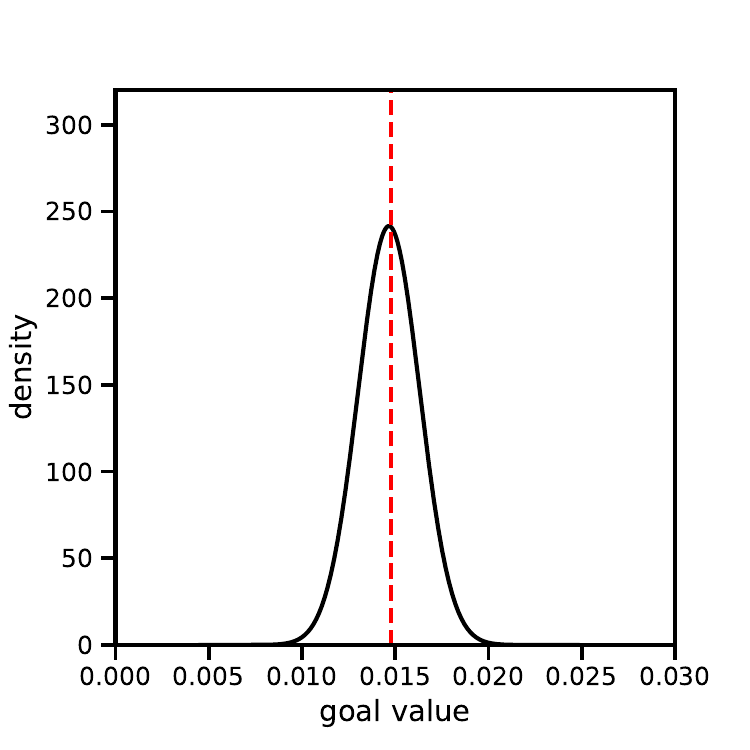} &
\includegraphics[width=0.27\textwidth]{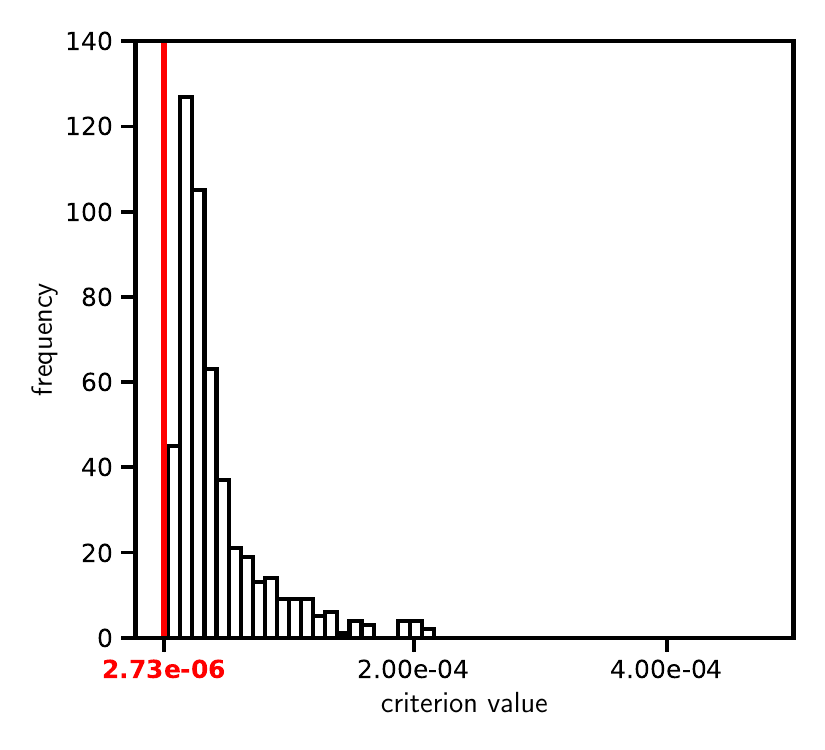} \\
\includegraphics[width=0.34\textwidth]{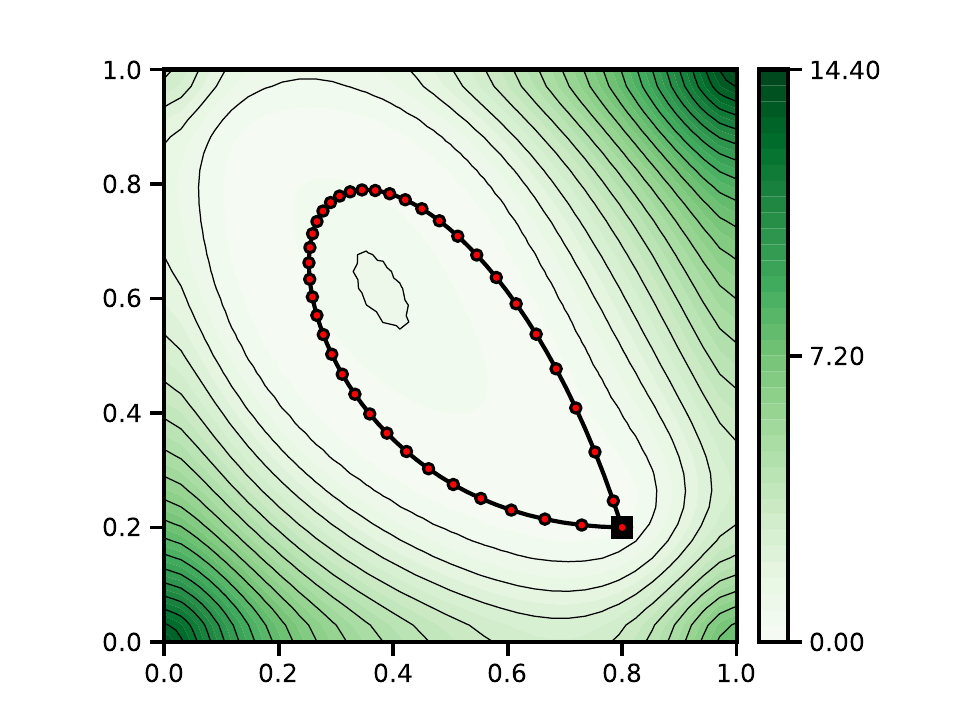} &
\includegraphics[width=0.26\textwidth]{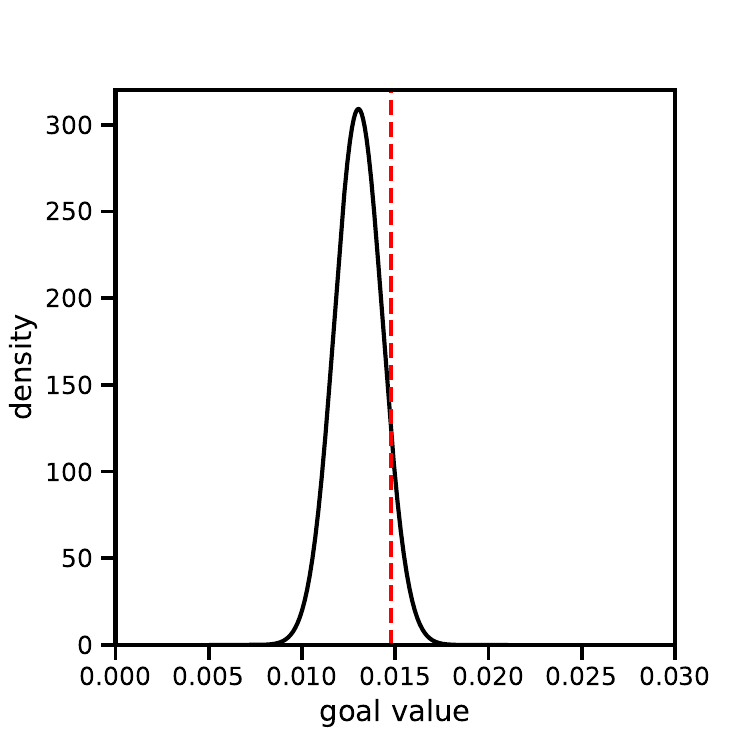} &
\includegraphics[width=0.27\textwidth]{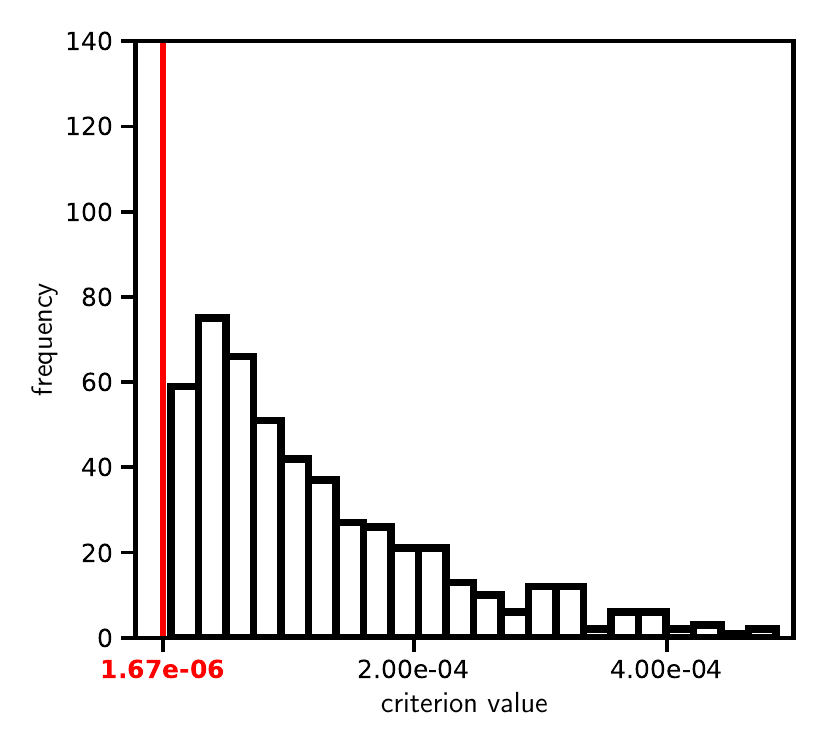}
\end{tabular}
\caption{Top: Results for a degree $\Nb=5$ \Bezier{}~curve where endpoints do not coincide. Bottom: Results for a degree $\Nb=5$ Bezier 
curve where endpoints do coincide. Left: Optimal path and pointwise parameter variance field. The square is the initial position of the path and circle is the terminal position. Middle: Posterior goal-density (true goal-value is the dotted vertical red line). Right: Optimal criterion value (vertical red line) and histogram of criterion values corresponding to uniformly random paths.}
\label{fig:loop_and_path}
\end{figure}

We first comment on the top row of Figure~\ref{fig:loop_and_path}; the case
where the initial and terminal positions of the optimal \Bezier{} path are
different. As seen in Figure~\ref{fig:loop_and_path}~(top-left), the optimal path describes a
Z-like trajectory. Where the sensor the sensor makes a sharp turn, a large amount of data is collected. This indicates that data in these regions are informative for goal inference. The posterior variance field plotted in the top-left of Figure~\ref{fig:loop_and_path} shows that the optimal design leads to preferential uncertainty reduction in the inversion parameter. That is, uncertainty is not reduced in every direction of $\M$, only in the important directions. 
The posterior goal-density in the top-middle of Figure~\ref{fig:loop_and_path} characterizes both the accuracy and certainty in goal inference. We note that the density is centered about the true goal-value and caution that the present behavior is
influenced by the choice of the true parameter and data noise.  
Lastly, we consider the histogram of criterion values in Figure~\ref{fig:loop_and_path}~(top-right).
We note that the optimal criterion value is lower when compared to values resulting from random designs. In fact, this difference is of
order $10^{2}$. This implies that a suboptimal choice of the sensor path can lead to
notably suboptimal reduction in the posterior uncertainty of $Z$.

We next repeat this study for the case where the initial and terminal positions coincide; see the bottom row of Figure~\ref{fig:loop_and_path}.
Here, we super-impose the computed optimal path
over the posterior variance field in
Figure~\ref{fig:loop_and_path}~(bottom-left).  The posterior goal-density in 
Figure~\ref{fig:loop_and_path}~(bottom-middle) has a smaller variance than the
corresponding result in the top panel, but is less-centered around the true
goal-valued. Again, the relative position of the posterior goal-density 
to the true value should be carefully interpreted. It is worth noting, 
however, 
that the predicted posterior goal-density includes the ground-truth parameter in 
its high-probability region. 
Similar to the results for the non-coincident endpoints, the optimal criterion
value is less than all the criterion values corresponding to randomly generated
designs. This is shown in Figure~\ref{fig:loop_and_path}~(bottom-right).
Interestingly, the criterion values corresponding to the random designs span a
wider range of values when compared to the top row results.

To further quantify the effectiveness of the optimal paths shown in
Figure~\ref{fig:loop_and_path}, we next examine the \emph{coefficient of
variation} for the goal functional. Given a probability law $\mu$, the coefficient of variation, denoted by $\cv(\mu)$,
is defined as the ratio of the standard deviation of $Z(m)$ to the
absolute value of its mean. So, $\cv(\mu)$ is a unitless measure of uncertainty that we can calculate with the formulas in~\eqref{eq:goal_mean_and_var}. Computing this value for the prior measure yields $C(\mu_{\pr}) = 0.248$. For the posterior measure
corresponding to the results in the top row of Figure~\ref{fig:loop_and_path},
we have $\cv(\mupoy) = 0.119$, and for the bottom row, $\cv(\mupoy) =
0.109$. This results in a $\approx 52\%$ reduction the in $\cv$ when the initial and terminal endpoints of the path coincide, and when
the endpoints do coincided, the reduction is $\approx 56\%$. Thus, the optimal
design corresponding to the bottom row of Figure~\ref{fig:loop_and_path} yields
slightly greater reduction in goal uncertainty than the other design.

\boldheading{A study with an obscured region and non-coincident endpoints} 
In this study we specify that the \Bezier{} paths have initial position
$(0.8, 0.2)$ and terminal position $(0.2, 0.8)$, and demonstrate our ability to
solve the path-OED problem with an obscured region $\Ovoid$; see Subsection~\ref{sec:path_constraints}. 
Interpreting the
model~\eqref{eq:conv_diff} as representing the concentration of some
contaminant, $\Ovoid$ may describe a region that is obscured or where the sensor
is unable to collect data effectively due to environmental factors. We take $\Ovoid$ to be a
disk of radius $\Rvoid = 0.16$ centered about $\xvoid = (0.5, 0.5)$. The radial basis function
$\void$ in~\eqref{eq:void}, used to approximate an indicator function for
$\Ovoid$, is parameterized with $\beta = 0.01$.

\begin{figure}[htbp]
  \centering

  \begin{minipage}[t]{0.245\textwidth}
    \centering $N_{b} = 4$
  \end{minipage}
  \hfill
  \begin{minipage}[t]{0.245\textwidth}
    \centering $N_{b} = 5$
  \end{minipage}
  \hfill
  \begin{minipage}[t]{0.245\textwidth}
    \centering $N_{b} = 8$
  \end{minipage}
  \hfill
  \begin{minipage}[t]{0.245\textwidth}
    \centering $N_{b} = 20$
  \end{minipage}

  \vspace{0.5ex}

  \includegraphics[width=0.23\textwidth]{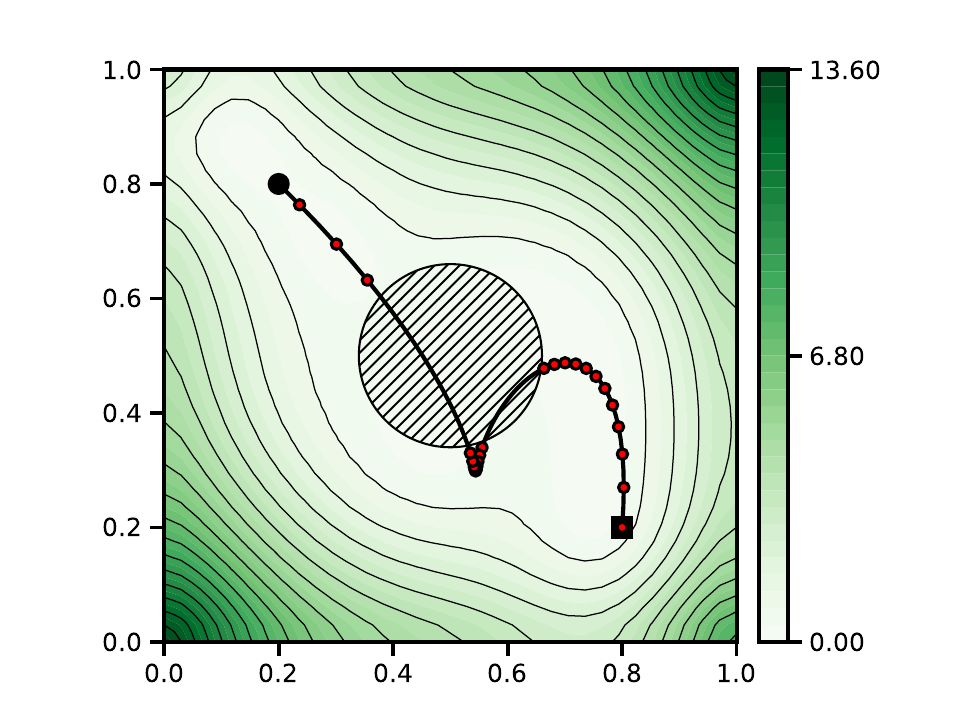}
  \hfill
  \includegraphics[width=0.23\textwidth]{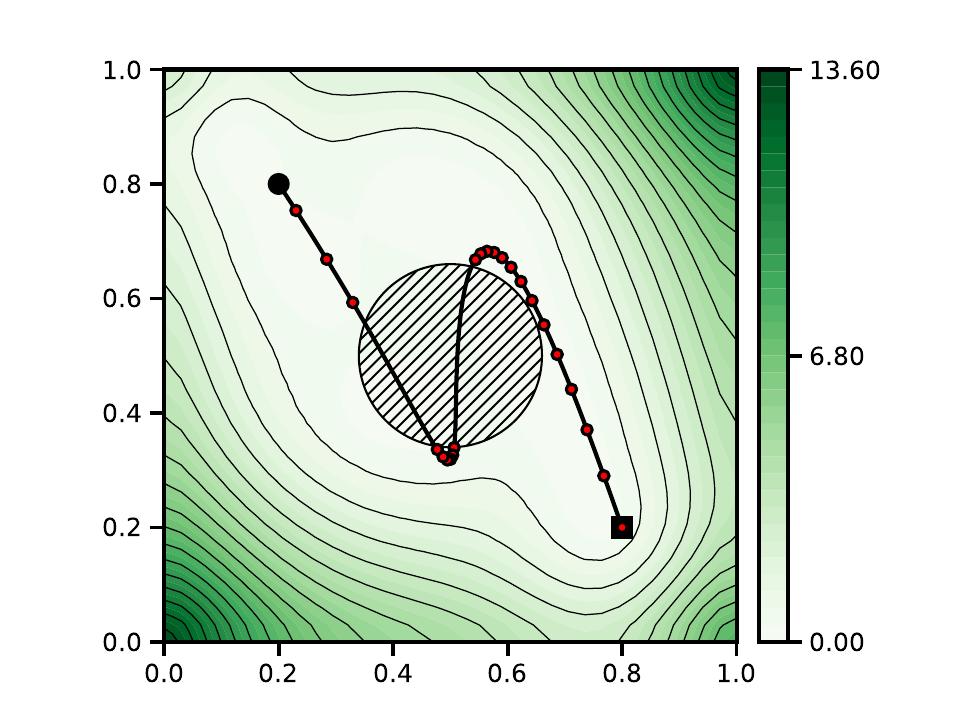}
  \hfill
  \includegraphics[width=0.23\textwidth]{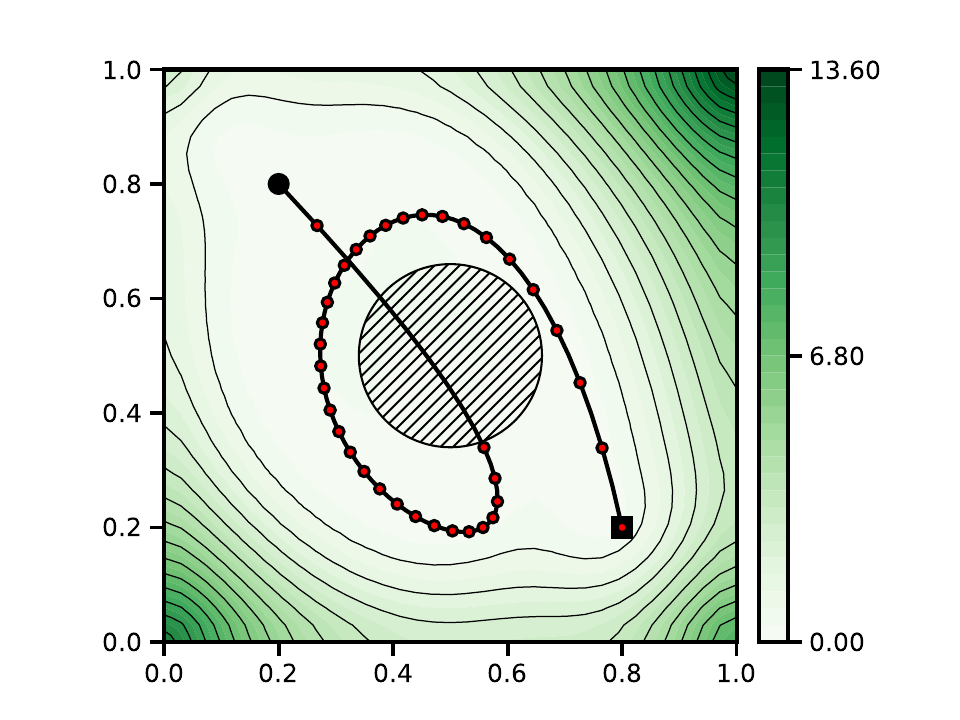}
  \hfill
  \includegraphics[width=0.23\textwidth]{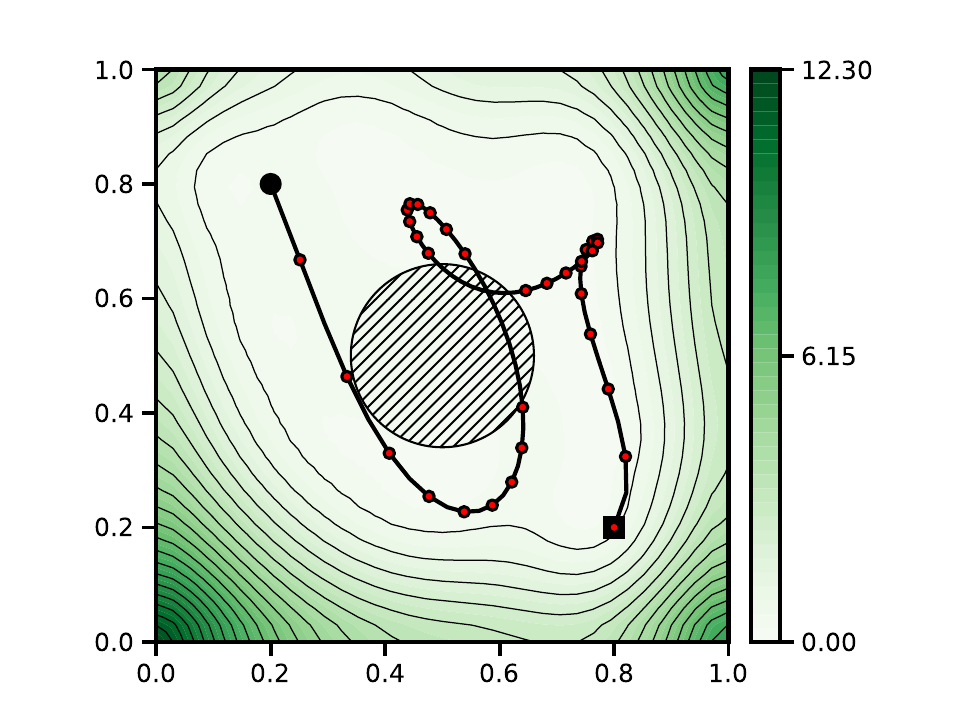}

  \vspace{1ex}

  \includegraphics[width=0.23\textwidth]{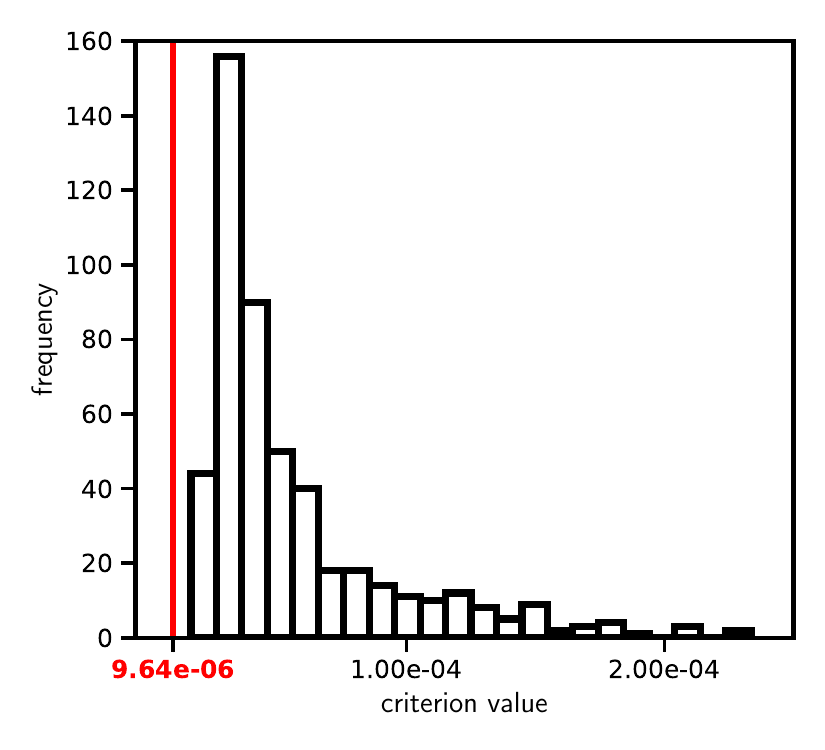}
  \hfill
  \includegraphics[width=0.23\textwidth]{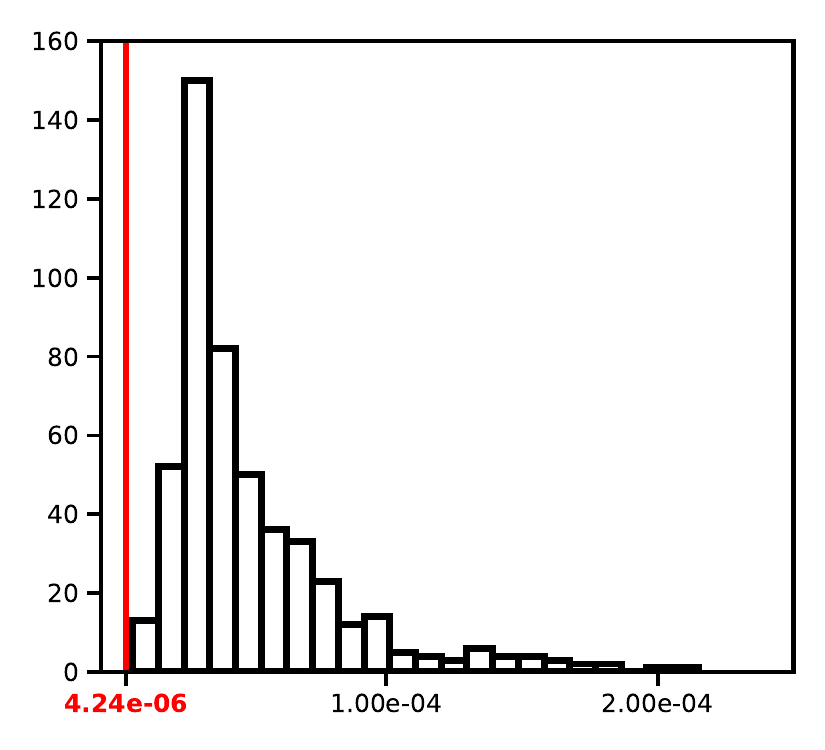}
  \hfill
  \includegraphics[width=0.23\textwidth]{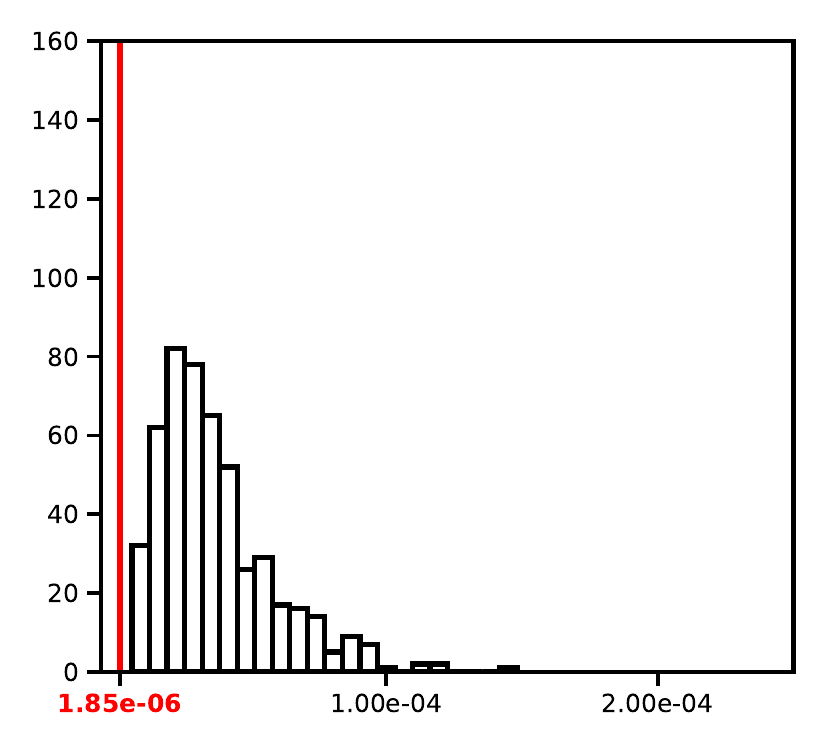}
  \hfill
  \includegraphics[width=0.23\textwidth]{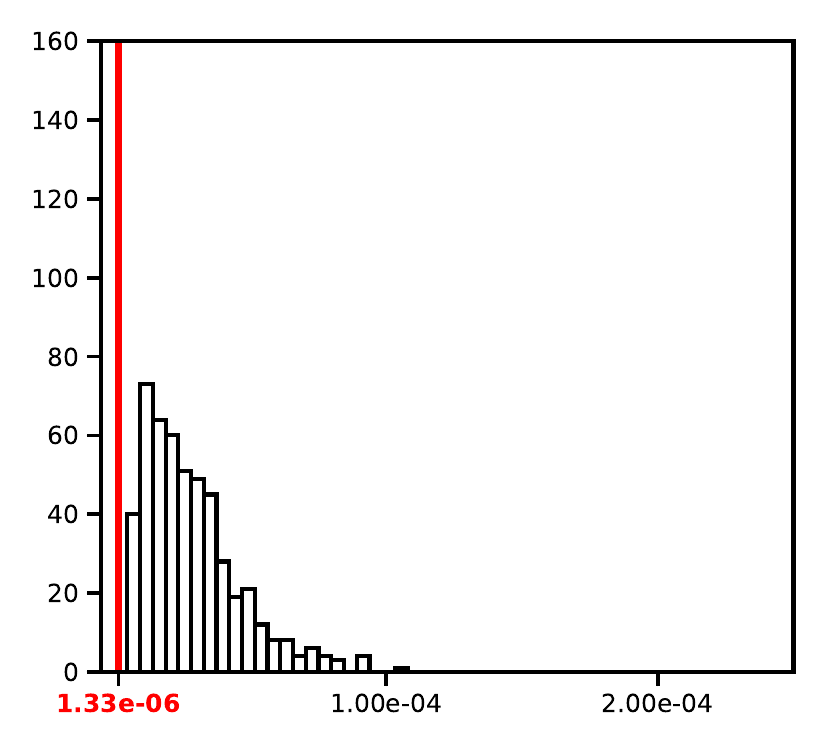}

  \caption{Top: optimal paths for $N_b \in \{4, 5, 8, 20\}$ and corresponding posterior variance fields. The square is the initial position of the path and circle is the terminal position. The obscured region is represented by the shaded disk. Bottom: Histogram of criterion values for uniformly randomly generated designs. The optimal criterion value is the red vertical line.}
  \label{fig:bezier_ends_histogram}
\end{figure}

In this experiment we llustrate how the degree of the \Bezier{} curve
impacts the results. To this end, we solve the path-OED problem with curves of
degree $\Nb \in \{4,5,8,20\}$; see Figure~\ref{fig:bezier_ends_histogram}.
A key observation from the top row of Figure~\ref{fig:bezier_ends_histogram}
is that the optimal paths enter the obscured region. This means that data is willingly sacrificed to explore more informative ares of
$\Om$. This claim is exemplified with the $N_b = 8$ case. Here, the sensor
completes most of a circumnavigation around the disk, then translates almost
directly through the center to reach the terminal position.
We also comment on the histograms in the bottom row of
Figure~\ref{fig:bezier_ends_histogram}. The optimal criterion value for the $\Nb
= 4$ case is $9.64\times 10^{-6}$, and for $\Nb = 20$, this value is $1.33\times 10^{-6}$. Thus,
increasing the design size by a factor of $5$ results in less than an order $10$
magnitude of improvement. This indicates that a relatively small degree number
can be sufficient for informative designs. In all cases, the optimal criterion value
is less than the criterion value corresponding to randomly generated designs.
However, this discrepancy is greatest for a low degree number. This implies that
path-OED is more effective when the number of control points is relatively
small.

Now we revisit the earlier observation that optimal curves enter the obscured
region. To provide further insight, we specify two intuitive paths that
navigate opposite sides of $\Ovoid$, instead of passing through. We compare
the effectiveness of the optimal and selected paths by examining their respective
posterior goal-densities. Figure~\ref{fig:bezier_path_vs_dens} shows the optimal and
specified paths (left) for $N_b = 5$ and corresponding posterior goal-densities (right). 
These results support our observation about quality over
quantity of data. Although the selected paths navigate different sides of the
obscured region, the widths of the corresponding posterior goal-densities are
comparable. A sensor that travels along the optimal path looses many
measurements by navigating through the obscured region, but is able to
acquire more informative data. This observation is corroborated by the corresponding posterior goal-density, whose variance is substantially smaller when compared to the prescribed paths; see Figure~\ref{fig:bezier_path_vs_dens} (right).

\begin{figure}[ht]
\centering
\includegraphics[height=5cm]{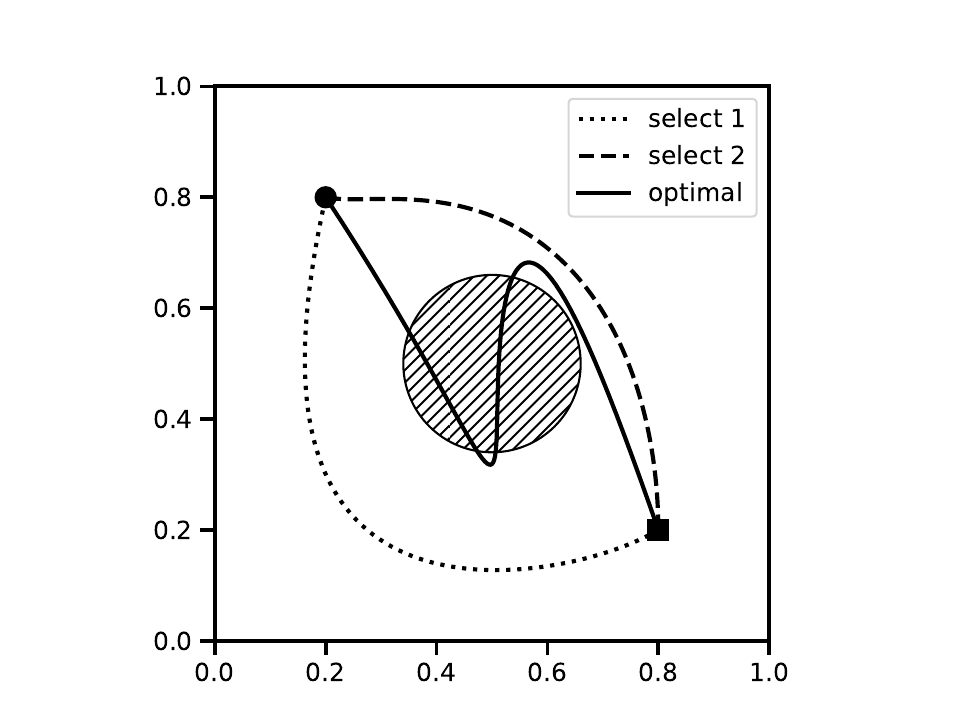}
\includegraphics[height=5cm]{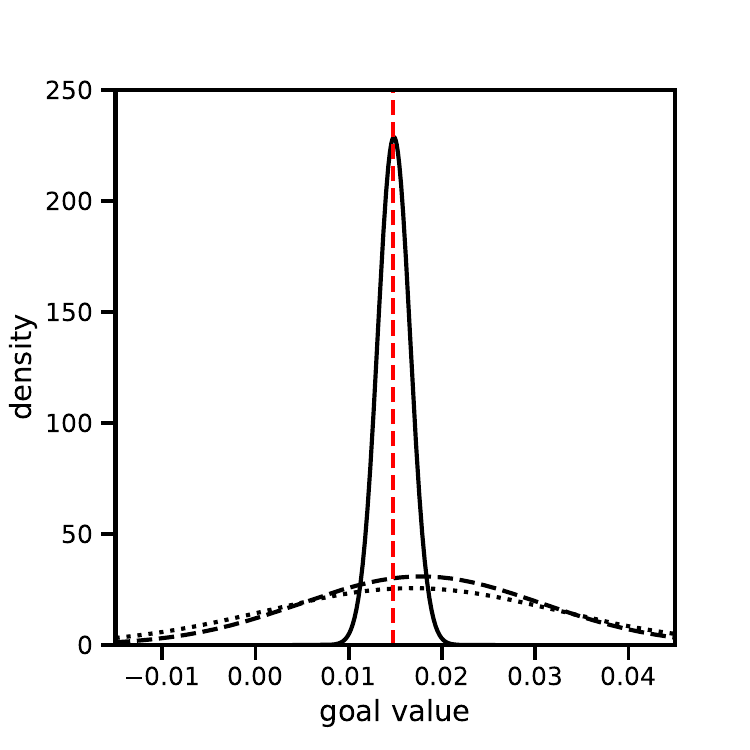}
\caption{Left: An optimal \Bezier{} path and two selected paths of degree $N_{b} = 5$. Right: Posterior goal-densities corresponding to the optimal and selected paths. The red dotted line is the true goal-value.}
\label{fig:bezier_path_vs_dens}
\end{figure}

\boldheading{A study with an obscured region and coincident endpoints} 
We repeat the previous experiment except we let the initial and terminal position
of the sensor be $(0.8, 0.2)$, resulting in closed paths. 
Figure~\ref{fig:bezier_loop_histogram} shows optimal \Bezier{} paths of degree
$N_b \in \{4, 5, 8, 20\}$, posterior variance fields, and histograms of
criterion values generated from evaluating the criterion at uniformly random designs.
For $N_b \in \{4, 5, 8\}$, we see that the paths loop around the obscured
region. When $N_b = 20$, this pattern is broken and the sensor performs a less
intuitive maneuver. In this case, it is valuable to sacrifice data to obtain
more informative measurements. Turning to the bottom row of
Figure~\ref{fig:bezier_loop_histogram}, we find that the degree $N_b \in \{4,
5\}$ optimal curves outperform the analogous results shown in
Figure~\ref{fig:bezier_ends_histogram} in terms of criterion value. However, as
we increase the size of the design, this pattern disappears and performance is
comparable.

\begin{figure}[htbp]
  \centering

  \begin{minipage}[t]{0.245\textwidth}
    \centering $N_{b} = 4$
  \end{minipage}
  \hfill
  \begin{minipage}[t]{0.245\textwidth}
    \centering $N_{b} = 5$
  \end{minipage}
  \hfill
  \begin{minipage}[t]{0.245\textwidth}
    \centering $N_{b} = 8$
  \end{minipage}
  \hfill
  \begin{minipage}[t]{0.245\textwidth}
    \centering $N_{b} = 20$
  \end{minipage}

  \vspace{0.5ex}

  \includegraphics[width=0.23\textwidth]{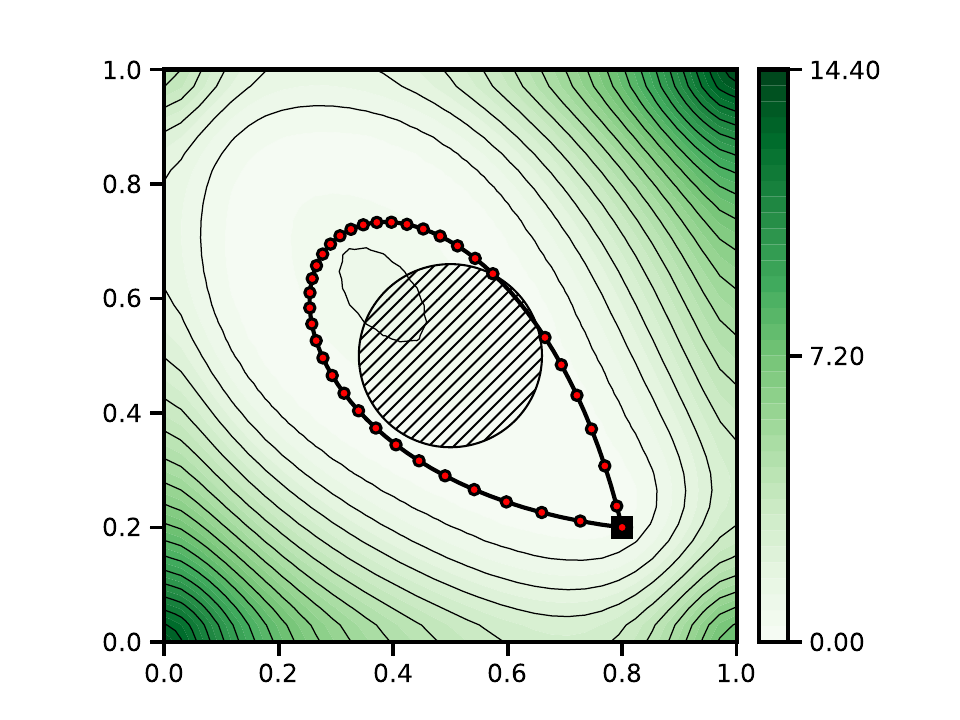}
  \hfill
  \includegraphics[width=0.23\textwidth]{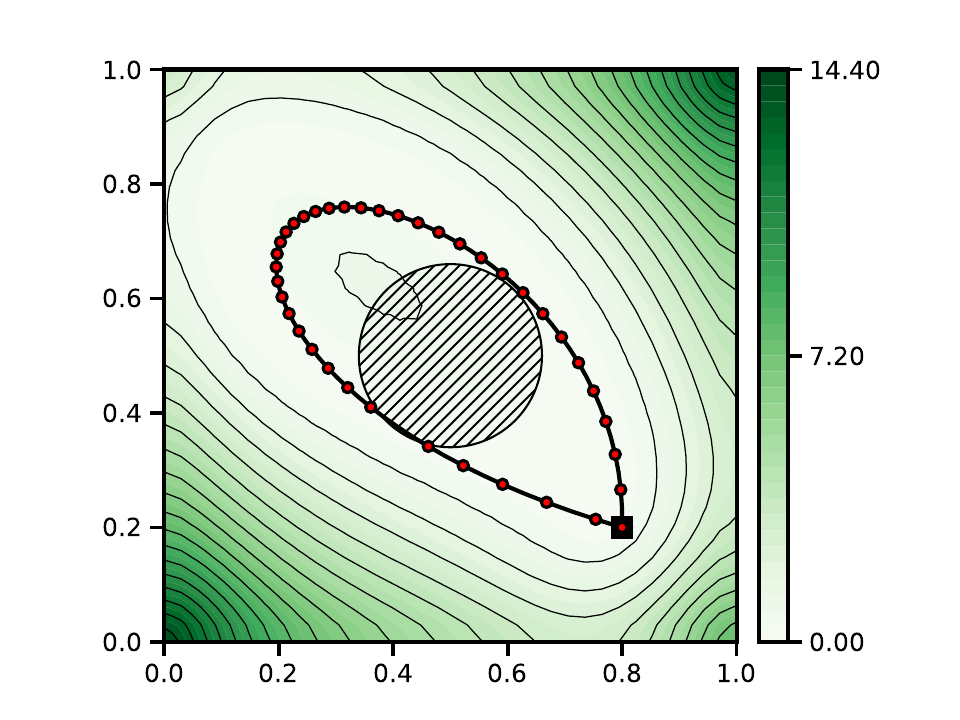}
  \hfill
  \includegraphics[width=0.23\textwidth]{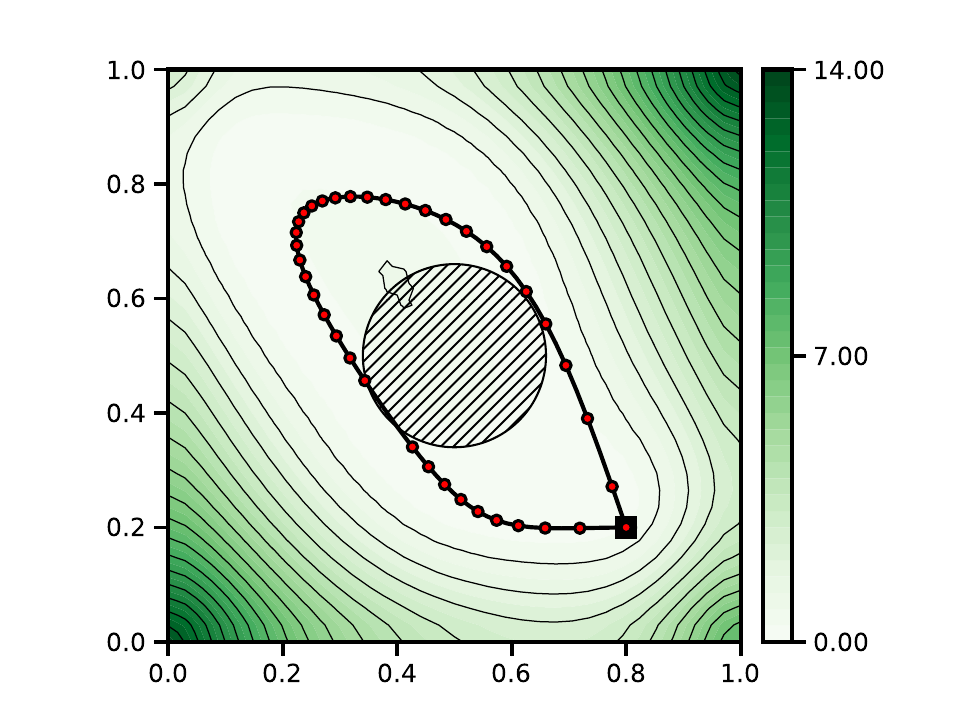}
  \hfill
  \includegraphics[width=0.23\textwidth]{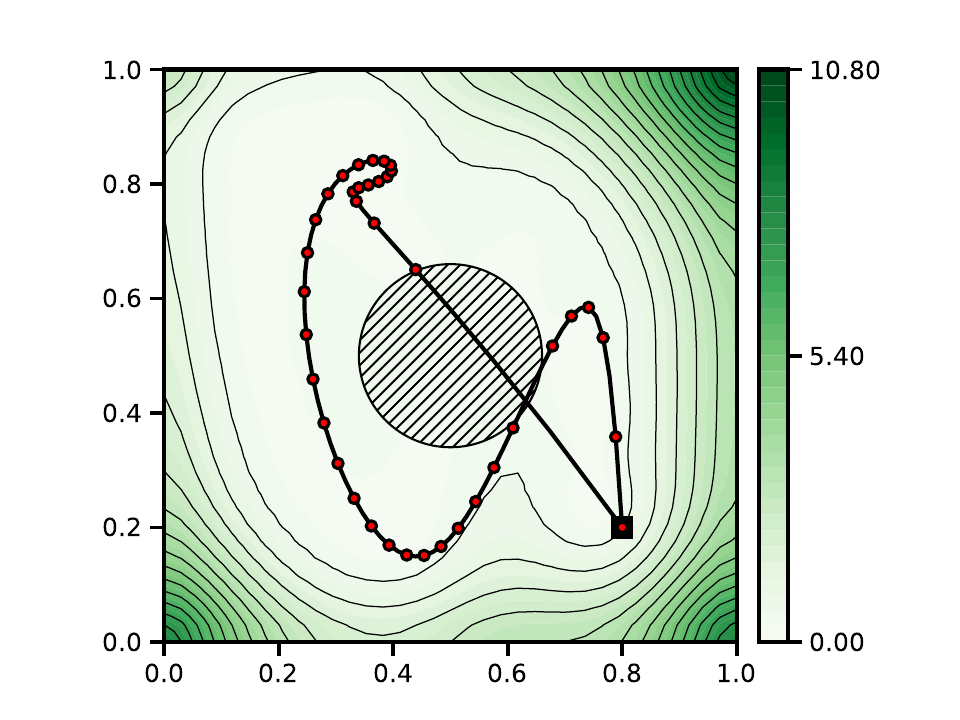}

  \vspace{1ex}

  \includegraphics[width=0.23\textwidth]{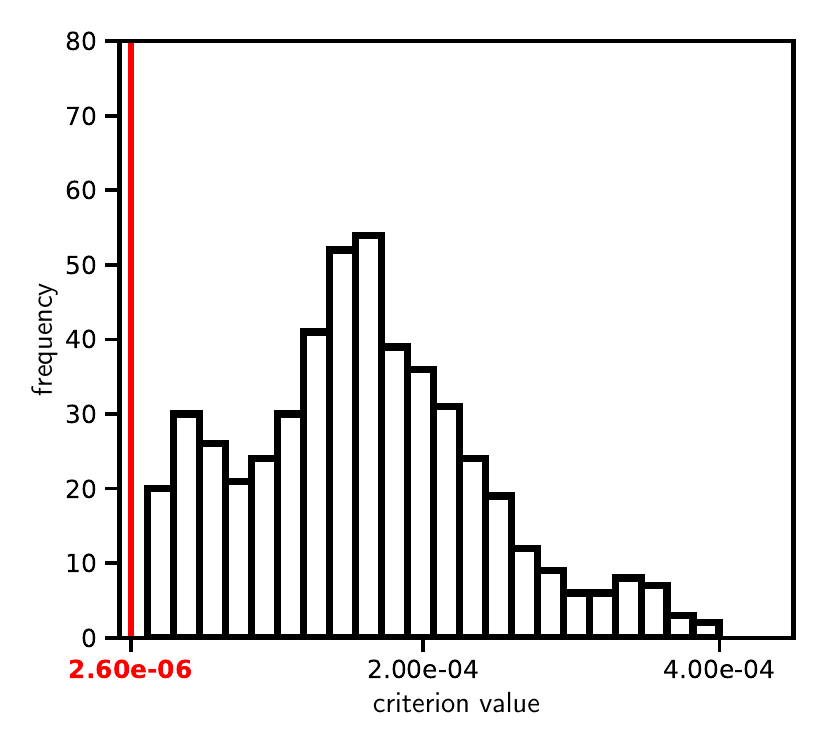}
  \hfill
  \includegraphics[width=0.23\textwidth]{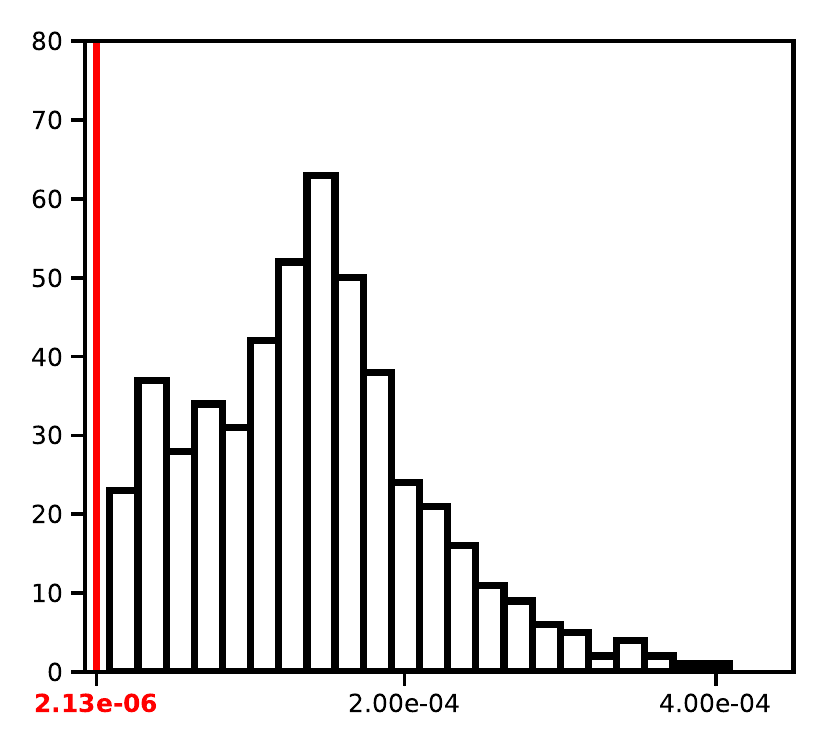}
  \hfill
  \includegraphics[width=0.23\textwidth]{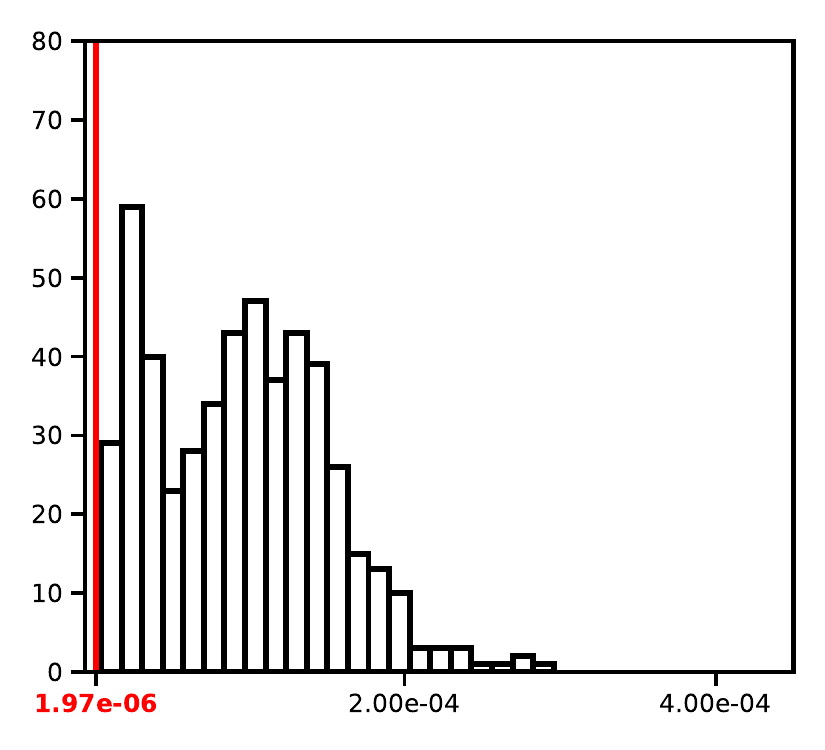}
  \hfill
  \includegraphics[width=0.23\textwidth]{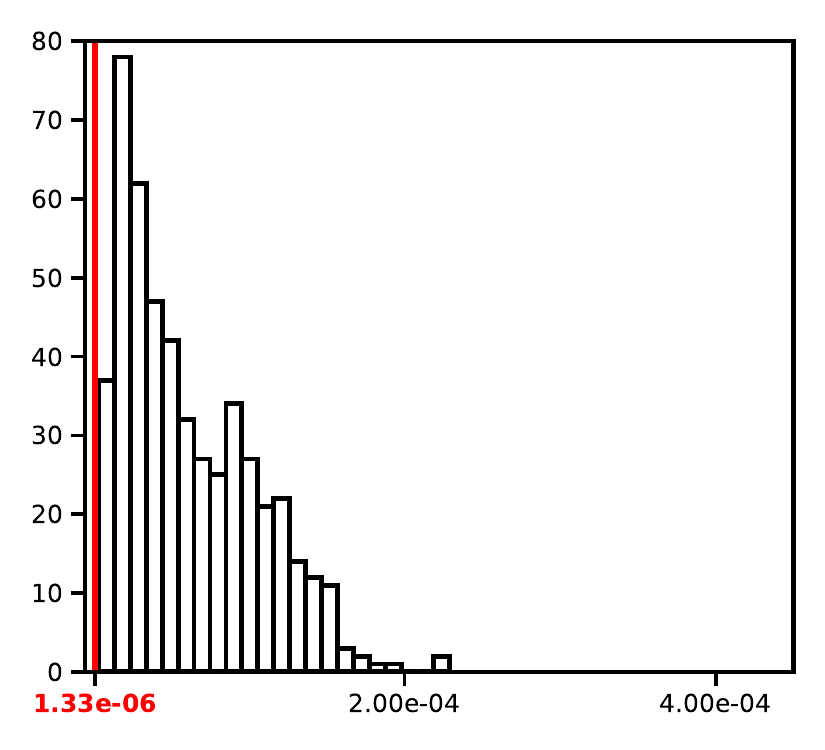}

  \caption{Top: optimal paths for $N_b \in \{4, 5, 8, 20\}$ and corresponding posterior variance fields. The square is the initial/terminal position and the obscured region is the shaded disk. Bottom: Histogram of criterion values for uniformly randomly generated designs. The optimal criterion value is the red vertical line.}
  \label{fig:bezier_loop_histogram}
\end{figure}

\subsection{An experiment with Fourier paths}
\label{sec:fourier_results}
Here, we solve the path-OED problem with Fourier paths. In this case, we require
the sensor's initial and terminal positions to coincide, forming a loop.
However, we do not specify the starting location---this is naturally determined by solving
the path-OED problem.  In what follows, after specifying the setup of the
experiment, we solve the path-OED problem for various Fourier mode numbers $\Nf$ and discuss
the results.

\boldheading{Problem setup}
In the present experiment, the boundary conditions in~\eqref{eq:conv_diff} are
selected such that $E_0 = \emptyset$ and $E_n = \partial \Om$. That is, 
we consider only homogenous Neumann boundary conditions. The global time
interval is $T = [0,2]$ and the inversion interval is $\Ty = [0.25, 1.0]$. The
sensor makes a measurement at every third time instance in the
discretization of $\Ty$, resulting in $50$ measurements. We let the velocity
field and amplitude function in~\eqref{eq:conv_diff} to be
$$
\vec F(\vx) \defeq 
\begin{bmatrix}
2(2x_2 - 1)\big(1 - (2x_1 - 1)^2\big)\\
-2(2x_1 - 1)\big(1 - (2x_2 - 1)^2\big)
\end{bmatrix}
\quad \text{and} \quad a(t) \defeq \frac{1.05}{2} \left( 1 - \frac{2}{\pi} \arctan(8t - 6) \right).
$$
This defines a recirculating velocity field.
As the concentration diffuses, it swirls about the point $(0.5, 0.5)$ while
being damped by $a(t)$. We depict both the velocity field and amplitude function
in Figure~\ref{fig:fourier_vel_amp}.

\begin{figure}[ht]
\centering
\includegraphics[height=0.25\textwidth]{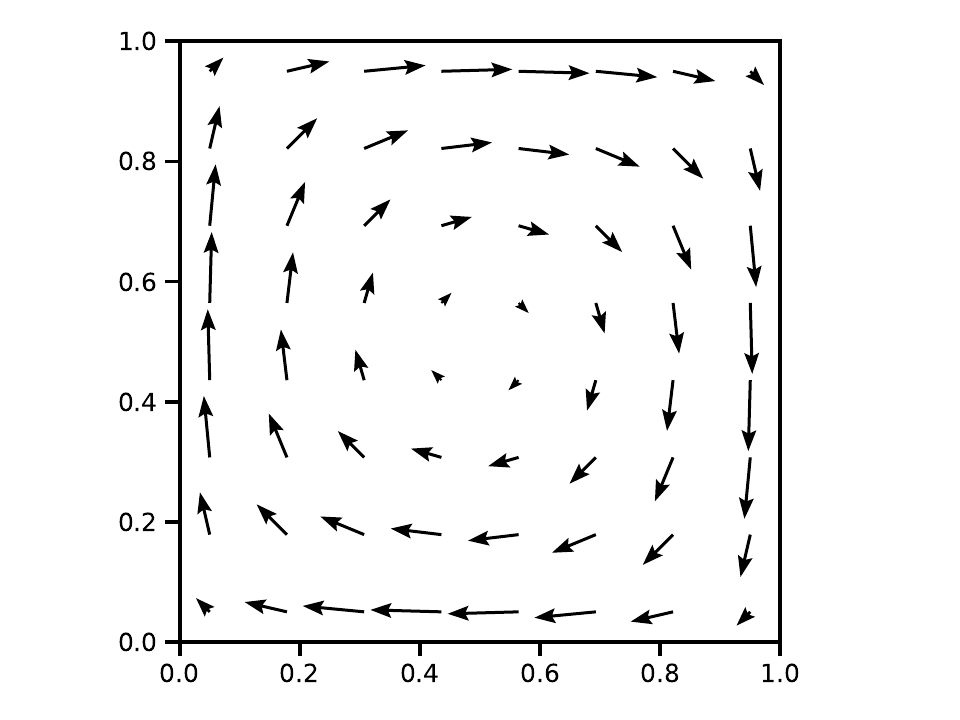}
\includegraphics[height=0.25\textwidth]{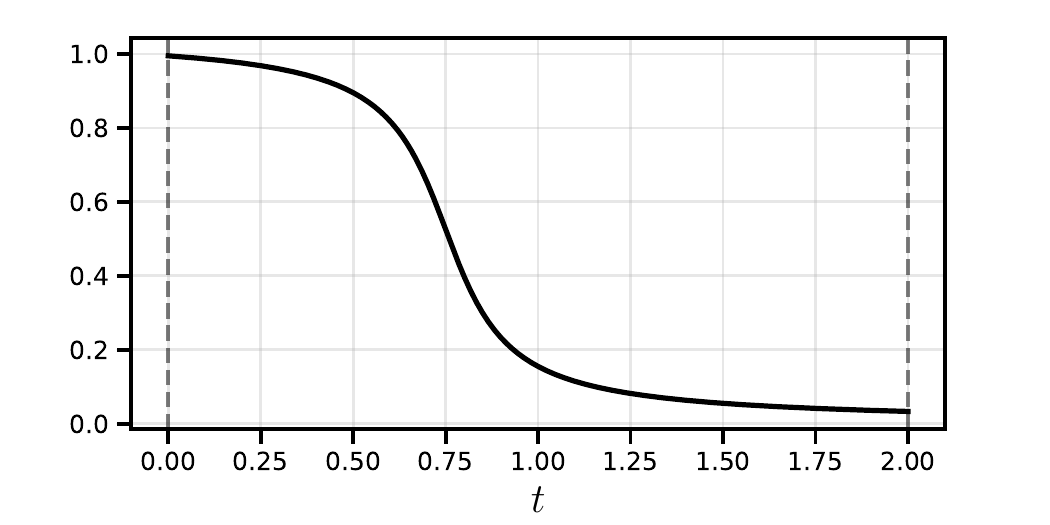}
\caption{The velocity field $\vec F(\vx)$ (left) and amplitude $a(t)$ (right) used for the Fourier experiment.}
\label{fig:fourier_vel_amp}
\end{figure}

As for the spatiotemporal indicator function in~\eqref{eq:space_time_ind} used to form $\mat \Psi$, we define $v$ with
$$
\Om_{v} = (0.2, 0.8)^2 \quad \text{and} \quad T_{v} = [1.5, 2].
$$ 

\boldheading{Solving the path-OED problem}
We consider Fourier paths with $\Nf \in \{1, 3, 5, 10\}$ number of modes.  For
each $\Nf$, we solve the path-OED problem by minimizing $\mat \Psi$ to yield the
optimal design. The regularization term $R(\xif; \gam)$, presented
in~\eqref{eq:reg_acc}, is appended to $\mat \Psi$ to limit acceleration or sharp
turns in the optimal paths. The choice of regularization parameter $\gam$ is
depends upon the mode number. We select $(\Nf, \gam)$ pairs as follows:
$$
(\Nf, \gam) \in \{(1, 3\times 10^{-7}), \ (3, 5 \times 10^{-9}), \ (5, 7\times 10^{-10}), 
\ (10, 9\times 10^{-11})\}.
$$

\begin{figure}[htbp]
  \centering

  \begin{minipage}[t]{0.23\textwidth}
    \centering $\Nf = 1$
  \end{minipage}
  \hfill
  \begin{minipage}[t]{0.23\textwidth}
    \centering $\Nf = 3$
  \end{minipage}
  \hfill
  \begin{minipage}[t]{0.23\textwidth}
    \centering $\Nf = 5$
  \end{minipage}
  \hfill
  \begin{minipage}[t]{0.23\textwidth}
    \centering $\Nf = 10$
  \end{minipage}

  \vspace{0.5ex}

  \includegraphics[width=0.245\textwidth]{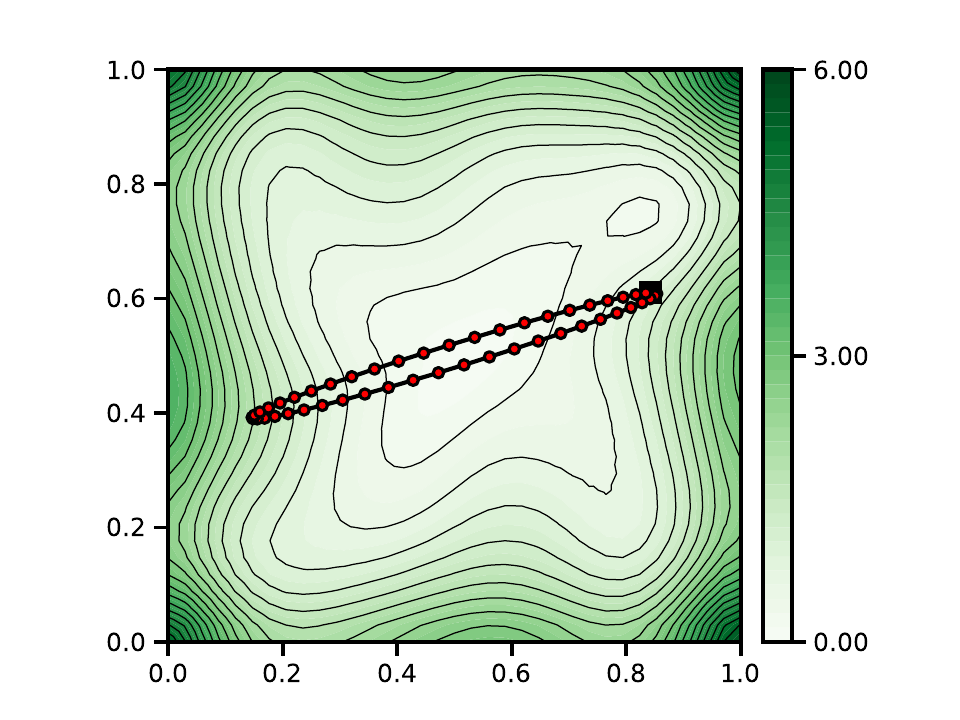}
  \hfill
  \includegraphics[width=0.245\textwidth]{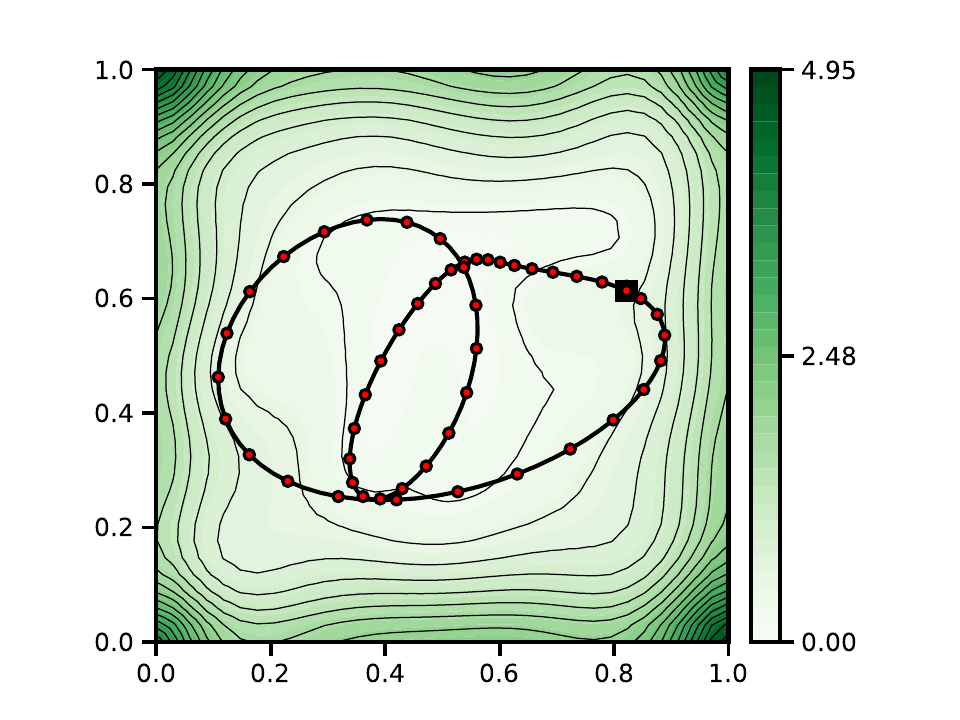}
  \hfill
  \includegraphics[width=0.245\textwidth]{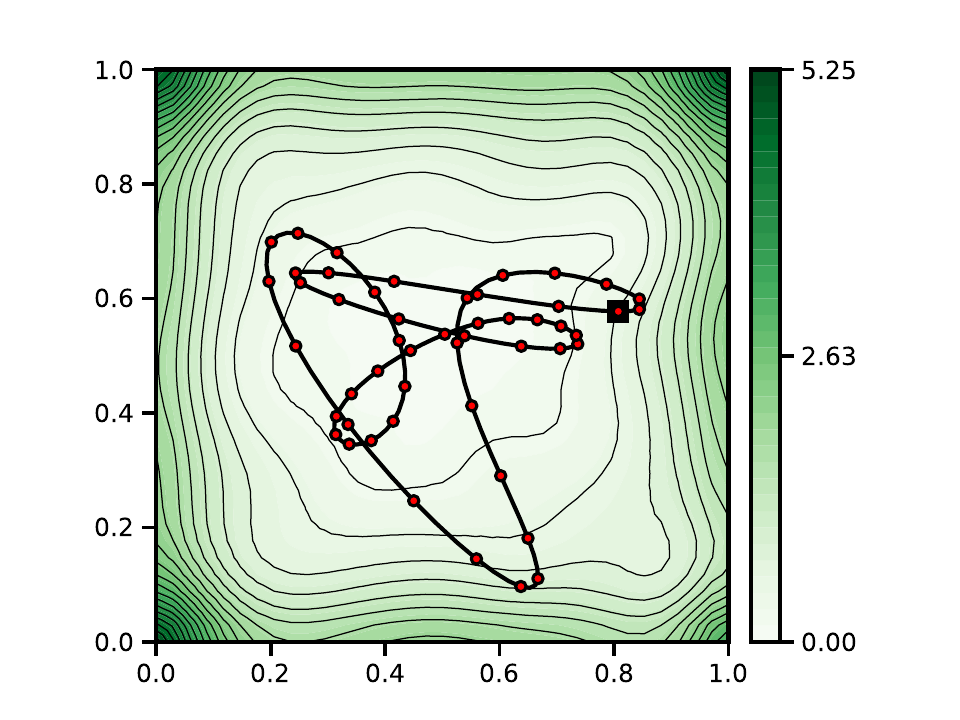}
  \hfill
  \includegraphics[width=0.245\textwidth]{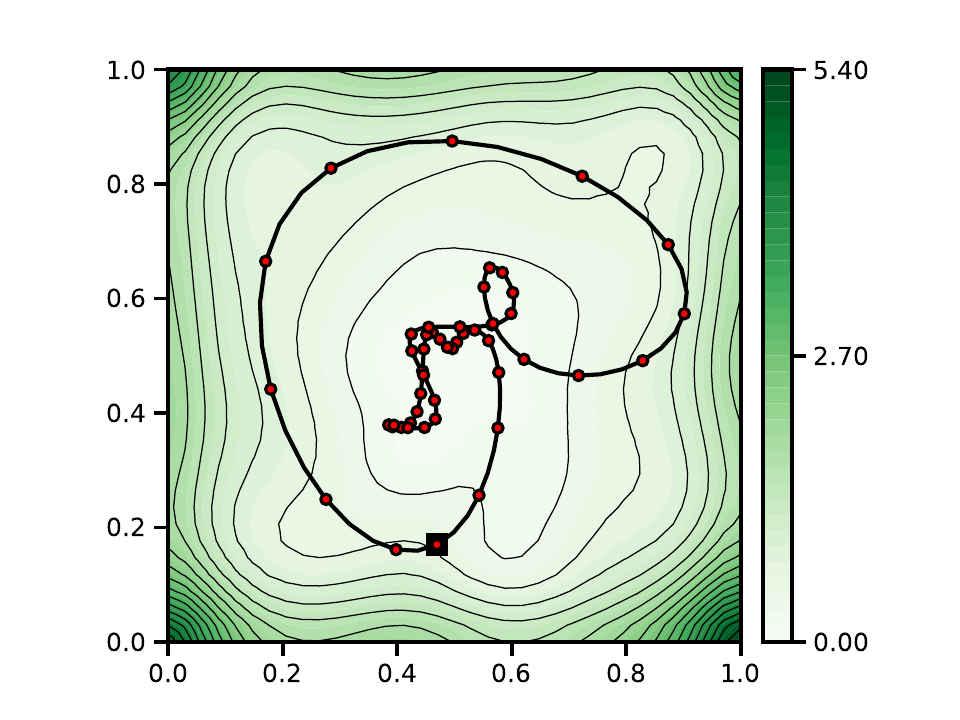}

  \vspace{1ex}

  \includegraphics[width=0.23\textwidth]{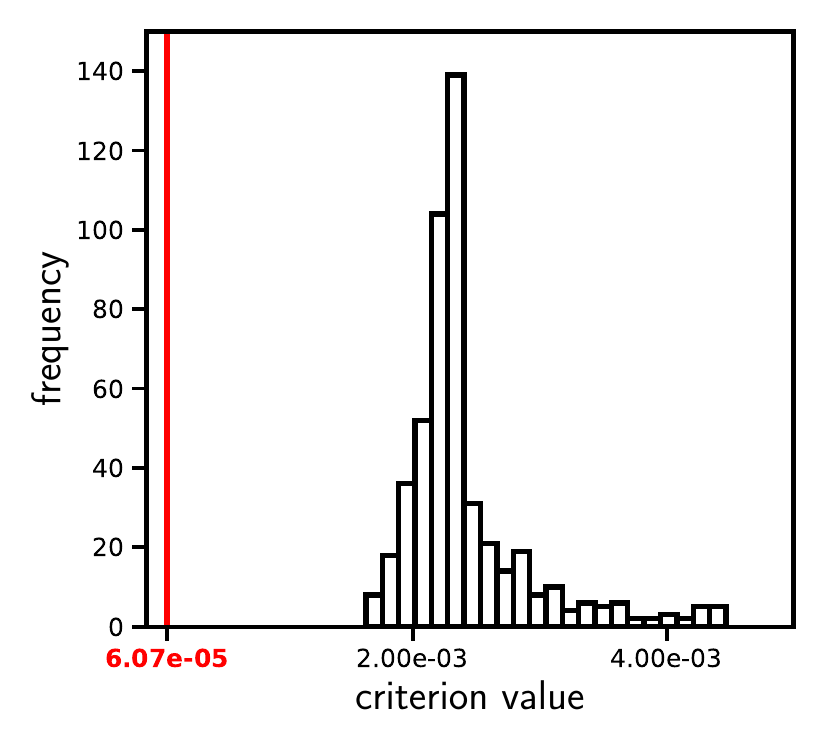}
  \hfill
  \includegraphics[width=0.23\textwidth]{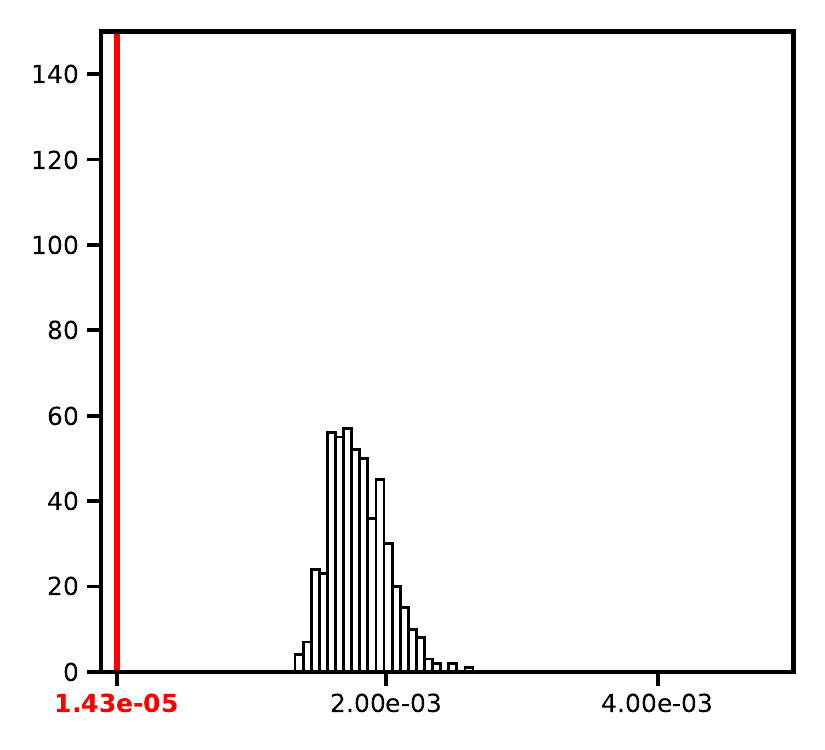}
  \hfill
  \includegraphics[width=0.23\textwidth]{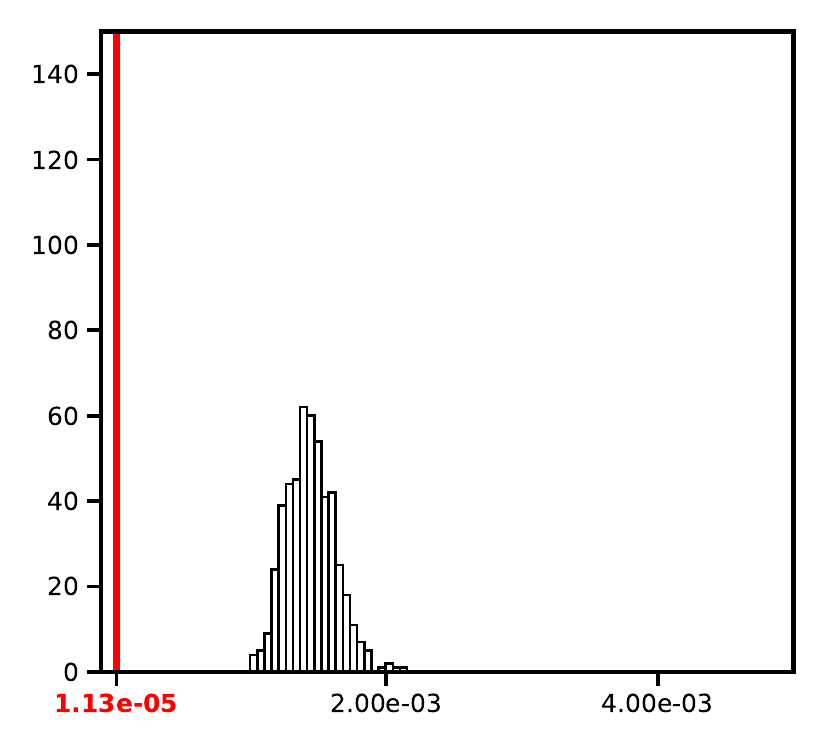}
  \hfill
  \includegraphics[width=0.23\textwidth]{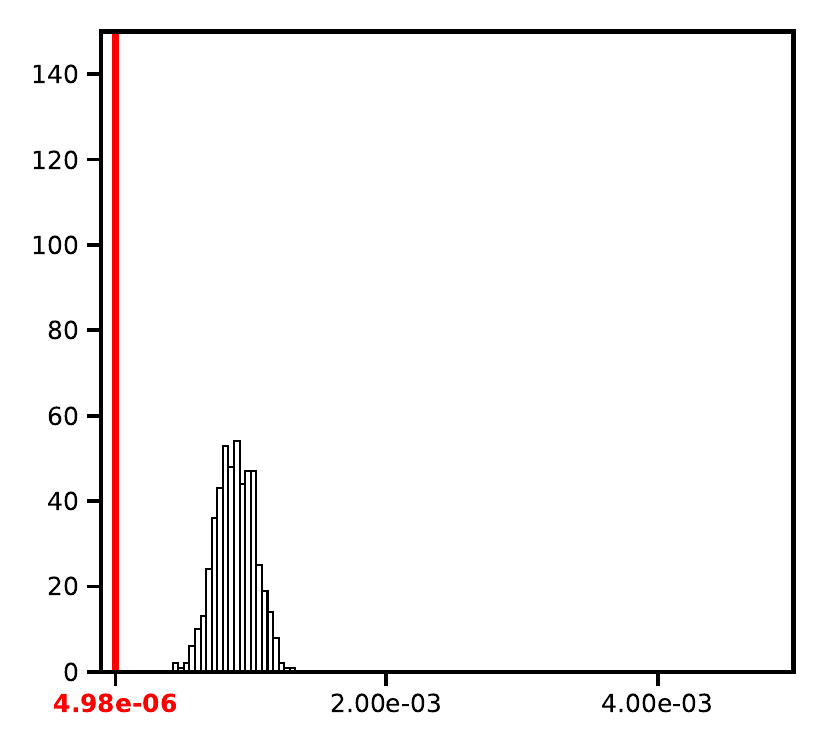}

  \caption{Top: optimal paths for $\Nf \in \{1, 3, 5, 10\}$ and corresponding posterior variance fields. The square is the initial/terminal position of the sensor. Bottom: Histogram of criterion values for uniformly randomly generated designs. The optimal criterion value is the red vertical line.}
  \label{fig:fourier_histogram}
\end{figure}

For each considered $\Nf$, we study the optimal paths, posterior variance
fields, and histograms of criterion values generated by evaluating $\mat
\Psi(\xif)$ for $1000$ uniformly random $\xif$. These results are presented in
Figure~\ref{fig:fourier_histogram}. As seen in the top row of the figure, the
optimal paths become more complex as we increase $\Nf$. 
For the $\Nf = 1$ case, Fourier
paths are ellipses. Hence, the top-left of Figure~\ref{fig:fourier_histogram}
shows the optimal ellipse. As with the \Bezier{} experiments in
Subsection~\ref{sec:bezier_results}, the optimal paths preferentially reduce
uncertainty in regions of $\Om$ that minimize variance in the goal-functional.

Now we examine the criterion value histograms on the bottom row of
Figure~\ref{fig:fourier_histogram}. Optimal paths are consistently orders of
magnitude better than randomly selected paths. This is indicated by the distance
from the red vertical line (optimal criterion value) to the histograms in
Figure~\ref{fig:fourier_histogram}. Increasing the number of the Fourier
modes
substantially reduces the optimal criterion value. This observation is
supported by the factor $10$ improvement on optimality achieved from $\Nf = 5$
to $\Nf = 10$. A significant observation is the large gap between the optimal
criterion values and the values corresponding to randomly generated designs.

\section{Conclusion}
In this work we developed a mathematical and computational framework for optimal
experimental design in infinite-dimensional Bayesian linear inverse problems
with mobile sensors, which we termed \emph{path-OED}. Our approach parameterizes
sensor trajectories with any sufficiently regular family of curves.  We define
the observation operator by computing local averages of the PDE solution about
measurement points on the sensor path.  This results in a bounded linear
observation operator, and enables us to deploy standard optimality criteria in
the path-OED setting.  We focused on the Bayesian c-optimality criterion and derived its
derivatives. This was facilitated by a more general result which is useful
beyond the path-OED setting: a formula for the derivative of a parameterized
covariance-like operator. We also showed how practical constraints can be
incorporated into the formulation. Proceeding from the function-space problem to
a finite-dimensional representation, we introduced a weighted time-space inner
product and derived matrix-free actions of relevant operators and their
adjoints. Numerical experiments showed that path-OED reduces prediction variance
more effectively than random or user-prescribed paths.

The framework discussed in this article has several limitations and exposes
directions for future research. For instance, it would be useful to extend
path-OED to other criteria such as the A- and D-optimality criteria. The
observation operator as formulated in the present work paves the way to
formulating these criteria. However, formulas for their gradients and efficient
computational methods must be derived.
Also, the proposed approach assumes that the inversion model is linear. Subsequently,
a path-OED framework for nonlinear infinite-dimensional Bayesian inverse
problems is an open direction to research. In this case, the solution operator
of the inverse problem is nonlinear. Known methods for addressing this
nonlinearity include linearizing the solution operator, which would result in
locally optimal solutions to the path-OED problem, constructing a Laplace
approximation to the posterior, or relying on sampling based methods.

Other avenues to investigate include developments that make path-OED more practical
to deploy in science and engineering applications. Our formulations are for a single sensor apparatus. Thus, we can
consider adapting path-OED for an arbitrary number of mobile sensors. Our
theoretical framework indicates how to carry out the infinite-dimensional
formulation of this extension. However, considerable work is required to
develop computational methods that enable solving path-OED problems
with multiple sensors. We can also investigate how to incorporate constraints
like battery life or obstacles into the path-OED problem. Other constraints
arise depending on physical considerations related to the sensor apparatus and
application. 

Lastly, we remark on online extensions of path-OED or adaptations that
incorporate sensor dynamics into the problem. Our approach is not online,
meaning that the path-OED problem is solved before any data acquisition has
occurred. In contrast, we can consider an online approach---utilizing ideas from
sequential OED and data-assimilation. Advantages of such a formulation include
the ability to incrementally update our prior with encountered data or the
ability to respond to environmental hazards in real-time. Related to this
concept, we can also adapt our path-OED approach to applications where dynamics
governing the behavior of the sensor must be considered. For example,
trajectories of the sensor can be considered as solutions to a dynamical system 
where the experimental design is a control input.  This
amounts to extending the formulations such as the ones in~\cite{Ucinski052008}
to the present infinite-dimensional Bayesian formulations.  This can be
accomplished by noting that we can parameterize the control instead of the
sensor trajectory directly.

\section*{Acknowledgments}
This article has been authored by employees of National Technology \&
Engineering Solutions of Sandia, LLC under Contract No.~DE-NA0003525 with the
U.S.~Department of Energy (DOE). The employees own all right, title and interest
in and to the article and are solely responsible for its contents.
%
%% UNCOMMENT FOR arXiv
%
SAND2026-16116O.

%
%% Comment out the following line for arXiv.
%
%The United States Government retains and the publisher, by accepting the article
%for publication, acknowledges that the United States Government retains a
%non-exclusive, paid-up, irrevocable, world-wide license to publish or reproduce
%the published form of this article or allow others to do so, for United States
%Government purposes. The DOE will provide public access to these results of
%federally sponsored research in accordance with the DOE Public Access Plan
%\url{https://www.energy.gov/downloads/doe-public-access-plan}.  This paper
%describes objective technical results and analysis. Any subjective views or
%opinions that might be expressed in the paper do not necessarily represent the
%views of the U.S.~Department of Energy or the United States Government.

This material is also based upon work supported by the U.S. Department of Energy,
Office of Science, Office of Advanced Scientific Computing Research Field Work
Proposal Number 23-02526.  

The work of Ahmed Attia was supported by the U.S. Department of Energy, Office of Science, Office of Advanced Scientific Computing Research, Scientific Discovery through Advanced Computing (SciDAC) Program through the FASTMath Institute under contract number DE-AC02-06CH11357 at Argonne National Laboratory.

\bibliographystyle{plain} % sn-mathphys-num}
\bibliography{refs}

\begin{appendices}

\section{Proof of Theorem~\ref{thm:obs}}
\label{appdx:path}

The mollifiers defined in~\eqref{eq:mollifiers} are functions of time, space, and the path parameter $\vxi$. We tend to suppress these dependencies for convenience, simply writing $\tau_{k}$ and $\dr$.

\begin{proof}

In the case that $u$ is continuous, the integral $\llangle \tau_{k} \dr, u \rrangle$ approximates $u(\vr_{k}, \tau_{k})$. Of course, this interpretation is lost when $u$ is only $L^{2}$, as $u$ may not be pointwise defined. To show Item~1 of Theorem~\ref{thm:obs}, we refer the reader to Theorem 4.22 in~\cite{brezis2011functional}. This result states that we can approximate members of $L^{2}$ arbitrarily close (in the $L^{2}$ sense) via a sequence of convolutions with mollifier functions. When $u$ is continuous, this becomes pointwise convergence.

Now, fix $\vxi \in \R^{N}$ and consider the operator $\mc{B}(\vxi)$. Linearity of $\mc{B}$ is a direct consequence of the $\llangle \cdot, \cdot \rrangle$ inner-product, defined in~\eqref{eq:Lot_inner_prod}. As for boundedness, the Cauchy-Schwarz inequality provides that
$$
\|\mc{B}u\|_{2}^{2} = \sum_{k=1}^{\ny} \llangle \tau_k\dr, u \rrangle^2 \leq \|u\|_{\U}^{2} \sum_{k=1}^{\ny} \|\tau_{k}\dr\|_{\U}^{2}.
$$
The mollifiers are strictly positive and bounded above. Furthermore, $\Om \times T \subset \R^{2} \times \R$ is a bounded set. Hence, $\|\tau_{k}\dr\|_{\U}^{2} < \infty$, for any $k \in \{1, 2, \dots, \ny\}$, implying that $\mc{B}$ is a bounded in $\U$.

Now we show that $\mc{B}(\cdot)$ is continuous in the path variable. For the following argument, note that $\|\tau_{k}\|_{L^{\infty}(T)}$ is finite and constant in $k$. Subsequently, define $\tau_{\infty} \defeq \|\tau_{k}\|_{L^{\infty}(T)}$ and fix $\vxi \in \R^N$, Then, for a sequence $\{\vxi_n\}_{n=1}^{\infty}$ such that $\underset{n \to \infty}{\text{lim}} \vxi_{n} = \vxi$ and any $u \in \U$, the Cauchy-Schwarz inequality provides that
\begin{align*}
\|\big(\mc{B}(\vxi_{n}) - \mc{B}(\vxi)\big)u\|_2^2 &= \sum_{k=1}^{\ny} \bllangle \tau_{k} \big(\dr(\vxi_{n}) - \dr(\vxi)\big), u \brrangle^{2}\\
&\leq \|u\|_{\U}^{2}\sum_{k=1}^{\ny} \|\tau_{k}\big(\dr(\vxi_{n}) - \dr(\vxi)\big)\|_{\U}^{2}\\
&\leq \ny \tau_{\infty}^{2} \|u\|_{\U}^{2} \|\dr(\vxi_{n}) - \dr(\vxi)\|_{\U}^{2}.
\end{align*}
In the case that $u \neq 0$,
\begin{equation}
\label{eq:obs_cont_bound}
\frac{\|\big(\mc{B}(\vxi_{n}) - \mc{B}(\vxi\big)u\|_{2}}{\|u\|_{\U}} \leq \ny^{1/2}\tau_{\infty} \|\dr(\vxi_{n}) - \dr(\vxi)\|_{\U}.
\end{equation}
We rewrite this by taking the supremum over the left side, obtaining the operator norm of $\mc{B}(\vxi_{n}) - \mc{B}(\vxi)$. Specifically,
\begin{equation}
\label{eq:obs_cont}
\opnorm{\mc{B}(\vxi_{n}) - \mc{B}(\vxi)} \leq \ny^{1/2}\tau_{\infty} \|\dr(\vxi_{n}) - \dr(\vxi)\|_{\U}.
\end{equation}

To conclude the argument we show that the right side of~\eqref{eq:obs_cont} vanishes. The mollifier functions $\dr(\vx, t; \vxi)$ are bounded above for any $\vxi$. Thus, there must be some $K > 0$ such that
$$
|\dr(\vxi_{n}) - \dr(\vxi)| < K, \quad \text{for all} n \in \N.
$$
For any $(\vx, t) \in \Om \times T$, the term $\dr(\vx, t; \vxi)$ is a composition of continuous mappings in $\vxi$. Thus,
$$
\underset{n \to \infty}{\text{lim}} |\dr(\vxi_{n}) - \dr(\vxi)|^{2} = 0.
$$
Subsequently, the Dominated Convergence Theorem provides that
$$
\underset{n \to \infty}{\text{lim}} \|\dr(\vxi_{n}) - \dr(\vxi)\|_{\U}^{2} = \int_{T}\int_{\Om} \underset{n \to \infty}{\text{lim}} |\dr(\vxi_{n}) - \dr(\vxi)|^{2} \, d\vx dt = 0.
$$
It follows that the right side of~\eqref{eq:obs_cont} vanishes and $\mc{B}(\cdot)$ must be continuous since
$$
\underset{n \to \infty}{\text{lim}} \opnorm{\mc{B}(\vxi_{n}) - \mc{B}(\vxi)} = 0, \quad \text{whenever} \quad \underset{n \to \infty}{\text{lim}} \vxi_{n} = \vxi.
$$

Lastly, we derive the adjoint of the observation operator. We note that for any $u \in \U$ and $\vec y \in \R^{\ny}$, the adjoint operator $\mc{B}^{*}$ satisfies $\langle \mc{B} u, \vec y \rangle_2 = \llangle u, \mc{B}^*\vec y \rrangle$. Performing a straightforward calculation, we find that
$$
\langle \mc{B} u, \vec y \rangle_{2} = \sum_{k=1}^{\ny}(\mc{B}u)_{k}y_{k} =\sum_{k=1}^{\ny}\llangle u, \tau_{k} \dr \rrangle y_{k} = \bllangle u, \dr \sum_{k=1}^{\ny}\tau_{k}y_{k}\brrangle.
$$
The resulting expression reveals the action of $\mc{B}^*$ on $\vec y$.

\end{proof}

\section{Proof of Theorem~\ref{thm:Cpo_deriv}}
\label{appdx:Cpo_deriv}

\begin{proof}

We obtain the desired result by proving that
\begin{equation}
\label{eq:Cpo_diff_limit_result}
\underset{h \to 0}{\text{lim}} \frac{1}{h} \left( \mc{C}(\xi + h) - \mc{C}(\xi)\right)= - \underset{h \to 0}{\text{lim}} \mc{C}(\xi) \bdot{\mc{H}}_{0}(\xi) \mc{C}(\xi).
\end{equation}
Our argument uses that $\mc{C}:\R \to \Lsym^{++}(\M)$ is continuous. Hence, we prove this fact first. Recall that $\mc{H} = \mc{H}_{0} + \mc{C}_{0}^{-1}$ and $\mc{I} = \mc{C}\mc{H}$. Thus, $\mc{H}(\xi + h) - \mc{H}(\xi) = \mc{H}_{0}(\xi + h) - \mc{H}_{0}(\xi)$ and, by submultiplicativity of the operator norm,
\begin{equation}
\label{eq:Cpo_conv_ineq}
\begin{alignedat}{1}
\opnorm{\mc{C}(\xi + h) - \mc{C}(\xi)} &= \opnorm{\mc{C}(\xi + h)(\mc{H}(\xi) - \mc{H}(\xi + h))\mc{C}(\xi)}\\
&\leq  \opnorm{\mc{C}(\xi)} \, \opnorm{\mc{C}(\xi + h)} \, \opnorm{\mc{H}_{0}(\xi) - \mc{H}_{0}(\xi + h)}.
\end{alignedat}
\end{equation}
Since $\mc{H}_{0}$ is Frech\'{e}t differentiable, it is continuous and we have that $\underset{h \to 0}{\text{lim}} \frac{1}{h} \opnorm{\mc{H}_{0}(\xi) - \mc{H}_{0}(\xi + h)} = 0$. Thus, the right side of~\eqref{eq:Cpo_conv_ineq} vanishes whenever $\mc{C}(\xi + h)$ is bounded. We prove this fact next.

Recall that $\mc{C}_{0} \in \Lsym^{++}(\M)$ and is not necessarily surjective. Subsequently, $\mc{C}_{0}^{1/2}$ exists and is a member of $\Lsym^{++}(\M)$. The inverse of this operator, denoted by $\mc{C}_{0}^{-1/2}$, is possibly unbounded, self-adjoint, strictly positive, and defined on the range of $\mc{C}_{0}^{1/2}$. Hence, we form the operator $\tilde{\mc{H}}_{0} = \mc{C}_{0}^{1/2}\mc{H}_{0}\mc{C}_{0}^{1/2}$ and perform the factorization
\begin{equation}
\label{eq:factored_H}
\mc{C}(\xi + h) = \mc{C}_{0}^{1/2}(\tilde{\mc{H}}_{0}(\xi + h) + \mc{I})^{-1}\mc{C}_{0}^{1/2}.
\end{equation}
Now, using the self-adjointedness of $\mc{C}_{0}^{1/2}$, we have that
\begin{equation}
\label{eq:CH_ineq}
\opnorm{\mc{C}(\xi + h)} = \opnorm{\mc{C}_{0}^{1/2}(\tilde{\mc{H}}_{0}(\xi + h) + \mc{I})^{-1}\mc{C}_{0}^{1/2}} \leq \opnorm{\mc{C}_{0}} \, \opnorm{(\tilde{\mc{H}}_{0}(\xi + h) + \mc{I})^{-1}}.
\end{equation}
Since $\mc{C}_{0}^{1/2} \in \Lsym^{++}(\M)$ and $\mc{H}_{0} (\xi + h) \in \Lsym^{+}(\M)$, we have that $\tilde{\mc{H}}_{0}(\xi + h) \in \Lsym^{+}(\M)$, for any $\xi,h \in \R$. Thus, the spectrum of $\tilde{\mc{H}}_{0}(\xi + h)$ is a subset of $\left[0, \opnorm{\tilde{\mc{H}}_{0}(\xi + h)}\right]$. Subsequently, shifting $\tilde{\mc{H}}_{0}(\xi + h)$ by $\mc{I}$ then inverting yields
$$
\opnorm{(\tilde{\mc{H}}_{0}(\xi + h) + \mc{I})^{-1}} \leq 1.
$$
Substituting this bound into~\eqref{eq:CH_ineq}, we have that $\opnorm{\mc{C}(\xi + h)} \leq \opnorm{\mc{C}_{0}}$ and~\eqref{eq:Cpo_conv_ineq} implies
$$
\opnorm{\mc{C}(\xi + h) - \mc{C}(\xi)}
\leq  \opnorm{\mc{C}(\xi)} \, \opnorm{\mc{C}_{0}} \, \opnorm{\mc{H}_{0}(\xi) - \mc{H}_{0}(\xi + h)}.
$$
It follows that $\mc{C}$ is continuous since $\underset{h \to 0}{\text{lim}} \opnorm{\mc{C}(\xi + h) - \mc{C}(\xi)} = 0$.

Now we derive the formula for the derivative of $\mc{C}$. Since $\mc{H}_{0}$ is differentiable, $\bdot{\mc{H}}_{0} = \bdot{\mc{H}} = \underset{h \to 0}{\text{lim}} \frac{1}{h} (\mc{H}(\xi + h) - \mc{H}(\xi)).$
Using this fact, the identity $\mc{I} \equiv \mc{C}\mc{H}$, and continuity of $\mc{C}$, we have that
\begin{align*}
0 &= \underset{h \to 0}{\text{lim}} \frac{1}{h} \left(\mc{C}(\xi + h)\mc{H}(\xi + h) - \mc{C}(\xi)\mc{H}(\xi)\right)\\
&= \underset{h \to 0}{\text{lim}} \frac{1}{h} \left(\mc{C}(\xi + h)\mc{H}(\xi + h) - \mc{C}(\xi + h)\mc{H}(\xi)\right) + \underset{h \to 0}{\text{lim}} \frac{1}{h} \left(\mc{C}(\xi + h)\mc{H}(\xi) - \mc{C}(\xi)\mc{H}(\xi)\right)\\
&= \underset{h \to 0}{\text{lim}} \frac{1}{h} \mc{C}(\xi + h) \left( \mc{H}(\xi + h) - \mc{H}(\xi) \right) + \underset{h \to 0}{\text{lim}} \frac{1}{h} ( \mc{C}(\xi + h) - \mc{C}(\xi) ) \mc{H}(\xi)\\
&= \mc{C}(\xi) \bdot{\mc{H}}_{0}(\xi) + \underset{h \to 0}{\text{lim}} \frac{1}{h} \left( \mc{C}(\xi + h) - \mc{C}(\xi) \right) \mc{H}(\xi)
\end{align*}
By subtracting the first term in the resulting expression and multiplying by $\mc{H}^{-1}(\xi)$, we obtain
\begin{equation}
\label{eq:Cpo_diff_quo}
\underset{h \to 0}{\text{lim}} \frac{1}{h} \left( \mc{C}(\xi + h) - \mc{C}(\xi)\right)= - \mc{C}(\xi) \bdot{\mc{H}}_{0}(\xi) \mc{C}(\xi).
\end{equation}

\end{proof}

\section{Gradient of the filtered c-optimality criterion}
\label{appdx:c_optimal_void}

In what follows, we derive a formula for the derivative of the c-optimality criterion for an inverse problem modified to incorporate an obscured region. We recall the \emph{filtered Hessian} $\Hvoid$ as presented in~\eqref{eq:hess_void} and $\Psi$ in~\eqref{eq:c_optimal}. Since $\Hvoid^{-1} \equiv \Cpo$, the filtered c-optimality criterion can be written as
\begin{equation}
\label{eq:psi_void}
\Psi(\vxi) = \langle \Hvoid^{-1}(\vxi) c, m \rangle = \langle \Cpo(\vxi) c, m \rangle.
\end{equation}
We then derive the derivatives of~\eqref{eq:psi_void}. To simplify subsequent calculations, we suppress dependency on $\vxi$.

Let $j \in \{1, \dots, N\}$ and recall that Theorem~\ref{thm:Cpo_partial} provides that
\begin{equation}
\label{eq:c_opt_grad_void_small}
\pj\Psi = - \big\langle \big(\pj \Hvoid\big) \Cpo c, \Cpo c \big\rangle.
\end{equation}
Thus, we must compute $\pj \Hvoid$ and manipulate the resulting expressions. By the definition of $\Hvoid$,
\begin{equation}
\label{eq:H_void_partial}
\begin{alignedat}{1}
\pj \Hvoid &= \sig^{-2} \mc{S}^* \pj\big(\mc{B}^*(\mat I - \mat P)\mc{B}\big)\mc{S}\\
&= \sig^{-2}\mc{S}^*\pj(\mc{B}^*\mc{B})\mc{S} - \sig^{-2}\mc{S}^*\pj(\mc{B}^*\Void\mc{B})\mc{S}.
\end{alignedat}
\end{equation}
We computed the first term of~\eqref{eq:H_void_partial} in the proof of Theorem~\ref{thm:Cpo_partial}. Hence, we seek a formula for $\pj(\mc{B}\mat P \mc{B})$.

Let $u \in \U$. Then, the definition of $\mc{B}$ and its adjoint $\mc{B}^{*}$ imply that
$$
(\mc{B}^*\Void\mc{B})u = \dr \sum_{k=1}^d \tau_k \void_k \llangle \tau_k \dr, u \rrangle.
$$
The resulting expression reveals a tensor representation for $\mc{B}^* \mat P \mc{B}$. Specifically,
\begin{equation}
\label{eq:void_tensor}
\mc{B}^* \mat P \mc{B} = \sum_{k=1}^{\ny} \tau_k \dr \void_k \otimes \tau_k \dr.
\end{equation}
Applying Lemma~\ref{lem:tens_prod_deriv} to~\eqref{eq:void_tensor}, we have that
\begin{equation}
\label{eq:obs_void_obs}
\pj (\mc{B}^* \mat P \mc{B}) = \sum_{k=1}^{\ny} \tau_k \dr \void_k \otimes \tau_k \pj \dr + \tau_k \pj (\dr \void_k) \otimes \tau_k \dr.
\end{equation}
Both $\dr$ and $\void_k$ are continuously differentiable in the path parameter $\vxi$. So,
\begin{equation}
\label{eq:dr_p_deriv}
\pj(\dr \void_k) = \dr (\grad p_k \cdot \pj \vec r_k) + \void_k \pj \dr.
\end{equation}
Combining~\eqref{eq:obs_void_obs}, ~\eqref{eq:dr_p_deriv}, and our computation for $\pj(\mc{B}^*\mc{B})$ in~\eqref{eq:BB_partial}, we write~\eqref{eq:H_void_partial} as\footnote{
Here we use that if $u, v, u' \in U$, then for all $v' \in U$, $((u'u) \otimes v)v' = u' (u \llangle v, v' \rrangle) = u'(\big(u \otimes v)v'\big).$ Hence, $(u'u) \otimes v = u'(u \otimes v).$
}
\begin{equation}
\label{eq:Hvoid_deriv_full}
\begin{alignedat}{1}
\pj \Hvoid &= \sig^{-2}\mc{S}^*\Big(\sum_{k=1}^{\ny} (1 - \void_{k})\big(\tau_k \dr \void_k \otimes \tau_k \pj \dr + \tau_k \pj (\dr \void_k) \otimes \tau_k \dr\Big)\mc{S}\\
&- \sig^{-2}\mc{S}^{*}\Big(\sum_{k=1}^{\ny}(\grad\void_k \cdot \pj \vec r_k) \tau_k \dr \otimes \tau_k \dr\Big)\mc{S}.
\end{alignedat}
\end{equation}

In the last step we substitute~\eqref{eq:Hvoid_deriv_full} back into~\eqref{eq:c_opt_grad_void_small} then perform a series of manipulations; since $\grad \void_k \cdot \pj \vr_k$ and $1 - \void_k$ are scalars, we can factor them out of $\llangle \cdot, \cdot \rrangle$. Subsequently, we find that
\begin{align*}
\pj \Psi &= -2\sig^{-2} \sum_{k=1}^{\ny} (1 - \void_k) \llangle \tau_k \dr, \mc{S}\Cpo c \rrangle \llangle \tau_k \pj \dr, \mc{S} \Cpo c \rrangle\\
&-\sig^{-2}\sum_{k=1}^{\ny}(\grad \void_k \cdot \pj \vr_k) \llangle \tau_k \dr, \mc{S}\Cpo c \rrangle^2.
\end{align*}
A more compact expression can be acquired recognizing the identities in~\eqref{eq:B_components}:
\begin{equation}
\label{eq:grad_psi_c_void_full}
\pj \Psi = -2\sig^{-2} \big\langle \mc{B} \mc{S}\Cpo c, (\mat I - \mat P) (\pj \mc{B}) \mc{S}\Cpo c \big\rangle_{2} - \sig^{-2} \big\langle \mc{B} \mc{S}\Cpo c,  (\pj \mat P) \mc{B} \mc{S}\Cpo c \rangle_{2}.
\end{equation}

\section{Proof of Theorem~\ref{thm:disk}}
\label{appdx:disk}

\begin{proof}

First we express the Fourier path $\rf$ as an affine transformation on $\xif \in \R^{4 \Nf}$. To do this, we first define
\begin{equation}
\label{eq:sin_cos}
C_{j} \defeq\cos(\om_{j}t) \quad \text{and} \quad S_{j} \defeq\sin(\om_{j}t), \quad \text{with} \quad \om_{j} = \frac{2\pi j}{|\Ty|}.
\end{equation}
We then build an operator  $\mat \Tf: \Ty \to \R^{2 \times 4\Nf}$ with action on $\xif$ defined by
\begin{equation}
\mat \Tf(t) \xif \defeq
\sum_{j = 1}^{\Nf}
\begin{bmatrix}
a_{j} C_{j}(t) + b_{j}S_{j}(t)\\
c_{j} C_{j}(t) + d_{j} S_{j}(t)
\end{bmatrix}.
\end{equation}
It is easy to show that for any $t \in \Ty$, the transformation $\Tf$ is linear in $\xif$. Hence, $\rf$ in~\eqref{eq:fourier_path} becomes
\begin{equation}
\rf (t; \vxi) = \bar{\vx} +\mat \Tf(t)\xif,
\end{equation}
implying that $\rf$ is an affine transformation. Subsequently, Remark~\ref{rmk:general_disk} provides that for any $\Rd > 0$,
$$
\|\rf - \bar{\vx}\|_2 \leq \Rd, 
\quad \text{whenever} \quad 
\|\vxi\|_{2} \leq \frac{\Rd}{\underset{t \in T_{y}}{\sup} \ \|\mat \Tf(t)\|_{2}}.
$$
We are then left to prove that $\|\mat \Tf(t)\|_2 = \sqrt{\Nf}$, for all $t \in \Ty$, yielding the desired result.

In what follows we fix $t \in \Ty$ and suppress this variable in our calculations. For $\vx \in \R^{2}$, observe that
\begin{align*}
\vx^{\top}\Tf\xif
  &= x_{1}\sum_{j=1}^{\Nf}\!\big(a_{j}C_{j}+b_{j}S_{j}\big)
   + x_{2}\sum_{j=1}^{\Nf}\!\big(c_{j}C_{j}+d_{j}S_{j}\big) \\
  &= \xif^{\top}
     \begin{bmatrix}
       x_{1}C_{1} & x_{1}S_{1} & x_{2}C_{1} & x_{2}S_{1} &
       \dots &
       x_{1}C_{\Nf} & x_{1}S_{\Nf} & x_{2}C_{\Nf} & x_{2}S_{\Nf}
     \end{bmatrix}^{\top}.
\end{align*}
This reveals the action of $\Tf^{\top}$ on $\xif$. Specifically,
$$
\mat T^{\top}_{f}\vx =
     \begin{bmatrix}
       x_{1}C_{1} & x_{1}S_{1} & x_{2}C_{1} & x_{2}S_{1} &
       \dots &
       x_{1}C_{\Nf} & x_{1}S_{\Nf} & x_{2}C_{\Nf} & x_{2}S_{\Nf}
     \end{bmatrix}^{\top}.
$$
This provides us with the following formula for the action of $\Tf^{\top}\Tf$ on $\xif$:
\begin{align*}
(\Tf^{\top}\Tf) \xif &= \left[ \left( C_{1} \sum_{j=1}^{\Nf} a_{j} C_{j} + b_{j}S_{j} \right), \, 
\left( S_{1}\sum_{j=1}^{\Nf} a_{j} C_{j} + b_{j}S_{j} \right), \, 
\left( C_{1} \sum_{j=1}^{\Nf} c_{j} C_{j} + d_{j} S_{j} \right), \,
\left( S_{1} \sum_{j=1}^{\Nf} c_{j} C_{j} + d_{j} S_{j} \right), \, \dots, \, \right. \\
&\left. \quad \left( C_{\Nf} \sum_{j=1}^{\Nf} a_{j} C_{j} + b_{j}S_{j} \right), \, \left( S_{\Nf}\sum_{j=1}^{\Nf} a_{j} C_{j} + b_{j}S_{j} \right), \, \left( C_{\Nf} \sum_{j=1}^{\Nf} c_{j} C_{j} + d_{j} S_{j} \right), \, \left( S_{\Nf}\sum_{j=1}^{\Nf} c_{j} C_{j} + d_{j} S_{j} \right) \right]^{\top}.
\end{align*}
From this, we can obtain an explicit matrix representation of $\Tf^{\top}\Tf$. In particular,
$$
\Tf^{\top}\Tf = 
\begin{bmatrix}
\begin{array}{cc|cc|c|cc}
\mat Y_{1,1} & \mat 0 & \mat Y_{1,2} & \mat 0 & \cdots & \mat Y_{1, \Nf}& \mat 0 \\ 
\mat 0 & \mat Y_{1,1} & \mat 0 & \mat Y_{1,2} & \cdots & \mat 0 & \mat Y_{1, \Nf}\\
\hline
\mat Y_{2, 1} & \mat 0 & \mat Y_{2,2} & \mat 0 & \cdots & \mat Y_{2, \Nf} & \mat 0\\
\mat 0 & \mat Y_{2, 1} & \mat 0 &  \mat Y_{2,2} & \cdots & \mat 0 & \mat Y_{2, \Nf}\\
\hline
\vdots & \vdots & \vdots & \vdots & \ddots & \vdots & \vdots\\
\hline
\mat Y_{\Nf, 1} & \mat 0 & \mat Y_{\Nf, 2} & \mat 0 & \cdots & \mat Y_{\Nf, \Nf} & \mat 0\\
\mat 0 & \mat Y_{\Nf, 1} & \mat 0 & \mat Y_{\Nf, 2} & \cdots & \mat 0 & \mat Y_{\Nf, \Nf}
\end{array}
\end{bmatrix},
\, \text{where} \,
\mat Y_{i,j}(t) \defeq
\begin{bmatrix}
C_{i}(t)C_{j}(t) & C_{i}(t)S_{j}(t)\\
C_{j}(t)S_{i}(t) & S_{i}(t)S_{j}(t)
\end{bmatrix}.
$$
This representation reveals that $\Tf^{\top}\Tf$ is the sum of two rank-1 operators. To see this, define
\begin{align*}
\vec v_{1} &\defeq \frac{1}{\sqrt{\Nf}}[C_{1} \ S_{1} \ 0 \ 0 \ C_{2} \ S_{2} \ \cdots \ C_{\Nf} \ S_{\Nf} \ 0 \ 0]^{\top},\\
\vec v_{2} &\defeq \frac{1}{\sqrt{\Nf}}[0 \ 0 \ C_{1} \ S_{1} \ 0 \ 0 \ \cdots \ 0 \ 0 \ C_{\Nf} \ S_{\Nf}]^{\top}.
\end{align*}
It is clear that $\{\vec v_{1}, \vec v_{2}\}$ is an orthonormal system in $\R^{4\Nf}$. Furthermore, it can be shown that
\begin{equation}
\label{eq:Tf_spectral}
\Tf^{\top}\Tf = \Nf(\vec v_{1}\vec v_{1}^{\top} + \vec v_{2}\vec v_{2}^{\top}).
\end{equation}
We recognize~\eqref{eq:Tf_spectral} as the spectral decomposition of $\Tf^{\top}\Tf$, implying that $\sqrt{\Nf}$ is the only eigenvalue of this matrix. The proof is finished since it follows that
$$
\|\mat T_{f}\|_{2} = \sqrt{\Nf}.
$$

\end{proof}

\section{A discretized solution operator and its adjoint}
\label{appdx:solution}

Here we derive the action discrete solution operator $\mat S: \RNM \to \RNMT$ from the model~\eqref{eq:conv_diff}. Subsequently, we develop an algorithm for computing the action of the adjoint $\mat S^{*}: \RNMT \to \RNM$.

Consider $\{t_\ell\}_{\ell=0}^{\nt}$, the discretization of $T$, and define $u_\ell \defeq u(\vx, t_\ell)$. Using the implicit Euler scheme,
$$
u_{t}(\vec x, t_{\ell+1}) \approx  \frac{u^{\ell+1} - u^{\ell}}{\dt}, \quad \ell \in \{0, \dots, \nt-1\}.
$$
Substituting this into~\eqref{eq:conv_diff},
\begin{equation}
\label{eq:conv_diff_temporal}
u^{\ell+1} - \Delta t\,\alpha \Delta u^{\ell+1} + \Delta t\, \vec v \cdot \grad u^{\ell+1} = u^{\ell} + a_{\ell+1}\Delta t\, m, \quad a_{\ell+1} \defeq a(t_{\ell+1}).
\end{equation}
Hence, obtaining a solution to~\eqref{eq:conv_diff} amounts to solving an elliptic PDE at each time-step. Now, define the space $
\ms{U} \defeq \{u \in H^{1}(\Om): \ u\big\vert_{E_{0}} = 0\}.$
Then, the weak form of~\eqref{eq:conv_diff_temporal} is: find $u^{\ell+1} \in \ms{U}$ such that
\begin{equation}
\label{eq:conv_diff_weak_temporal}
\langle u^{\ell+1}, \zeta\rangle + \Delta t\, \alpha \langle\grad u^{\ell+1}, \grad\zeta\rangle + \Delta t\, \langle\vec v \cdot u^{\ell+1}, \zeta\rangle = \langle u^{\ell}, \zeta\rangle + a_{\ell+1}\Delta t\, \langle m, \zeta\rangle, \quad \text{for all } \zeta \in \ms{U}.
\end{equation}
We discretize~\eqref{eq:conv_diff_weak_temporal} in space with the finite element method; see Subsection~\ref{sec:hilbert}. Define $\mat A \in \R^{\nx \times \nx}$ by
$$
A_{i,j} \defeq \langle \phi_{i}, \phi_{j} \rangle  + \alp \dt\, \langle \grad \phi_{i}, \grad \phi_{j} \rangle + \dt\, \langle \vec v \cdot \grad \phi_{i}, \phi_{j} \rangle, \quad i,j \in \{1, 2, \dots, n_{x}\}.
$$
Now, let $\vec u_{\ell} \in \RNM$ be the finite element coefficients of the solution at the $\ell$-th time and form the snapshot matrix 
$\mat U = [\vec u_{1} \, \dots \, \vec u_{\nt}]$. Then, $\mat S \vec m = \mat U$ where $\vec u_{\ell}$ are found by solving
\begin{equation}
\label{eq:solution_system}
\mat A \vec u_{\ell+1} = \mat M \vec u_{\ell} + a_{\ell+1}\dt\, \mat M \vec m, \quad \ell \in \{0, \dots \nt-1\}.
\end{equation}

\boldheading{The adjoint} To derive the adjoint, let $\vec a \defeq [a_{1} \, \dots \, a_{\nt}]^{\top}$ and define
\begin{equation}
\label{eq:power_poly}
p_{\ell}(x;\vec a) \defeq \sum_{j=1}^{\ell} a_{j}x^{\ell - j + 1}.
\end{equation}
We then reveal a recursive relationship and write the result in terms of~\eqref{eq:power_poly}. Specifically, observe that
\begin{align*}
\vec u_{1} &= \mat A^{-1} \mat M \vec u_{0} + a_{1}\Delta t\, \mat A^{-1} \mat M \vec m\\
\vec u_{2} &= \mat A^{-1} \mat M \vec u_{1} + a_{2}\Delta t\, \mat A^{-1} \mat M \vec m\\
&= (\mat A^{-1}\mat M)^{2}\vec u_{0} + a_{1}\Delta t\, (\mat A^{-1}\mat M)^{2} \vec m + a_{2}\Delta t\, \mat A^{-1}\mat M \vec m\\
\vdots\\
\vec u_{\ell} &= (\mat A^{-1}\mat M)^{\ell}\vec u_{0} + \Delta t\, p_{\ell}(\mat A^{-1}\mat M; \vec a)\vec m.
\end{align*}
Assuming that $\vec u_{0} = \vec 0$, the recursive relationship is defined by
\begin{equation}
\label{eq:recursive}
\vec u_{\ell} = \Delta t\, p_{\ell}(\mat A^{-1}\mat M; \vec a)\vec m, \quad \ell \in \{1, \dots, \nt\}. 
\end{equation}

In the next step we let $\mat V \in \RNMT$ and seek $\mat S^{*}$ such that $\llangle \mat S \vec m, \mat V \rrangle_{\mat M} = \langle \vec m, \mat S^{*} \mat V \rangle_{\mat M}$. Using the identity for $\llangle \cdot, \cdot \rrangle_{\mat M}$ presented in~\eqref{eq:MT_identities_one}, we have that
$$
\llangle\mat S \vec m, \mat V\rrangle_{\mat M} = \vec 1^{\top} (\mat U \odot \mat M \mat V)\vec w = \sum_{\ell=1}^{n} w_{\ell} \vec 1^{\top}(\vec u_{\ell} \odot \mat M \vec v_{\ell}).
$$
Note that for any $\vec m_{1}, \vec m_{2} \in \R^{\nx}$,
$$
\vec 1^{\top} (\vec m_{1} \odot \vec m_{2}) = \vec m_{1}^{\top} \vec m_{2} = \vec m_{2}^{\top} \vec m_{1}.
$$
Thus, using~\eqref{eq:recursive}, the identity above, then inserting $\mat M^{-1} \mat M = \mat I$ in the right place, we have
\begin{align*}
\sum_{\ell=1}^{n} w_{\ell} \vec 1^{\top}(\vec u_{\ell} \odot \mat M \vec v_{\ell}) &= \dt\, \sum_{\ell=1}^{n_{t}} w_{\ell} \vec 1^{\top} \big( p_{\ell}(\mat A^{-1}\mat M; \vec a) \vec m \odot \mat M \vec v_{\ell} \big)\\
&= \left(\dt\, \sum_{\ell=1}^{n_{t}} w_{\ell} \vec v_{\ell}^{\top}\mat M p_{\ell}(\mat A^{-1}\mat M; \vec a)\mat M^{-1}\right)\mat M\vec m\\
&= \langle \vec m, \mat S^{*} \mat V \rangle_{\mat M},
\end{align*}
where
\begin{equation}
\label{eq:s_adjoint_power}
\mat S^{*} \mat V \defeq \dt\, \left( \sum_{\ell=1}^{n_{t}} w_{\ell} \vec v_{\ell}^{\top}\mat M p_{\ell}(\mat A^{-1}\mat M; \vec a)\mat M^{-1} \right)^{\top}.
\end{equation}

To obtain a useful representation for the action of $\mat S^{*} \mat V$, we must perform some additional steps. Let $\ell \in \{1, \dots, \nt\}$ and recall the definition of $p_{\ell}$ given in~\eqref{eq:power_poly}. Then, 
\begin{align*}
\mat M^{-1} p_{\ell}(\mat A^{-1} \mat M; \vec a)^{\top} \mat M &= \mat M^{-1} \left( \sum_{j=1}^{\ell} a_{\ell} (\mat A^{-1} \mat M)^{\ell - j + 1} \right)^{\top} \mat M\\
&= \sum_{j=1}^{\ell} a_{\ell} \mat M^{-1} (\mat M \mat A^{-\top})^{\ell - j + 1} \mat M\\
&= \sum_{j=1}^{\ell} a_{\ell} (\mat A^{-\top} \mat M)^{\ell - j + 1}.
\end{align*}
Thus, by the definition of~\eqref{eq:power_poly}, we have that $\mat M^{-1} p_{\ell}(\mat A^{-1} \mat M; \vec a)^{\top} \mat M = p_{\ell}(\mat A^{-\top} \mat M; \vec a)$. Returning to~\eqref{eq:s_adjoint_power},
\begin{align*}
\mat S^{*} \mat V &= \left(\dt\, \sum_{\ell=1}^{n_{t}} w_{\ell} \vec v_{\ell}^{\top}\mat M p_{\ell}(\mat A^{-1}\mat M; \vec a)\mat M^{-1}\right)^{\top}\\
&= \dt\, \sum_{\ell=1}^{n_{t}} w_{\ell} \big(\mat M^{-1} p_{\ell}(\mat A^{-1}\mat M; \vec a)^{\top}\mat M\big) \vec v_{\ell}\\
&= \dt\, \sum_{\ell=1}^{n_{t}} w_{\ell} p_{\ell}(\mat A^{-\top} \mat M; \vec a) \vec v_{\ell}.
\end{align*}
By examining each term in the resulting sum and expanding $p_\ell(\mat A^{-\top} \mat M; \vec a)$, we find the recursive relation
\begin{align*}
\dt^{-1}\, \mat S^{*}\mat V &= a_{1}w_{1}(\mat A^{-\top}\mat M)\vec v_{1}\\
&+ a_{1}w_{2}(\mat A^{-\top}\mat M)^{2}\vec v_{2}+ a_{2}w_{2}(\mat A^{-\top}\mat M)\vec v_{2}\\
&\vdots\\
&+ a_{1}w_{\nt}(\mat A^{-\top} \mat M)^{\nt} \vec v_{\nt} + a_{2}w_{\nt}(\mat A^{-\top} \mat M)^{\nt - 1} \vec v_{\nt} + \dots + a_{\nt}w_{\nt}(\mat A^{\top} \mat M) \vec v_{\nt}.
\end{align*}
Factoring the powers of $\mat A^{-\top} \mat M$, this becomes
\begin{align*}
\dt^{-1}\, \mat S^{*}\mat V &= (\mat A^{-\top} \mat M)(a_{1}w_{1}\vec v_{1} + a_{2}w_{2}\vec v_{2} + \dots + a_{\nt}w_{\nt}\vec v_{\nt})\\
&+ (\mat A^{-\top}\mat M)^{2}(a_{1}w_{2}\vec v_{2} + a_{2}w_{3}\vec v_{3} + \dots + a_{\nt-1}w_{\nt}\vec v_{\nt})\\
& \vdots\\
&+ (\mat A^{-\top} \mat M)^{\nt - 1}(a_{1}w_{\nt -1}\vec v_{\nt -1} + a_{2} w_{\nt} \vec v_{\nt})\\
&+ (\mat A^{-\top}\mat M)^{\nt}(a_{1}w_{\nt}\vec v_{\nt}).
\end{align*}
Thus,
\begin{equation}
\label{eq:final_S_adjoint}
\mat S^{*} \mat v = \dt^{-1}\, \sum_{\ell=1}^{\nt} (\mat A^{-\top} \mat M)^{\ell} \sum_{j=1}^{\nt - \ell + 1} a_{j} w_{j + \ell - 1} \vec v_{j + \ell -1}.
\end{equation}
From~\eqref{eq:final_S_adjoint} we derive Algorithm~\ref{alg:solution_adjoint}; a method for applying $\mat S^*$ to $\mat V \in \RNMT$ with only $\nt$ solves.
\begin{algorithm}[H]
  \caption{Algorithm for computing $\mat S^{*} \mat V$}
  \begin{algorithmic}[1]
    \State \textbf{Input:} $\mat V \in \RNMT$, $\vec a \in \R^{\nt}$, $\vec w \in \R^{\nt}$
    \State \textbf{Output:} $\vec q = \mat S^{*} \mat V \in \RNM$
    \State Set $\vec q = \vec 0 \in \R^{\nx}$
    \For{$\ell = 1$ to $n_{t}$}
      \State Compute
      $$
      \vec z = \vec q + \sum_{j=1}^{\ell} a_{j} w_{\nt - \ell + j}\vec v_{\nt - \ell + j}
      $$
      \State Solve $\mat A^{\top} \vec q = \mat M \vec z$ for $\vec q$
    \EndFor
    \State \Return $\dt\, \vec q$
  \end{algorithmic}
  \label{alg:solution_adjoint}
\end{algorithm}

\end{appendices}

\null \vfill
\begin{flushright}
    \scriptsize \framebox{\parbox{5.0in}{
    The submitted manuscript has been created by UChicago Argonne, LLC,
    Operator of Argonne National Laboratory (``Argonne"). Argonne, a
    U.S. Department of Energy Office of Science laboratory, is operated
    under Contract No. DE-AC02-06CH11357. The U.S. Government retains for
    itself, and others acting on its behalf, a paid-up nonexclusive,
    irrevocable worldwide license in said article to reproduce, prepare
    derivative works, distribute copies to the public, and perform
    publicly and display publicly, by or on behalf of the Government.
    The Department of
    Energy will provide public access to these results of federally sponsored research in accordance
    with the DOE Public Access Plan. http://energy.gov/downloads/doe-public-access-plan. }}
    \normalsize
  \end{flushright}

\end{document}